\documentclass[12pt]{article}

\usepackage[margin=1in]{geometry}

\usepackage{latexsym, amsmath, amsfonts, amsthm, amssymb}
\usepackage{times}
\usepackage{graphicx, tocloft}
\usepackage{mathrsfs} 
\usepackage{xcolor}
\usepackage{hyperref}
\usepackage{enumerate}
\usepackage{thm-restate}
\usepackage{esint}
\usepackage{faktor}

\numberwithin{equation}{section}

\def \Vol {\mathrm{Vol}}

\def \tr {\text{\rm Tr}}
\def \Ric {\text{\rm Ric}}
\def \Hess {\mathrm{Hess}}
\def \Len {\mathrm{Len}}

\def \diam {{\rm diam}}
\def \div {{ \rm div}}
\def \supp {{ \rm supp}\,}
\def \Tr {{ \rm Tr }}

\def \Graph {\mathrm{Graph}}
\def \Exp {\mathrm{Exp}}
\renewcommand{\Re}{\mathrm{Re}}

\def \cc {\mathrm c}
\def \CC {\mathrm C}

\def \qqquad {\qquad \qquad }
\def \RR {\mathbb R}

\def \HH {\mathbf H}
\def \bC {\mathbb C}

\def \rY {\mathrm Y}
\def \rS {\mathrm S}
\def \rT {\mathrm T}
\def \rR {\mathrm R}
\def \rD {\mathrm D}
\def \rJ {\mathrm J}
\def \rI {\mathrm I}

\def \eps {\varepsilon}
\def \vphi {\varphi}

\def \bM {\mathbf{M}}

\def \bz {\mathbf 0}
\def \bL {\mathbf{L}}

\def \cS {\mathcal S}

\def \rF {\mathrm F}

\def \cP {\mathcal P}

\def \cH {\mathcal H}
\def \cL {\mathcal L}
\def \sA {\mathscr A}

\def \sN {\mathscr N}
\def \cL {\mathcal L}
\def \cE {\mathcal E}
\def \cV {\mathcal V}

\def \sL {\mathscr L}
\def \sD {\mathscr D}

\def \CD {\mathrm{CD}}

\newtheorem{theorem}{Theorem}[section]
\newtheorem{Theorem}{Theorem}
\newtheorem*{theorem*}{theorem}
\newtheorem{definition}[theorem]{Definition}

\newtheorem{lemma}[theorem]{Lemma}

\newtheorem{proposition}[theorem]{Proposition}
\newtheorem*{proposition*}{Proposition}
\newtheorem{corollary}[theorem]{Corollary}
\newtheorem{Corollary}{Corollary}
\newtheorem{example}[theorem]{Example}

\theoremstyle{definition}
\newtheorem{remark}[theorem]{Remark}
\newtheorem*{remark*}{Remark}

\DeclareMathAlphabet{\mathcal}{U}{eus}{m}{n}

\def\qed{\hfill $\vcenter{\hrule height .3mm
		\hbox {\vrule width .3mm height 2.1mm \kern 2mm \vrule width .3mm
			height 2.1mm} \hrule height .3mm}$ \bigskip}

\title{Curvature-Dimension for Autonomous Lagrangians}
\author{Rotem Assouline}
\date{}

\begin{document}

\maketitle
 
\begin{abstract}
	We introduce a curvature-dimension condition for autonomous Lagrangians on weighted manifolds, which depends on the Euler-Lagrange dynamics on a single energy level. By generalizing Klartag's needle decomposition technique to the Lagrangian setting, we prove that this curvature-dimension condition is equivalent to displacement convexity of entropy along cost-minimizing interpolations in an $L^1$ sense, and that it implies various consequences of lower Ricci curvature bounds, as in the metric setting. As examples we consider classical and isotropic Lagrangians on Riemannian manifolds. In particular, we generalize the horocyclic Brunn-Minkowski inequality to complex hyperbolic space of arbitrary dimension, and present a new Brunn-Minkowski inequality for contact magnetic geodesics on odd-dimensional spheres.
\end{abstract}

\tableofcontents

\section{Introduction}

\subsection{Background}

Riemannian manifolds with Ricci curvature bounded from below possess a variety of geometric and analytic properties. Some of those properties make sense beyond the classical setting of Riemannian geometry, and can be used to define more general notions of lower Ricci curvature bounds. The $\CD(K,N)$ condition of Bakry and {\'E}mery \cite{BGL}, which for a Riemannian manifold $(M,g)$ is equivalent to the conditions $\Ric_g \ge Kg$ and $\dim M \le N$, may be formulated in terms of a Bochner-type inequality; since it involves only the Laplace operator and the measure, it is a property of abstract Markov semigroups. The synthetic characterization of curvature-dimension bounds via optimal transport, initiated by Sturm \cite{Sturm2} and Lott-Villani \cite{LV}, provides a way of defining curvature-dimension bounds on geodesic metric measure spaces with no differentiable structure.

\medskip
While optimal transport on manifolds with general Lagrangian costs is a well-studied subject \cite{BB,BB07,Fig,FF,Vil,HPR}, to the best of our knowledge, the use of optimal transport to harness curvature-dimension bounds has not been explored for Lagrangians other than Riemannian, sub-Riemannian \cite{BKS,BMR}, Finslerian \cite{Oh09,Oh18,OhBook} and Lorentzian \cite{Mc20,CM24,MS,Br} metrics, with the exception of the works by Lee \cite{Lee} and Ohta \cite{Oh14}. It was shown in \cite{Lee} that a lower bound on a certain Ricci curvature associated to a Hamiltonian system implies convexity of entropy along smooth displacement interpolations with respect to the corresponding Lagrangian cost. Here we take a complementary approach, namely that of $L^1$, rather than $L^2$, mass transport, and do not make any a-priori assumptions on the regularity of the displacement interpolation. 

\medskip
One of the main goals of this work is to demonstrate the equivalence between three notions of curvature-dimension bounds for a weighted manifold endowed with a Lagrangian: the first, in terms of a quantity on the tangent bundle generalizing the weighted Riemannian (and Finslerian) Ricci curvature; the second, in terms of a Bochner-type formula for solutions to the Hamilton-Jacobi equation; and the third, in terms of displacement convexity of entropy along cost-minimizing interpolations between probability measures. 

\medskip
It is a feature of our curvature-dimension condition that it only depends on the Euler-Lagrange dynamics on a single energy level. If the Lagrangian comes from a Finsler metric (that is, if it is two-homogeneous on each fiber of the tangent bundle), then this determines the dynamics on all energy levels; but if there is no homogeneity then this is no longer true. By adding a constant to the Lagrangian, one can pass to a different energy level - but the validity of the curvature-dimension condition depends on the chosen energy level. At this point it is worth mentioning that our results do not follow from analogous results on Finsler manifolds, since in general there is no Finsler metric whose geodesic flow coincides with the Euler-Lagrange flow on the given energy level (although there is a Finsler metric whose unit-speed geodesics are \emph{reparametrizations} of solutions to the Euler-Lagrange equation on the given energy level \cite{IS}). 

\medskip
For more information on the synthetic characterization of curvature-dimension bounds, see \cite{St23,Vil,OhBook}. Notions of curvature for Lagrangian and  Hamiltonian systems were studied in \cite{Ag07,Ag,AG,BM,Fou,Gr}; for the special case of (electro)magnetic Lagrangians, which serve as one of our main examples, see \cite{Asse,BA,Gouda,Gro,Woj}. Displacement convexity in optimal transport with Lagrangian cost was studied previously in Schachter \cite{SchThesis} and in Yang \cite{Yang}. 

\subsection{Main results}
Before presenting our results, let us briefly introduce the objects involved in their formulation; precise definitions  are given in Chapters \ref{presec}, \ref{ricsec} and \ref{transportsec}. A \emph{Lagrangian} on a smooth manifold $M$ is a function 
$$L:TM \to \RR,$$
where $\pi:TM \to M$ is the tangent bundle. We thus concern ourselves with autonomous, i.e. time-independent, Lagrangians. The rather standard assumptions we impose on our Lagrangian are given in Section \ref{assumpsec}. Associated to such a Lagrangian is a Hamiltonian $$H : T^*M \to \RR, \qquad H(p) : = \sup_{v \in T_xM}[p(v) - L(v)], \quad p \in T^*_xM, \, x \in M.$$ 

\medskip
The \emph{cost} is the function $\cc:M\times M \to \RR$ defined by
\begin{equation}\label{ccdef}\cc(x_0,x_1) : = \inf_\gamma\int L(\dot\gamma(t))dt,\end{equation}
where the infimum is taken over curves joining $x_0$ to $x_1$. A curve realizing this infimum will be called a \emph{minimizing extremal}; under our assumptions it will always exist and solve the \emph{Euler-Lagrange equation}. Solutions of the Euler-Lagrange equation are projections to $M$ of trajectories of a flow $\Phi^L$ on $TM$, the \emph{Euler-Lagrange flow}.

\medskip
The \emph{Legendre transform} $\cL : T^*M \to TM$ provides an identification between the tangent and cotangent bundles. The \emph{gradient} $\nabla u$ of a function $u$ is defined to be the image of its differential under the Legendre transform: $$\nabla u : = \cL du.$$ The \emph{energy} is the function
$$E = H\circ\cL^{-1}:TM \to \RR.$$

\medskip
Since the infimum in \eqref{ccdef} allows the domain of $\gamma$ to be any finite interval, minimizing extremals satisfy $E(\dot\gamma)\equiv 0$, see Section \ref{costsec}. We will thus work exclusively on the energy level 
$$SM : = E^{-1}(0).$$
In the Riemannian case we take the Lagrangian to be $L=(g+1)/2$, so that $SM$ is the unit tangent bundle. In any case, the energy level $SM$ will never intersect the zero section of $TM$. 

\medskip
In addition to the Lagrangian $L$, we fix a measure $\mu$ on $M$ with a smooth density, and define an operator $\bL$ by 
$$\bL  u : = \div_\mu({\nabla}u).$$ 

\medskip
In Section \ref{ricsec} we construct a quantity on $TM$, a \emph{weighted Ricci curvature} $\Ric_{\mu,N}$ depending on the Lagrangian $L$, the measure $\mu$ and a parameter $N$, which generalizes the weighted Ricci curvature of a Riemannian or Finslerian manifold (see e.g. \cite{BGL,Sturm2,Vil}). The construction is based on the formalism of semisprays and nonlinear connections \cite{Gr,FL,BM,KMS}; the resulting notion of weighted Ricci curvature coincides with the one in \cite{Oh14} except for some extra terms which account for the volume distortion of the Euler-Lagrange flow in the tangential direction, see the discussion following Theorem \ref{mainthm}.

\medskip
Let us also recall some concepts from optimal transport theory (some of our definitions are a bit different from the standard ones). For two probability measures $\mu_0,\mu_1$ on $M$, set
$$\CC(\mu_0,\mu_1) := \inf_\kappa\int_{M\times M}\cc(x_0,x_1) \, d\kappa(x_0,x_1),$$
where the infimum is over couplings $\kappa$ between the measures $\mu_0$ and $\mu_1$, i.e. measures on $M\times M$ whose marginals are $\mu_0$ and $\mu_1$. A coupling attaining this infimum is called an \emph{optimal coupling} between $\mu_0$ and $\mu_1$. Let 
$$SM\times_0[0,\infty) : = \faktor{SM\times[0,\infty)}{\sim}$$
where $\sim$ is the equivalence relation defined by:
$$(v,0)\sim(v',0) \qquad \text{ whenever} \qquad \pi(v) = \pi(v'),$$
i.e. when $v$ and $v'$ share the same basepoint. For $0 \le \lambda\le 1$, define a map
$$\Exp_\lambda : SM\times_0[0,\infty)\to M, \qquad \Exp_\lambda(v,\ell) = \pi(\Phi^L_{\lambda\ell}(v))$$
(recall that $\Phi^L$ is the Euler-Lagrange flow). Thus $\Exp_\lambda(v,\ell)=\gamma(\lambda\ell)$, where $\gamma:[0,\ell]\to M$ is a solution to the Euler-Lagrange equation satisfying $\dot\gamma(0) = v$.

\medskip
A \emph{transport plan} is a Borel probability measure on $SM\times_0[0,\infty)$. A family $\{\mu_\lambda\}_{0 \le \lambda \le 1}$ of Borel probability measures on $M$ will be called a \emph{displacement interpolation} between $\mu_0$ and $\mu_1$ if:
\begin{itemize}
    \item There exists a transport plan $\Pi$ such that
    $$\mu_\lambda = (\Exp_\lambda)_*\Pi, \qquad \lambda \in [0,1].$$
    Informally, this means that $\mu_\lambda$ is the law of $\gamma(\lambda\ell)$, where $\gamma:[0,\ell]\to M$ is a random extremal such that the pair $(\dot\gamma(0),\ell)$ is distributed according to $\Pi$.
    \item For every $0 \le \lambda \le \lambda' \le 1$, the measure 
    $$\kappa_{\lambda,\lambda'}: = (\Exp_\lambda\times\Exp_{\lambda'})_*\Pi$$
    on $M\times M$ is an optimal coupling between $\mu_\lambda$ and $\mu_{\lambda'}$. 
\end{itemize} 

\medskip
For an absolutely-continuous probability measure $\mu_0 = f_0\mu$ on $M$ and for $N \in(1,\infty)$ we set
$$\rS_N[\mu_0|\mu] : = -\int_Mf_0^{-1/N}\,d\mu_0 = -\int_Mf_0^{1-1/N}\,d\mu.
$$
We also set
$$\rS_\infty[\mu_0|\mu] :=  
\int_M\log f_0\,d\mu_0 = \int_Mf_0\log f_0\,d\mu.
$$
We can now state our main result. Denote by $\cP_1(L)$ the collection of absolutely continuous probability measures on $M$ with finite first moment with respect to the cost $\cc$, see Section \ref{MKsec}.

\begin{Theorem}\label{mainthm}
    Let $M$ be a smooth manifold and let $L:TM \to \RR$ be a Lagrangian satisfying assumptions (I)-(III) from Section \ref{assumpsec} which is smooth away from the zero section. Let $\mu$ be a measure on $M$ with a smooth density, and let ${K} \in \RR$ and $N \in [n,\infty]$. Then the following conditions are equivalent:
    \begin{enumerate}[(i)]
        \item $\Ric_{\mu,N} \ge {K}$ on $SM$.
        \item Every local solution $u$ to the Hamilton-Jacobi equation $H(du) = 0$ satisfies
        \begin{equation}\label{mainBochnereq}(d\bL u)(\nabla u) + \frac{(\bL u)^2}{N-1} + {K} \le 0.\end{equation}
        \item For every $\mu_0 = f_0\mu,\mu_1 = f_1\mu \in \cP_1(L)$ there exists a displacement interpolation $\mu_\lambda = (\Exp_\lambda)_*\Pi$ between $\mu_0$ and $\mu_1$ such that for every $\lambda \in [0,1]$,
        $$\rS_{N}[\mu_\lambda|\mu] \le 
        \begin{cases}
            \int\left[\tau_{1-\lambda}^{{K},{N}}\cdot \left(-f_0^{-1/{N}}\circ\Exp_0\right) + \tau_{\lambda}^{{K},{N}}\cdot \left(-f_1^{-1/{N}}\circ\Exp_1\right)\right]d\Pi & {N} < \infty\\\noalign{\vskip9pt}
            (1-\lambda)\cdot\rS_\infty[\mu_0|\mu] + \lambda\cdot\rS_\infty[\mu_1|\mu] - \frac{K}{2}\cdot\lambda\cdot(1-\lambda)\cdot \int\ell^2 d\Pi & {N} =\infty,
        \end{cases}$$
        where the coefficients $\tau_t^{{K},N}:SM\times_0[0,\infty)\to[0,\infty]$ are defined in Section \ref{convsec}, and $\ell : SM\times_0[0,\infty)\to [0,\infty)$ denotes the second variable.
    \end{enumerate}
\end{Theorem}
\begin{Corollary}
    Let $N \in [n,\infty]$. Then the following conditions are equivalent:
    \begin{enumerate}[(i)]
        \item $\Ric_{\mu,N} \ge 0$ on $SM$.
        \item For every $\mu_0,\mu_1\in\cP_1(L)$ there exists a displacement interpolation $\{\mu_\lambda\}_{0 \le \lambda \le 1}$ between $\mu_0$ and $\mu_1$ such that the function $$\lambda\mapsto\rS_{N}[\mu_\lambda|\mu]$$ is convex.
    \end{enumerate}
\end{Corollary}

If any of the equivalent conditions in Theorem \ref{mainthm} is met then we say that the pair $(\mu,L)$ satisfies $\CD(K,N)$.

\begin{remark*}\normalfont
    The equivalence (ii)$\iff$(iii) in Theorem \ref{mainthm} requires only $C^2$ regularity of $L$ away from the zero section, see Theorem \ref{DCthm}. The equivalence (i)$\iff$(ii) holds for every $N \in (-\infty,\infty]\setminus[1,n)$, see Corollary \ref{equivalencecor}. The equivalence (ii)$\iff$(iii) can be extended to the range $N \in (-\infty,0]$ for an appropriate definition of $\rS_N$, see e.g. \cite[Theorem 18.6]{OhBook} and references therein; we do not include this regime in the present work.
\end{remark*}

\medskip
Theorem \ref{mainthm} extends to the Lagrangian setting the equivalence between lower bounds on weighted Ricci curvature and convexity of entropy along displacement interpolations in the space of probability measures \cite{Mc97,LV,Sturm2,Vil}. However, in contrast to the Lott-Sturm-Villani theory, here the characterization is in terms of $L^1$ optimal transport; see \cite{ACMCS,CGS,CMil} for analogous results in the metric-measure setting. A unique feature of the $L^1$ optimal transport problem is that its solution gives rise to a disintegration of the manifold into disjoint trajectories along which mass travels. In Section \ref{methodssec} we elaborate on how this fact can be put to use.

\medskip
A weaker form of the Bochner-type inequality \eqref{mainBochnereq}, with $N-1$ replaced by $N$, was proved in \cite{Lee}. In order to obtain the sharp inequality, it is necessary to separate the tangential and normal components of the Hessian of a solution to the Hamilton-Jacobi equation (this is the same as separating  the tangential and normal components of a Jacobi field; here `tangential' and `normal' is with respect to the gradient $\nabla u$). In Riemannian geometry, where the Hamilton-Jacobi equation reads $|du|_g = 1$, the Hessian of a solution has no tangential component (when $u$ is the distance function from a point, this is the content of Gauss' Lemma). In the Lagrangian setting this is no longer true, and it is therefore necessary to incorporate into the definition of the weighted Ricci curvature a term which measures the non-homogeneity of the Euler-Lagrange flow (called the \emph{deviation} in  \cite{Gr}), and which vanishes for Finsler Lagrangians, see Section \ref{spraysec}. 

\medskip
One significant consequence of the curvature-dimension condition for metric measure spaces is the Brunn-Minkowski inequality \cite{CMS,Sturm2,CM}, which gives a lower bound on the measure of a set formed by midpoints of geodesic segments joining a given pair of sets. We prove the following analogous result:

\begin{Theorem}[A Brunn-Minkowski inequality for $\CD(K,N)$ Lagrangians]\label{BMthm0}
    Suppose that the pair $(\mu,L)$ satisfies $\CD_{}({K},N)$ for some ${K} \in \RR$ and some $N \in [n,\infty]$. Then for every pair $A_0,A_1\subseteq M$ of Borel sets and every $0 < \lambda < 1$,
        $$\mu(A_\lambda) \ge \bM_{\frac{1}{N}}\left(\beta_{1-\lambda}^{{K},N}(A_0,A_1)\cdot\mu(A_0),\beta_{\lambda}^{{K},N}(A_0,A_1)\cdot\mu(A_1);\lambda\right),$$
        where
        $$A_\lambda : = \{\gamma(\lambda \ell) \, \mid \, \gamma:[0,\ell] \to M \, \, \text{ is a minimizing extremal, $\quad \gamma(0) \in A_0, \, \,\gamma(\ell) \in A_1$}\},$$
        and the generalized means $\bM$ and distortion coefficients $\beta_t^{{K},N}(A_0,A_1)$ are defined in Section \ref{BMsec}. In particular, if the pair $(\mu,L)$ satisfies $\CD(0,N)$ and if $\mu(A_0)\mu(A_1) > 0$ then
        \begin{equation}\label{BMsimpleeq}
            \mu(A_\lambda) \ge 
            \begin{cases}
                \left((1-\lambda)\cdot\mu(A_0)^{1/N} + \lambda \cdot \mu(A_1)^{1/N}\right)^N & N < \infty\\
                \mu(A_0)^{1-\lambda}\cdot\mu(A_1)^\lambda & N = \infty.
            \end{cases}
        \end{equation}
\end{Theorem}

In the Riemannian case, the Brunn-Minkowski inequality is known to imply the curvature-dimension condition \cite{MPR}. The extent to which this reverse implication holds for arbitrary Lagrangians will be the subject of a separate work.

\medskip
Theorem \ref{BMthm0} follows from a more general Borell-Brascamp-Lieb-type inequality, Theorem \ref{BBLthm}. In addition to the Brunn-Minkowski and Borell-Brascamp-Lieb inequalities, we show that the curvature-dimension condition implies other familiar results from Riemannian geometry, which were extended to Finsler manifolds and metric measure spaces in \cite{CMS,Sturm2,LV,Oh09,Oh18,CM17}: a Bonnet-Myers theorem (Theorem \ref{Bonnetthm}), a Bishop-Gromov inequality (Theorem \ref{BGthm}), isoperimetric inequalities by comparison to one-dimensional model spaces (Theorem \ref{isoperimthm}), and Functional inequalities (Theorem \ref{pointhm} and Theorem \ref{lsthm}).
 
\medskip
We discuss two types of examples of Lagrangians on Riemannian manifolds: \emph{Classical} Lagrangians, which are the sum of kinetic energy, potential energy and a magnetic vector potential, and \emph{isotropic} Lagrangians, which are spherically-symmetric on each fiber. As a particular case of the former, we obtain from Theorem \ref{BMthm0} a higher-dimensional analogue of the horocyclic Brunn-Minkowski inequality proved in \cite{AK}:

\begin{Theorem}[Horocyclic Brunn-Minkowski inequality in complex hyperbolic space]\label{horoBMthm}
    Let $\mathbb{C}\HH^d$ denote the complex hyperbolic space of complex dimension $d$. Let $A_0,A_1 \subseteq \mathbb{C}\HH^d$ be Borel sets of positive measure and let $0\le\lambda\le 1$. Denote by $A_\lambda$ the set of points of the form $\gamma(\lambda \ell)$, where $\gamma:[0,\ell]\to \mathbb{C}\HH^d$ is a unit-speed horocycle contained in a single complex geodesic and satisfying $\gamma(0) \in A_0$ and $\gamma(\ell) \in A_1$. Then
    $$\Vol(A_\lambda)^{1/n} \ge (1-\lambda)\cdot\Vol(A_0)^{1/n} + \lambda\cdot\Vol(A_1)^{1/n},$$
    where $\Vol$ denotes the hyperbolic 
    volume measure and $n = 2d$.
\end{Theorem}

We also obtain the following spherical companion to the horocyclic Brunn-Minkowski inequality, where this time the interpolating curves are \emph{contact magnetic geodesics} on odd-dimensional spheres, see Section \ref{magsec} for the definition. Each fiber of the Hopf fibration $\pi_{\mathrm{Hopf}}:S^{2d+1}\to\mathbb{C}\mathbf{P}^d$ is a contact magnetic geodesic.

\begin{Theorem}[Brunn-Minkowski inequality for contact magnetic geodesics on $S^{2d+1}$]\label{contactthm}
    Let $S^{2d+1} = \{z \in \mathbb{C}^{d+1}\, \mid \, |z|=1\}$ denote the unit sphere in $(d+1)$-dimensional complex Euclidean space. Let $A_0,A_1 \subseteq S^{2d+1}$ be Borel sets of positive measure and let $0\le\lambda\le 1$. Fix $0\le s < 1$ and denote by $A_\lambda$ the set of points of the form $\gamma(\lambda \ell)$, where $\gamma:[0,\ell]\to S^{2d+1}$ is a unit-speed contact magnetic geodesic of strength $s$ satisfying $\gamma(0) \in A_0$ and $\gamma(\ell) \in A_1$. Then
    $$\Vol(A_\lambda)^{1/n} \ge (1-\lambda)\cdot\Vol(A_0)^{1/n} + \lambda\cdot\Vol(A_1)^{1/n},$$
    where $\Vol$ denotes the spherical volume measure and $n = 2d+1$.
\end{Theorem}

When $s=0$, Theorem \ref{contactthm} follows from the (stronger) spherical Brunn-Minkowski inequality \cite{Cordero,CMS}, since contact magnetic geodesics of strength $0$ are ordinary geodesics (great circles). 

\medskip
Another interesting example is the Lagrangian of the classical \emph{many-body system} consisting of $k$ bodies $x_1,\dots,x_k \in \RR^d$ ($d \ge 3$), with masses $m_i > 0$ and gravitational constant $G > 0$:
$$L = \frac12\sum_{i=1}^km_i|\dot x_i|^2 + \frac12 + \sum_{1 \le i <  j \le k}\frac{Gm_im_j}{|x_i-x_j|}$$
(the constant $1/2$ is added since we always restrict to the energy level $\{E = 0\}$). We prove in Section \ref{magsec} that this Lagrangian satisfies $\CD(0,\infty)$ with respect to the Lebesgue measure on $(\RR^d)^k$; in fact, this remains true if we replace the last term with any positive function $U = U(x_1,\dots,x_k)$ which is superharmonic as a function on $(\RR^d)^k$, see \cite{Ag07,Lee,Oh14} and Remark \ref{shrmk} below.

\subsection{On the proofs}\label{methodssec}
In the proof of the implication (ii)$\implies$(iii) in Theorem \ref{mainthm}, as well as in most of the proofs in Section \ref{appsec}, we use the \emph{needle decomposition technique}, put forth by Klartag \cite{Kl} as a method for proving geometric and functional inequalities on Riemannian manifolds. Roughly speaking, a needle decomposition is a partition of a Riemannian manifold into a disjoint collection of geodesic segments, constructed so as to enable a change of variables which reduces the desired inequality to an analogous inequality on the real line. This strategy, known as \emph{localization}, originated in the works Gromov and Milman \cite{GM}, and Lov{\'a}sz and Simonovits \cite{LS}, and the idea behind it can be traced back to Payne and Weinberger \cite{PW}. Since the work of Klartag, the needle decomposition technique has been extended to metric measure spaces by Cavalletti and Mondino \cite{CM}, to Finsler manifolds by Ohta \cite{Oh18}, and more recently to Lorentzian metric measure spaces by Braun and McCann \cite{BrMc} and Cavalletti and Mondino \cite{CM24}. 

\medskip
The solution to the Monge-Kantorovich problem by Evans and Gangbo \cite{EG} and the ensuing works of Caffarelli, Feldman and McCann \cite{CFM} and Feldman and McCann \cite{FeMc}, provided a detailed description of the $L^1$ optimal transport map, which was utilized by Klartag in his construction of needle decompositions on Riemannian manifolds. Of importance to us are also the works of Bernard and Buffoni \cite{BB}, Figalli \cite{Fig} and Fathi and Figalli \cite{FF} on the Monge-Kantorovich problem for general Lagrangian costs, which give a clear idea of how to generalize the needle decomposition procedure to the Lagrangian setting. We adopt some of the terminology used in those works.

\medskip
The following needle decomposition theorem resembles the main result in \cite{Kl}.

\begin{Theorem}[Needle decomposition for $\CD(K,N)$ Lagrangians]\label{needlethm}
    Let $M$ be a smooth manifold and let $L:TM \to \RR$ be a Lagrangian satisfying assumptions (I)-(III) from Section \ref{assumpsec}. Let $\mu$ be a measure on $M$ with a smooth density. Suppose that the pair $(\mu,L)$ satisfies $\CD_{}({K},N)$ for some ${K} \in \RR$ and some $N \in (-\infty,\infty]\setminus[1,n)$. Let $f : M \to \RR$ be a $\mu$-integrable function satisfying $$\int_Mfd\mu = 0.$$
    Assume that for some $x_0 \in M$,
    $$\int_M\left(|\cc(x_0,\cdot)| + |\cc(\cdot,x_0)|\right)fd\mu < \infty.$$
    Then there exists a collection $\{\mu_\alpha\}_{\alpha \in \sA}$ of Borel measures on $M$ and a measure $\nu$ on the set $\sA$ such that the following hold:
    \begin{enumerate}[$(i)$]
        \item For $\nu$-almost every $\alpha \in \sA$, the measure $\mu_\alpha$ is either a Dirac measure, or a measure of the form
        \begin{equation}\label{needleformeq}\mu_\alpha = \left(\gamma_\alpha\right)_*m_\alpha\end{equation}
        where $m_\alpha$ is a measure on an interval $I_\alpha \subseteq \RR$ satisfying $\CD({K},N)$ with respect to the Euclidean metric on $\RR$, and $\gamma_\alpha : I_\alpha \to M$ is a minimizing extremal.
        \item \underline{Disinte}g\underline{ration o}f\underline{ measure}: The measure $\mu$ disintegrates as 
        $$\mu = \int_{\sA}\mu_\alpha d\nu(\alpha).$$
        More precisely, for every Borel function $h : M \to \RR$, the function $\alpha\mapsto\int_Mhd\mu_\alpha$ is $\nu$-measurable and
        \begin{equation}\label{disinteq}\int_Mh\,d\mu = \int_{\sA}\left(\int_Mhd\mu_\alpha\right)d\nu(\alpha).\end{equation}
        \item \underline{Mass Balance}: For $\nu$-almost every $\alpha \in \sA$,
        \begin{equation}\label{MBeq}
            \int_Mfd\mu_\alpha = 0.
        \end{equation}
        Moreover, if $\mu_\alpha$ takes the form \eqref{needleformeq}, then for every $t \in \RR$,
        \begin{equation}\label{detailedmbeq}
            \int_Mfd\mu_{\alpha,t} \ge 0 \qquad \text{ where }\qquad \mu_{\alpha,t} = (\gamma_\alpha)_*\left(m_\alpha\vert_{[t,\infty)}\right).
        \end{equation}
    \end{enumerate}
\end{Theorem}

Our proof of Theorem \ref{needlethm} is modeled on \cite{Kl}. While some modifications are due in the Lagrangian setting, many of the arguments carry over quite 
directly. Some of the arguments were simplified, notably the proof of the regularity theorem for dominated functions, Theorem \ref{regularitythm}. We hope that our nearly self-contained proof will serve as an additional source of reference to the rather involved construction laid out in \cite{Kl}. 

\medskip
The paper is organized as follows. In Section \ref{presec} we give some basic background on Lagrangians and Hamiltonians, and specify our assumptions on the Lagrangian. In Section \ref{ricsec}, after invoking some definitions from the theory of nonlinear connections, we define the weighted Ricci curvature and prove the equivalence (i)$\iff$(ii) in Theorem \ref{mainthm}. In Section \ref{domsec} we discuss dominated functions, which generalize $1$-Lipschitz functions from the Riemannian setting. In Section \ref{transportsec} we discuss optimal transport with Lagrangian cost and prove Theorem \ref{needlethm} and the equivalence (ii)$\iff$(iii) in Theorem \ref{mainthm}. In Section \ref{appsec} we use Theorem \ref{needlethm} to prove various results under the curvature-dimension condition. In Section \ref{examplesec} we consider examples, and prove Theorems \ref{horoBMthm} and \ref{contactthm}.

\subsection*{Acknowledgements}
I would like to express my deepest gratitude to my advisor, Prof. Bo'az Klartag, for his unwavering support and encouragement, and for being a true mentor and role model. I am also grateful to Prof. Itai Benjamini, for years of professional and personal mentorship; to Prof. Yanir Rubinstein, for hosting me at the University of Maryland during the writing of parts of this work; to Roee Leder, for helpful discussions; and to Prof. Uri Bader, for pointing me to a useful reference. Supported by a grant from the ISF.

\section{Lagrangians and Hamiltonians}\label{presec}
    \subsection{Tonelli Lagrangians}
    Let $M$ be a smooth manifold. A (\emph{time-independent}, or \emph{autonomous}) \emph{Lagrangian} is a function $$L : TM \to \RR.$$ References for the standard definitions and results mentioned below include \cite{CI,FM,CS} and \cite[Appendix B]{FF}. We denote the zero section of $TM$ by $\bz$.

    \begin{definition}[Tonelli Lagrangian]\normalfont
     We shall say that a Lagrangian $L$ is \emph{Tonelli} if it satisfies the following assumptions:

    \begin{itemize}
        \item \underline{Smoothness}: The Lagrangian $L$ is $C^1$-smooth on $TM$ and $C^2$-smooth on $TM\setminus\bz$.
        \item \underline{Strict convexit}y: The restriction of $L$ to each fiber of $TM$ is strictly convex, i.e. its Hessian is positive definite.
        \item \underline{Su}p\underline{erlinearit}y: For every compact set $A \subseteq M$ and every Riemannian metric $g$ on $M$ there exists a constant $C_{A,g} > 0$ such that
        $$L(v) \ge |v|_g - C_{A,g} \qquad \text{ for all } \, v \in T_xM, \, x \in A.$$ 
    \end{itemize}
\end{definition}

Naturally associated to a Tonelli Lagrangian is a \emph{Hamiltonian}, which is the function 
$$H : T^*M \to \RR, \qqquad H(p) : = \sup_{v \in T_xM}\left[p(v) - L(v)\right], \qquad p \in T^*_xM, \, x \in M.$$

Strict convexity and superlinearity of $L$ imply that the supremum is achieved at a unique point $\cL p \in T_xM$, where 
$$\cL  : T^*M \to TM$$
 is the Legendre transform, defined in local canonical coordinates $(x^j,v^j)$ on $TM$ and $(x^j,p_j)$ on $T^*M$ by the relation
\begin{equation}p_j = \frac{\partial L}{\partial v^j}\Big\vert_{\cL  p} \qquad 1\le j \le n,\end{equation}
or equivalently
\begin{equation}(\cL p)^j = \frac{\partial H}{\partial p^j}\Big\vert_{p}, \qquad 1\le j \le n.\end{equation}

Moreover, the Hamiltonian $H$ is $C^1$ smooth on $T^*M$ and $C^2$-smooth on $T^*M\setminus \cL^{-1}(\bz)$, strictly convex and superlinear, and the Legendre transform is a homeomorphism from $T^*M$ to $TM$ and a $C^1$ diffeomorphism from $T^*M\setminus\cL^{-1}(\bz)$ to $TM\setminus\bz$. 

\medskip
An \emph{extremal} is a $C^2$ curve $\gamma$ on $M$ solving the \emph{Euler-Lagrange equation}
\begin{equation}\label{ELeq}\frac{d}{dt}\left(\frac{\partial L}{\partial v^j}\Big\vert_{\dot\gamma(t)}\right) = \frac{\partial L}{\partial x^j}\Big \vert_{\dot\gamma(t)}, \qquad 1\le j \le n.\end{equation}

Since \eqref{ELeq} is a second-order ordinary differential equation, its solutions are projections to $M$ of integral curves of a flow on $TM$, namely the \emph{Euler-Lagrange flow} $\Phi^L$, whose infinitesimal generator is the \emph{Euler-Lagrange vector field} $X_L$ on $TM$.

\medskip
A \emph{Tonelli minimizer} is a piecewise-$C^1$ curve $\gamma:[a,b] \to M$ minimizing the \emph{action}
$$\int_a^bL(\dot\gamma(t))dt$$
among all piecewise-$C^1$ curves joining its endpoints and defined on the interval $[a,b]$. By Tonelli's theorem, every Tonelli minimizer is an extremal. 

\medskip
The Hamiltonian $H$ gives rise to a Hamiltonian flow $\Phi^H_t$ on $T^*M$, which is the flow of the vector field $X_H$ defined by the relation $$dH = \omega_0(X_H,\cdot),$$ where $\omega_0$ is the canonical symplectic form on $T^*M$ (given in canonical local coordinates on $T^*M$ by $dx^i\wedge dp_i$). The Euler-Lagrange flow and the Hamiltonian flow are conjugate via the Legendre transform:
$$\Phi^L_t : = \cL\circ\Phi_t^H\circ\cL^{-1}.$$
Solutions to \eqref{ELeq} are projections to $M$ of integral curves of the flow $\Phi^L$, and therefore also of the flow $\Phi^H$. The \emph{energy} is the function
$$E : TM \to \RR \qquad E : = H\circ\cL^{-1}.$$
If $\gamma$ is any solution to \eqref{ELeq} then $E(\dot\gamma)$ is constant. The definitions of $H,\cL$ and $E$ imply that
\begin{equation}
    \label{EplusLeq}E(v) + L(v) = (\cL^{-1}v)(v), \qquad v \in TM.
\end{equation}

\medskip
We define the \emph{indicatrix bundle} by
    $$SM : = \{v \in TM \, \vert \, E(v) = 0\}.$$
Thus $SM = \cL (S^*M)$, where
    $$S^*M := \{p \in T^*M \, \vert \, H(p) = 0\}.$$

\begin{remark*}
    We will concern ourselves only with the dynamics on $SM$. This is, of course, a normalization, and other energy levels can be considered by adding a constant to $L$ as long as it remains supercritical. In the Riemannian case we take the Lagrangian to be $L(v) = (|v|^2+1)/2$, in which case $SM$ is the unit tangent bundle since the energy is $E(v) = (|v|^2-1)/2$.
\end{remark*}

The (time-independent) \emph{Hamilton-Jacobi equation} is the first-order nonlinear partial differential equation
\begin{equation}\tag{HJ}\label{HJeq}H(du) = 0.\end{equation}
A local solution to \eqref{HJeq} with prescribed second-order data at a point can always be found:

\begin{lemma}\label{characlemma}
    Let $x \in M$, let $p \in S^*_xM$ and let $W$ be a Lagrangian subspace of $T_pT^*M$ which is contained in the kernel of $dH$. Let $k \ge 2$ and assume that $H$ is $C^k$ in a neighborhood of $p$ and that $0$ is a regular value of $H$. Then there exists a neighborhood $U \ni x$ and a $C^{k+1}$ solution $u: U \to \RR$ to \eqref{HJeq} such that $du\vert_x = p$ and the tangent space to  the graph of $du$ at $p$ is $W$.
\end{lemma}
\begin{proof}
    This amounts to finding a $C^k$ Lagrangian submanifold of $T^*M$ through $p$ which is contained in $S^*M$, a goal that can be  achieved using the method of characteristics, see e.g. \cite[Section 3.2]{Evans}, \cite[Section 1.5]{CS} or \cite[Section 5]{Fat03}. We include a short proof for completeness. By our assumptions, 
	$$X_H\vert_p \in W\subseteq T_pS^*M.$$
	Let $W_0\subseteq W$ be an $(n-1)$-dimensional subspace not containing $X_H$, and let $Q_0$ be any $(2n-2)$-dimensional submanifold of $S^*M$ containing $p$ and transverse to $X_H$, such that $W_0\subseteq T_pQ_0$. Since $0$ is a regular value of $H$, the submanifold $S^*M$ is $C^k$-smooth, so the submanifold $Q_0$ can be taken to be $C^k$ as well. The canonical symplectic form $\omega_0$ restricts to a symplectic form on $Q_0$, and $W_0$ is a Lagrangian subspace of $T_pQ_0$. Using Darboux coordinates on $Q_0$ in a neighborhood of $p$, we can find a Lagrangian submanifold $N_0 \subseteq Q_0$ containing $p$, such that $T_pN_0 = W_0$. 
    
    \medskip
    Since the Hamiltonian vector field $X_H$ is transverse to $N_0$, by applying a flowout of $X_H$ to the submanifold $N_0$ we obtain a Lagrangian submanifold $N \subseteq T^*M$ (see e.g. \cite[Theorem 22.23]{LeeSM}). Since $N_0 \subseteq S^*M$ and $X_H$ is tangent to $S^*M$, the submanifold $N$ is contained in $S^*M$, and by our construction, the tangent space to $N$ at $p$ is spanned by $W_0$ and $X_H$ hence equals $W$.  Since $H$ is $C^k$ in a neighborhood of $p$, the Hamiltonian vector field $X_H$ is $C^{k-1}$, whence the resulting Lagrangian submanifold $N$ is $C^k$.
\end{proof}

\subsection{Cost and supercriticality}\label{costsec}

A Lagrangian gives rise to a cost function on $M\times M$, obtained by minimizing the action among all curves joining a given pair of points. If $L$ is a Riemannian or Finslerian Lagrangian, then this cost is simply distance.

\begin{definition}[Cost and minimizing extremals]\normalfont
    The \emph{cost} associated with $L$ is the function 
    $$\cc: M\times M \to \RR, \qquad \cc(x,x') : = \inf_\gamma\int L(\dot\gamma(t))dt,$$
    where the infimum is over all piecewise-$C^1$ curves joining $x$ to $x'$, without restriction on their domain. A curve $\gamma$ achieving the infimum above will be called a \emph{minimizing extremal}. Clearly if $\gamma$ is a minimizing extremal then it is a Tonelli minimizer, hence an extremal. If $\gamma$ is defined on an unbounded interval then we say that it is a minimizing extremal if its restriction to every finite subinterval is a minimizing extremal. 
\end{definition}

If $\gamma$ is a minimizing extremal then
$$E(\dot\gamma) \equiv 0.$$
Indeed, $E(\dot\gamma)$ is constant since $\gamma$ solves the Euler-Lagrange equation, and the fact that this constant is $0$ can be proved by considering the variation $\gamma_\lambda(t) = \gamma(\lambda t)$ and differentiating the action $\int_0^{T/\lambda}L(\dot\gamma_\lambda(t))dt$ at $\lambda =1$. Thus, every minimizing extremal is the projection to $M$ of an integral curve of the Euler-Lagrange flow $\Phi^L$ on the indicatrix bundle $SM$.

\begin{lemma}\label{triangleineqlemma}
    The function $\cc$ satisfies the triangle inequality:
    $$\cc(x,x') + \cc(x',x'') \ge \cc(x,x''), \qquad x,x',x'' \in M,$$
    with equality if and only if $x,x',x''$ lie on a single minimizing extremal in this order.
    A curve $\gamma : I \to M$ is a minimizing extremal if and only if for every $t < t' < t'' \in I$,
    \begin{equation}\label{gammacollineq}\cc(\gamma(t),\gamma(t')) + \cc(\gamma(t'),\gamma(t'')) = \cc(\gamma(t),\gamma(t'')).\end{equation}
\end{lemma}

\begin{proof}
    The triangle inequality is a direct consequence of the fact that the action is additive with respect to concatenation of curves. If equality holds, then by concatenating the minimizing extremal joining $x$ to $x'$ to the one joining $x'$ to $x''$, we obtain a minimizing extremal joining $x$ to $x''$ which must be a solution to the Euler-Lagrange equation by Tonelli's theorem, whence $x,x',x''$ lie on a single minimizing extremal. If $\gamma$ is a minimizing extremal then equality must hold in \eqref{gammacollineq}, for otherwise the concatenation of the minimizing extremals from $x$ to $x'$ and from $x'$ to $x''$ gives a curve joining $x$ to $x''$ with smaller action than that of $\gamma$. 
\end{proof}



We shall impose a condition on our Lagrangian which is known as Ma{\~n}{\'e} supercriticality \cite{Man,CDI,CIPP,BB}. In order for the cost $\cc$ to make sense, the action of a closed curve should be at the very least nonnegative; indeed, otherwise, by traversing a closed curve with negative action many times, every pair of points could be joined by a curve with arbitrarily low action, and we would have $\cc\equiv-\infty$. We shall require a bit more, in order to avoid some intricacies in the critical case.

\begin{definition}\label{SCdef}\normalfont
    A Tonelli Lagrangian $L:TM \to \RR$ will be called \emph{supercritical} if for every closed curve $\gamma:[a,b] \to M$,
    $$\int_a^b L(\dot\gamma(t))dt > 0.$$
    In particular,
        $$\cc(x,x') + \cc(x',x) > 0$$
    for every $x \ne x' \in M$.    
\end{definition}

\begin{remark*}
    By taking $\gamma$ in Definition \ref{SCdef} to be constant we see that supercriticality implies that
    $$L\vert_{\bz} > 0,$$
    and consequently, by \eqref{EplusLeq}, that
    $$E\vert_{\bz} < 0.$$
    In particular,
    $$SM\cap \bz = \varnothing.$$
\end{remark*}

If $c(L)$ denotes the Ma{\~n}{\'e} critical value of $L$ (which is the infimum over constants $c \in \RR$ such that the cost associated with $L + c$ is not identically $-\infty$, see \cite{CDI}) then we have the chain of implications:
$$c(L) < 0 \implies \text{$L$ is supercritical in the sense of Definition \ref{SCdef}} \implies c(L) \le 0.$$

Both implications cannot be reversed: the  Lagrangian \eqref{horoLagrangianeq} on the hyperbolic plane, which gives rise to the horocycle flow, satisfies the middle condition but not the leftmost, while the contact magnetic Lagrangian \eqref{contactLagrangianeq} on, say, the unit sphere $S^3\subseteq\mathbb{C}^4$, satisfies the rightmost condition but not the middle one when $s=1$.




\subsection{Assumptions on the Lagrangian}\label{assumpsec}

    We make the following assumptions on the Lagrangian $L$:
    \begin{enumerate}[(I)]
        \item The Lagrangian $L$ is Tonelli and supercritical, and $0$ is a regular value of the energy $E$.
        \item For every $x,x' \in M$ there exists at least one minimizing extremal joining $x$ to $x'$. In particular, $M$ is connected.
        \item For every compact set $A \subseteq M$ there exists a compact set $\tilde A \supseteq A$ such that every minimizing extremal with endpoints in $A$ is contained in $\tilde A$.
    \end{enumerate}
\begin{remark*}
    We do not require completeness of the Euler-Lagrange flow. In particular, if $L$ is a Finsler Lagrangian, we don't require the Finsler manifold to be complete (in this case assumption (II) means that the manifold is geodesically-convex).
\end{remark*}
\medskip
Under assumptions (I)-(III), the cost $\cc$ is finite and continuous. Moreover
$$\cc(x,x) = 0 \qquad \text{for all $x \in M$},$$
as can be seen by taking $\gamma$ to be a constant curve defined on an arbitrarily short time interval.

\begin{example}\normalfont\label{Lagrangiansexample}
    In the following examples, the Lagrangian $L$ satisfies all the assumptions.
    \begin{itemize}
        \item \underline{Tonelli La}g\underline{ran}g\underline{ians on com}p\underline{lete Riemannian manifolds}: $(M,g)$ is a complete Riemannian manifold and $L = L' + c$ where $L'$ is any Tonelli Lagrangian satisfying uniform superlinearity and convexity conditions and $c$ is a sufficiently large constant, see \cite{FM,CDI}. If $M$ is compact, then uniform superlinearity and convexity follow from the Tonelli assumption. 
        \item \underline{Geodesicall}y\underline{ convex Riemannian and Finsler metrics}: $(M,g)$ is a geodesically-convex Riemannian manifold and 
        $$L(v) = \frac{|v|_g^q+ q - 1}{q}, \qquad v \in TM$$
        for some $q > 1$. More generally, $|v|_g$ can be replaced by $F(v)$, where $F$ is a Finsler metric. See \cite{OhBook} for more details on Finsler metrics.
        \item \underline{Classical La}g\underline{ran}g\underline{ians}: $(M,g)$ is a Riemannian manifold, and 
        \begin{equation*}
            L(v) = \frac{|v|_g^2}{2} + U(x) - \eta(v), \qquad v \in T_xM, \, \, x \in M,
        \end{equation*}
        where $U$ is a smooth positive function and $\eta$ is a 1-form, see Section \ref{magsec} (note the difference in sign from the convention used in classical mechanics). If $|\eta|_g < \sqrt{2U}$ then $L$ is supercritical; if, in addition, $g$ is complete and $U$ is bounded from above and from below by positive constants then $L$ satisfies (I)-(III).
    \end{itemize} 
\end{example}

Supercriticality, together with assumptions (II) and (III), imply some important global properties.

\begin{definition}[Diameter]\normalfont
    The \emph{diameter} of a set $A \subseteq M$ is defined to be
    $$\diam(A) : = \sup\{\ell > 0 \, \mid \, \text{$\exists$ a minimizing extremal $\gamma:[0,\ell] \to M$ such that $\gamma(0),\gamma(\ell) \in A$}\}.$$
\end{definition}
\begin{lemma}\label{diamlemma}
    Every compact set has finite diameter. If the Euler-Lagrange flow is complete then the converse is true: a set of finite diameter is contained in a compact set.
\end{lemma}
\begin{proof}
    Assume by contradiction that there exists a compact set $A \subseteq M$ and a sequence of minimizing extremals $\gamma_i:[0,t_i]\to M$  with endpoints in $A$ such that $t_i \to \infty$. By assumption (III), the curves $\gamma_i$ are all contained in a fixed compact set $\tilde A \supseteq A$. By superlinearily of $H$, the set
    $$S\tilde A = \{v \in T_xM \, \mid \, E(v) = 0, \, x \in \tilde A\}$$
    is compact, and $\dot\gamma_i(0)\in S\tilde A$ for all $i$ since $\gamma_i$ are minimizing extremals. We may therefore pass to a subsequence and assume that $\dot\gamma_i(0)$ converge to a tangent vector $v \in T_xM$ for some $x \in A$. Since each $\gamma_i$ is a minimizing extremal, the sequence $\gamma_i$ converges on compact subsets of the half-line $[0,\infty)$ to a minimizing extremal $\gamma$ defined on the entire ray $[0,\infty)$ and contained in the compact set $\tilde A$. Now we use the argument in \cite[Theorem V]{CDI} to get a contradiction to supercriticality. Since $E(\dot\gamma) \equiv 0$, compactness of $S\tilde A$ and the fact that the $SM$ does not intersect the zero section imply that there exists an increasing sequence $m_j$ of integers and some $0 < \eps < 1$ such that the sequence 
	$$\{(\dot\gamma(m_j),\dot\gamma(m_j+\eps))\}_{j =1}^\infty \in SM \times SM$$
	converges to some element of $SM\times SM$ whose components have distinct basepoints $z,w \in \tilde A$, i.e.
	$$z_j : = \gamma(m_j) \to z  \qquad \text{ and } \qquad w_j : = \gamma(m_j + \eps) \to w .$$
	Since $\gamma$ is a minimizing extremal, 
	\begin{align*}
		\cc(z,w) + \cc(w,z) & = 
		\lim_{j \to \infty}\left(\cc(z_j,w_j) + \cc(w_j,z_{j+1})\right)
		\\ & = \lim_{j \to \infty} \left(\int_{m_j}^{m_j + \eps}L(\dot\gamma(t))dt + \int_{m_j + \eps}^{m_{j+1}}L(\dot\gamma(t))dt\right)\\
		& = \lim_{j\to \infty} \left(\int_{m_j}^{m_{j+1}}L(\dot\gamma(t))dt\right)\\
		& = \lim_{j \to \infty} \cc(z_j,z_{j+1}) \\
        & = \cc(z,z)\\
		& = 0,
	\end{align*}
	a contradiction to the assumption that $L$ is supercritical.

    \medskip
    Assume now that $A$ is a closed set of diameter $D < \infty$ and that the Euler-Lagrange flow is complete, i.e. $\Phi^L:\RR\times TM$ is defined and continuous on all of $\RR \times TM$ (where we write $\Phi^L(t,x) = \Phi^L_t(x))$. Let $x_0 \in A$. By the definition of diameter, the set $A$ is contained in the set 
    $$\pi(\Phi^L(S_{x_0}M\times[0,D])),$$
    which is compact (here $S_{x_0}M = SM\cap T_{x_0}M$ is the indicatrix at $x_0$).
\end{proof}

\begin{lemma}\label{localconvlemma}
    For every $x \in M$ and every neighborhood $U \ni x$ there exists a neighborhood $U \supseteq U' \ni x$ such that every $x',x'' \in U'$ are joined by a unique minimizing extremal which lies entirely in $U$. In particular, for every extremal $\gamma$ there exists $t > 0$ such that $\gamma\vert_{[0,t]}$ is uniquely minimizing.
\end{lemma}

\begin{proof}
    Assume by contradiction that for every neighborhood $U' \ni x$ there exist $x',x''$ which are joined by two distinct minimizing extremals. Then there are two sequences $\gamma_i:[0,T_i] \to M$ and $\theta_i:[0,S_i] \to M$ of minimizing extremals such that $\gamma_i$ and $\theta_i$ are distinct for each $i\ge 1$, but share the same endpoints, and those endpoints converge to $x$. By passing to a subsequence we may assume that $\dot\gamma_i(0),\dot\theta_i(0)$ both converge to $v,w \in S_xM$, respectively. Since $\gamma_i$ and $\theta_i$ are distinct solutions to Euler-Lagrange, the inverse function theorem implies that it cannot be the case that both $T_i$ and $S_i$ converge to zero. On the other hand, they are both bounded by Lemma \ref{diamlemma}. Thus at least one of the sequences of curves has to have a subsequence converging to a closed curve; but a limit of minimizing extremals is a minimizing extremal, and closed minimizing extremals do not exist by supercriticality. This contradiction shows that there exists a neighborhood $U'\subseteq U$ containing $x$, such that every $x',x'' \in U'$ are joined by a unique minimizing extremal. Furthermore, if $U'$ is small enough then this minimizing extremal must lie in $U$. Indeed, otherwise there exists a sequence $\gamma_i$ of minimizing extremals with both endpoints tending to $x$ and with at least one point outside $U$; a similar argument shows that such a sequence will converge to a closed minimizing extremal which, again, is ruled out by supercriticality.
\end{proof}

Let $u : M \to \RR$ be a $C^1$ function. The \emph{gradient} of $u$ is the vector field
$$\nabla u : = \cL du.$$

\begin{lemma}\label{HJlemma}
   Let $U \subseteq M$ be an open set and let $u$ be a $C^1$ function. Then $u$ solves the Hamilton-Jacobi equation \eqref{HJeq} if and only if the integral curves of $\nabla u$ are zero-energy extremals.
\end{lemma}
\begin{proof}
    If integral curves of $\nabla u$ are zero-energy extremals then in particular $0 = E(\nabla u) = H(du)$. Conversely, suppose that $H(du) = 0$ on $U$. Let $T >0$ and let $\theta:[0,T] \to U$ be a piecewise-$C^1$ curve. By the definitions of $H$ and $\cL$,
    \begin{align*}
        \int_0^T\left[du(\dot\theta(t)) - L(\dot\theta(t))\right]dt & \le  \int_0^T\sup_{v \in T_{\theta(t)}M}\left[du(v) - L(v)\right]dt\\
        & = \int_0^TH(du\vert_{\theta(t)})dt \\
        & = 0,
    \end{align*}
    with equality if and only if $\dot\theta \equiv \cL du = \nabla u$.
    It follows that
    $$\int_0^TL(\dot\theta(t))dt \ge u(\theta(T)) - u(\theta(0)),$$
    and equality holds if and only if $\theta$ is an integral curve of $\nabla u$. 
    
    \medskip
    Now let $T' > 0$, let $\gamma:[0,T'] \to U$ be an integral curve of $\nabla u$ and let $t \in [0,T']$. By Lemma \ref{localconvlemma}, there exists $\eps >0$ such that for every $t-\eps<t'<t''<t+\eps$, the points $\gamma(t')$ and $\gamma(t'')$ are joined by a unique minimizing extremal which is contained in $U$; the above argument shows that this minimizing extremal must be $\gamma\vert_{[t',t'']}$.
\end{proof}

\section{Ricci curvature of Lagrangians}\label{ricsec}

The notion of Ricci curvature for Lagrangians developed below is equivalent to the one introduced in Lee \cite{Lee} following Agrachev and Gamkrelidze \cite{AG} (see also Ohta \cite{Oh14}). However, our approach is closer to the one taken in Grifone \cite{Gr} (see also \cite{FL,BM,MHSS}). We will depart from the above sources when dealing with \emph{weighted} Ricci curvature, since we wish to incorporate into it the tangential component of the volume distortion of the Euler-Lagrange flow; this will be done in Section \ref{weightedsec}. 

\subsection{Semisprays and connections}\label{spraysec}

We state below without proof various well-known facts regarding nonlinear connections, and refer the reader to \cite{Gr,KMS,Lee,MHSS,Sh} for more details. Let $M$ be a smooth manifold and let 
$$\pi:TM \to M$$
be its tangent bundle. Denote by
$$d\pi : TTM \to TM$$
the differential of the projection $\pi$.

\medskip
For the reader's convenience, most of the constructions described below will be presented both in coordinate-free notation and in local coordinates; in the latter case we use canonical local coordinates
$$x^1,\dots,x^n,v^1,\dots,v^n$$
on $TM$. Einstein's summation convention is in force throughout. The notation $\partial_{x^i}$ stands for both the coordinate vector fields on $M$ and their counterparts on $TM$.

\medskip
The \emph{canonical lift} of a curve $\gamma$ on $M$ is the curve $\dot\gamma$ on $TM$, and the \emph{canonical lift} of a vector field $X$ on $M$ is the vector field
$$X_*X \quad \text{ on } \quad \Graph(X).$$
Here by $X_*X$ we mean the pushforward of the vector field $X$ via the map $X : M \to TM$, whose image is denoted by $$\Graph(X)\subseteq TM.$$
In coordinates, if $X = X^i\partial_{x^i}$, then
$$X_*X = X^i\partial_{x^i} + (X^i\partial_{x^i}X^j)\partial_{v^j}.$$

The \emph{vertical bundle} $\cV TM$ is the subbundle of $TTM$ defined by:
$$\cV TM := d\pi^{-1}(\bz) =  \{\xi \in TTM \, \mid \, d\pi(\xi) = 0 \}.$$
In coordinates,
$$\cV TM = \mathrm{span}(\partial_{v^1},\dots,\partial_{v^n}).$$
A section of $\cV TM$ is called a \emph{vertical vector field}. There is a natural bundle map 
$$\downarrow\,:\cV TM\to TM$$
given in coordinates by
$$\downarrow:\cV_vTM\to T_xM, \qquad \downarrow(\partial_{v^i}) = \partial_{x^i}, \qquad v\in T_xM, \,\,x\in M.$$

Given a vector field $X$ on $M$, there is a unique vertical vector field $\uparrow X$ on $TM$, called the \emph{vertical lift} of $X$, such that $\downarrow(\uparrow X)\vert_v = X\vert_x$ for every $x \in M$ and every $v \in T_xM$. In coordinates,
$$\uparrow (X^i\partial_{x^i}) = X^i\partial_{v^i}.$$
The \emph{almost tangent structure} is the fiberwise linear map $\rJ:TTM \to \cV TM$ given in local coordinates by
$$\rJ\partial_{x^i} = \partial_{v_i},\qquad \rJ\partial_{v^i} = 0, \qquad 1 \le i \le n.$$
Note that $\rJ^2 = 0$. The coordinate-free definitions of $\rJ$ and $\downarrow$ involve the pullback bundle $\pi^*TM$, see \cite{Gr} for details.

\medskip
Denote by $V$ the vertical vector field on $TM$ whose flow is given by $v \mapsto e^{t}v$ for all $v \in TM$ and $t \in \RR$. In coordinates:

$$V = v^i\partial_{v^i}.$$

\medskip
A \emph{semispray} is a vector field $\Sigma$ on $TM$ which is smooth on $TM\setminus\bz$ and satisfies
\begin{equation}\label{sprayeq}
    \rJ\Sigma = V.
\end{equation}

Condition \eqref{sprayeq} means that integral curves of $\Sigma$ are canonical lifts of curves on $M$. In coordinates, a semispray $\Sigma$ takes the form

$$\Sigma = v^i\partial_{x^i} + \Sigma^i\partial_{v^i}$$
for some locally-defined functions $\Sigma^i$ obeying the transformation rule
$$\tilde\Sigma^i = \frac{\partial\tilde x^i}{\partial x^j}\,\Sigma^j + \frac{\partial^2\tilde x^i}{\partial x^j\partial x^k}\,v^jv^k$$
under a change of coordinates $x^i \mapsto \tilde x^i$. A curve $\gamma$ on $M$ is called a \emph{geodesic} of the semispray $\Sigma$ if its canonical lift $\dot\gamma$ is an integral curve of $\Sigma$.

\medskip
A semispray $\Sigma$ gives rise to a (nonlinear) \emph{ connection} on $TM$. This is by definition a splitting of the double tangent bundle $TTM$ into a Whitney sum:
$$TTM = \cV TM\oplus \cH TM,$$
where $\cH TM$ is a subbundle called the \emph{horizontal bundle}; a vector $\xi \in \cH TM$ is  called \emph{horizontal}. Thus specifying a connection is the same as specifying a fiberwise linear projection $\pi_\cV$ whose image is $\cV TM$. Let
$$\pi_\cH : = \rI - \pi_\cV$$
denote the resulting projection onto $\cH TM$.  The connection associated to a semispray $\Sigma$ is characterized by the relation
\begin{equation}\label{piHdef}2\pi_\cH = \rI + [\rJ,\Sigma]= \rI  + [\rJ\,\cdot\,,\Sigma] - \rJ[\,\cdot\,,\Sigma],\end{equation}
where $[\cdot,\cdot]$ is the Fr{\"o}licher-Nijenhuis bracket. In coordinates:

$$\pi_{\cH}(\partial_{x^i}) = \partial_{x^i} - \Gamma_i^j\partial_{v^j}, \qquad \pi_{\cH}(\partial_{v^i}) = 0, \qquad 1 \le i \le n,$$
where $\Gamma_i^j$ are the \emph{connection coefficients}:
$$\Gamma_i^j : = -\frac12\cdot\partial_{v^i}\Sigma^j.$$
It follows from \eqref{sprayeq} and \eqref{piHdef} that
\begin{equation*}
    \pi_\cV\Sigma = \Lambda : = \frac{\Sigma - [V,\Sigma]}{2}.
\end{equation*}
In coordinates:
$$\Lambda = \Lambda^j\partial_{v^j} \qquad \text{ where } \qquad \Lambda^j : = v^i\Gamma_i^j + \Sigma^j.$$
\begin{remark}
    The vector field $\Lambda$ measures the deviation of the semispray $\Sigma$ from being 2- homogeneous on each fiber. Semisprays satisfying $\Lambda = 0$ are called \emph{sprays}. The infinitesimal generator of the geodesic flow of a Finsler manifold is a spray. 
\end{remark}
\medskip
Associated to a connection is an \emph{almost complex structure} $\rF$ on $TTM$, characterized by the relations
$$\rF\rJ = \pi_{\cH} \qquad \text{ and }\qquad \rF\pi_{\cH} = -\rJ.$$
In coordinates:
$$\rF(\partial_{v^i}) = \partial_{x^i} - \Gamma_i^j\partial_{v^j}, \qquad \rF\left(\partial_{x^i} - \Gamma_i^j\partial_{v^j}\right) = -\partial_{v^i}, \qquad 1 \le i \le n.$$
Clearly $\rF^2 = - \rI$. For convenience let us define vector and covector fields

$$E_i : = \partial_{x^i} - \Gamma_i^j\partial_{v^j}, \qquad \eps^j : = dv^j + \Gamma_i^jdx^i, \qquad 1 \le i \le n.$$
Together with the vector and covector fields $dx^i$ and $\partial_{v^i}$, they constitute a local frame $(E_i,\partial_{v^i})$ and a dual coframe $(dx^i,\eps^i)$. 
With this notation,
\begin{equation}\label{frameeq}\rJ = dx^i\otimes \partial_{v^i}, \quad  \pi_\cH = dx^i\otimes E_i, \quad \pi_\cV = \eps^i\otimes\partial_{v^i} \quad \text{ and } \quad \rF = d\eps^i\otimes E_i - dx^i\otimes \partial_{v^i}.\end{equation}

\medskip
The \emph{covariant derivative} induced by the connection is given for two vector fields $X,Y$ on $M$ by
$$\rD_XY : = \,\downarrow\pi_{\cV}(Y_*X).$$
This is well defined since $Y_*X$ is a vector field on the submanifold $\Graph(Y)$, which intersects each fiber once, so the fiberwise linear map $\downarrow$ maps it to a vector field on $M$. More precisely
$$\rD_XY\vert_x = \,(\downarrow\circ\,\pi_{\cV}\circ dY)(X\vert_x), \qquad x\in M.$$
In coordinates, if $X = X^i\partial_{x^i}$ and $Y = Y^i\partial_{x^i}$, then
$$\rD_XY = \left(X^k\partial_{x^k}Y^j + X^k\cdot (\Gamma_k^j\circ Y)\right)\partial_{x^j}.$$
Note that $\rD_XY$ is linear only in $X$. If $X$ is a vector field on $M$ then
\begin{equation}\label{Lambdaeq}X_*X =\Sigma\vert_{\mathrm{Graph(X)}} \qquad \iff \quad \rD_XX = \,\downarrow(\Lambda\circ X).\end{equation}
Indeed, $\rJ(X_*X) = V = \rJ\Sigma$ since $X_*X$ is a canonical lift, so by \eqref{frameeq}, $X_*X = \Sigma$ if and only if $\pi_\cV(X_*X) = \pi_{\cV}\Sigma = \Lambda$. Now apply $\downarrow$ to both sides.

\begin{definition}[Ricci curvature]\normalfont 
    We define the \emph{Ricci curvature} of the semispray $\Sigma$ by
    $$\Ric : TM\setminus\bz \to \RR, \qquad \Ric : = \tr(\rF\pi_\cV[\Sigma,\pi_\cH]).$$
    In coordinates:
    \begin{equation}\label{Riccicoordeq}
    \begin{split}
    \Ric = & \eps^i([\Sigma,E_i]).
    \end{split}\end{equation}
\end{definition}

\begin{remark*}\normalfont
    The \emph{curvature} of a connection is the vector-valued two-form on $TM$ defined by
    $$\rR : = -\frac12[\pi_{\cH},\pi_{\cH}].$$
    In coordinates:
    $$\rR = \rR_{ij}^k\,dx^i\otimes dx^j\otimes \partial_{v^k} \qquad \text{ where } \qquad \rR_{ij}^k = \partial_{x^i}\Gamma_j^k - \partial_{x^j}\Gamma_i^k + \Gamma_i^\ell\partial_{v^\ell}\Gamma_j^k - \Gamma_j^\ell\partial_{v^\ell}\Gamma_i^k.$$
    If $\Lambda = 0$ then $\Ric = -\tr(\rF \circ i_\Sigma\rR) = \eps^i(\rR(E_i,\Sigma))$ and we have the formula
    $$\Ric = v^j(\partial_{x^i}\Gamma_j^i - \partial_{x^j}\Gamma_i^i + \Gamma_i^k\partial_{v^k}\Gamma_j^i - \Gamma_j^k\partial_{v^k}\Gamma_i^i).$$
\end{remark*}
\begin{example}\normalfont
    If $\Sigma$ is the geodesic vector field of a Riemannian metric $g$, then $\Ric$ is the Ricci curvature of the metric, viewed as a function on $TM$ (i.e. $v \mapsto \Ric_g(v) = \Ric_g(v,v)$). More generally, if $\Sigma$ is the geodesic vector field of a Finsler metric $F$ then $\Ric$ is the Finslerian Ricci curvature (see \cite[Section 5.3]{OhBook}; keep in mind that our notation is a bit different). This corresponds to taking $L = F^2/2$ in the next section.
\end{example}
\subsection{The connection associated to a Lagrangian}\label{Lagrangianconnectionsec}

Let $L:TM \to \RR$ be a Lagrangian satisfying assumptions (I)-(III) from Section \ref{assumpsec}. 
In the present chapter we shall also assume that $L$ is smooth on $TM \setminus \bz$. The Legendre transform
$$\cL : T^*M \to TM$$
is then a diffeomorphism from $T^*M\setminus\bz$ to $TM\setminus\bz$. 
The Euler-Lagrange vector field 
$$\Sigma : = X_L$$
is a semispray, given in coordinates by

$$\Sigma = v^i\partial_{x^i} + \Sigma^i\partial_{v^i} \qquad \text{ where } \qquad \Sigma^i = L^{v^iv^j}\left(L_{x^j} - v^kL_{v^jx^k}\right).$$
Here subscripts denote partial derivatives, and superscripts indicate that a matrix inverse was taken. A computation reveals that the connection coefficients are

\begin{equation}\label{LagrangianGammaeq}\Gamma_i^j = \frac12 L^{v^jv^k}\left(\Sigma L_{v^iv^k} + L_{v^kx^i} - L_{v^ix^k}\right).\end{equation}



Here and below, a vector field followed by a function (e.g. $Xf$ or $X(f)$) will always denote the derivative of the function with respect to the vector field. Let 
$$\omega_0^*: = (\cL^{-1})^*\omega_0$$
denote the pullback of  the canonical symplectic form $\omega_0$ on $T^*M$ via the inverse Legendre transform, and let $g$ be the $(2,0)$-tensor on $TM$ given by
$$g := \omega_0^*(\rF\,\cdot,\cdot),$$
where $\rF$ is the almost complex structure introduced in the previous section. In coordinates:
\begin{equation}\label{gdef}\omega_0^* = g_{ij}\,dx^i\wedge \eps^j \qquad \text{ and } \qquad g = g_{ij}\,dx^i dx^j, \qquad \text{ where }\quad g_{ij} := L_{v^iv^j}.\end{equation}
If we take traces in \eqref{LagrangianGammaeq}, the second and third terms on the right hand side cancel out, and we get
\begin{equation}\label{trGammaeq}\Gamma_i^i = \frac12 L^{v^iv^k}\Sigma L_{v^kv^i} = \Sigma(\log\sqrt{\det g}).\end{equation}

We will also need the quantities $\Lambda_\parallel, \Lambda_\perp^2 :TM\setminus\bz \to \RR$ defined by

$$\Lambda_\parallel : = \frac{g(\rF\Lambda,\Sigma)}{g(\Sigma,\Sigma)} \qquad \text{ and } \qquad \Lambda_\perp^2 : = \frac{g(\rF\Lambda_\perp,\rF\Lambda_\perp)}{g(\Sigma,\Sigma)}, \qquad \text{ where } \qquad \rF\Lambda_\perp : = \rF\Lambda - \Lambda_\parallel\Sigma.$$
In coordinates,
\begin{equation}\label{Lambdacomponentsdef}\Lambda_\parallel = \frac{g_{ij}v^i\Lambda^j}{g_{ij}v^iv^j} \qquad \text{ and } \qquad \Lambda_\perp^2 = \frac{g_{ij}(\Lambda^i - \Lambda_\parallel v^i)(\Lambda^j - \Lambda_\parallel v^j)}{g_{ij}v^iv^j}.\end{equation}

\medskip
Recall that the gradient of a $C^1$ function $u$ is given by $\nabla u = \cL du$. In coordinates,
$$L_{v^i}(\nabla u) = (du)_i = \frac{\partial u}{\partial x^i}.$$
Differentiation with respect to $x^j$ gives
$$L_{v^ix^j}(\nabla u) + L_{v^iv^k}(\nabla u)\partial_{x^j}(\nabla u)^k = \frac{\partial^2u}{\partial x^i\partial x^j},$$
and after rearrangement,
\begin{equation}\label{partialxinablaueq}\partial_{x^i}(\nabla u)^j = L^{v^jv^k}(\nabla u)\left(\partial^2_{x^ix^k}u - L_{v^kx^i}(\nabla u)\right).\end{equation}

\medskip
Let $\nabla^2u$ be the $(1,1)$ form on $M$ given by 
$$(\nabla^2 u)(X) = \rD_X\nabla u$$
for every vector field $X$. In coordinates:
\begin{equation}\label{nabla2ucooreq}\nabla^2u = (\nabla^2u)_i^j\,dx^i\otimes \partial_{x^j} \qquad \text{ where } \qquad (\nabla^2u)_i^j = \partial_{x^i}(\nabla u)^j + \Gamma_i^j\circ\nabla u.\end{equation}
Denote also
$$\Delta u : = \tr(\nabla^2u).$$

From \eqref{LagrangianGammaeq}, \eqref{gdef} and \eqref{partialxinablaueq} it follows that
$$
    (g_{ik}\circ \nabla u)(\nabla^2u)_j^k = \partial^2_{x^ix^j}u + \frac12 \left(\Sigma L_{v^iv^j} - L_{v^jx^i} - L_{v^ix^j}\right)\circ(\nabla u).
$$
In particular, since the right hand side is symmetric, the operator $\nabla^2u$ is self-adjoint with respect to the Riemannian metric 
$$(\nabla u)^*g = (g_{ij}\circ\nabla u)dx^idx^j.$$ 

\begin{proposition}[The Bochner-Weitzenb{\"o}ck formula, cf. \text{\cite[Theorem 4.4]{Oh14}}]\label{Bochnerprop}
    Let $U\subseteq M$ be an open set and let $u:U\to \RR$ be a solution to the Hamilton-Jacobi equation \eqref{HJeq}. Then:
    \begin{align}\label{appBochnereq}
        (d\Delta u)(\nabla u) + |\nabla^2u|^2 + \Ric(\nabla u)=0,
    \end{align}
    where $|\cdot|$ is the Frobenius norm with respect to the Riemannian metric $(\nabla u)^*g$.    
\end{proposition}

\begin{lemma}
    For every vector field $Y$ on $U$, we have
    \begin{equation}\label{downarrowpiveq}\downarrow(\pi_\cV[\uparrow Y,\Sigma])\vert_{\Graph(\nabla u)} = - (\nabla^2u)Y + [Y,\nabla u].\end{equation}
\end{lemma}
\begin{proof}
    By the definitions of $\uparrow,\Gamma_i^j, E_j$ and $\pi_\cV$,
    \begin{align*}
        \pi_\cV[\uparrow Y,\Sigma] & = \pi_\cV[Y^j\partial_{v^j},v^i\partial_{x^i} + \Sigma^i\partial_{v^i}] \\
        & = \pi_\cV(Y^j\partial_{x^j} + Y^j\partial_{v^j}\Sigma^i\partial_{v^i} - v^i\partial_{x^i}Y^j\partial_{v^j}) \\
        & = \pi_\cV(Y^j\partial_{x^j} - 2Y^j\Gamma_j^i\partial_{v^i}- v^i\partial_{x^i}Y^j\partial_{v^j}) \\
        & =\pi_\cV(Y^jE_j - Y^j\Gamma_j^i\partial_{v^i}- v^i\partial_{x^i}Y^j\partial_{v^j}) \\
        & = - Y^i\Gamma_i^j\partial_{v^j}- v^i\partial_{x^i}Y^j\partial_{v^j}.
    \end{align*}
    Thus, by the definition of $\downarrow$,
    \begin{align*}
        \downarrow(\pi_\cV[\uparrow Y,\Sigma])\vert_{\Graph(\nabla u)} & = \downarrow(- Y^i\Gamma_i^j\partial_{v^j}- v^i\partial_{x^i}Y^j\partial_{v^j})\vert_{\Graph(\nabla u)} \\
        & = (- Y^i\Gamma_i^j\circ\nabla u- (\nabla u)^i\partial_{x^i}Y^j)\partial_{x^j}\\
        & = (- Y^i\Gamma_i^j\circ\nabla u- Y^i\partial_{x^i}(\nabla u)^j)\partial_{x^j} + [Y,\nabla u]\\
        & = - (\nabla^2u)Y + [Y,\nabla u],
    \end{align*}
    and \eqref{downarrowpiveq} is proved.
\end{proof}

\begin{proof}[Proof of Proposition \ref{Bochnerprop}]
    We may assume that $U$ is contained in a coordinate domain. By Lemma \ref{HJlemma}, the integral curves of $\nabla u$ are zero-energy extremals, so their canonical lifts are integral curves of $\Sigma$ contained in $SM$. Equivalently,
    $$\Graph(\nabla u)\subseteq SM \qquad \text{ and } \qquad (\nabla u)_*(\nabla u) = \Sigma\vert_{\Graph(\nabla u)}.$$
    Let $X$ be a vector field on $U$. Let $\xi$ be any vector field on $TU$ satisfying
    $$\xi\vert_{\Graph(\nabla u)} = (\nabla u)_*X.$$
    By the definitions of $\rD_X\nabla u $ and $\uparrow$, the vector fields $\pi_\cV\xi$ and $\uparrow \, \rD_X\nabla u$ coincide on $\Graph(\nabla u)$; since $\Sigma$ is tangent to the graph, it follows that they have the same Lie bracket with $\Sigma$. Thus, on $\Graph(\nabla u)$ we have
    \begin{align*}
        \pi_\cV[\pi_\cH\xi,\Sigma] & = \pi_\cV[\xi,\Sigma] - \pi_\cV[\pi_\cV\xi,\Sigma]\\
        & = \pi_\cV[(\nabla u)_*X,(\nabla u)_*(\nabla u)] - \pi_\cV[\uparrow\rD_X\nabla u,\Sigma]\\
        & = \pi_\cV(\nabla u)_*[X,\nabla u] - \pi_\cV[\uparrow(\nabla^2u)X,\Sigma].
    \end{align*}
    Applying $\downarrow$ to both sides and using \eqref{downarrowpiveq} with $Y = (\nabla^2u)X$, we get
    \begin{align*}
        \downarrow\left(\pi_\cV[\pi_\cH\xi,\Sigma]\vert_{\Graph(\nabla u)}\right) & = \rD_{[X,\nabla u]}\nabla u + (\nabla^2u)^2X - [(\nabla^2u)X,\nabla u]\\
        & = (\nabla^2 u)[X,\nabla u] - [(\nabla^2u)X,\nabla u] + (\nabla^2u)^2X\\
        & = -[\nabla^2u,\nabla u]X + (\nabla^2u)^2X.
    \end{align*}
    Now set $X = \partial_{x^i}$ and sum over $i$; since 
    $$\pi_\cH(\nabla u)_*\partial_{x^i} = E_i \qquad \text{ and } \qquad dx^i\circ\downarrow \circ \,\pi_\cV= \eps^i,$$
    formula \eqref{Riccicoordeq} implies that the left hand side becomes $-\Ric(\nabla u)$, and we get
    \begin{align*}
        -\Ric(\nabla u) &= -\tr([\nabla^2u,\nabla u]) + \tr((\nabla^2u)^2)\\
        & = (d\Delta u)(\nabla u) + |\nabla^2u|^2,
    \end{align*}
    as desired (the identity $\tr([A,X]) = -X\tr A = -(d\Tr A)(X)$ is easy to verify in coordinates).
\end{proof}
\begin{remark*}
    If we do not take traces in the above proof, then we obtain a \emph{Riccati equation} for the $(1,1)$-tensor $\nabla^2u$. Write
    \begin{equation*}
        R_{\nabla u} : = \left[\eps^j\left(\pi_\cV[\Sigma,E_i]\right)\circ \nabla u\right]\,dx^i\otimes\partial_{x^j}.
    \end{equation*}
    Then the $(1,1)$ tensor $\nabla^2u$ satisfies
    \begin{equation*}
        [\nabla u,\nabla^2u] + (\nabla^2u)^2 + R_{\nabla u} = 0,
    \end{equation*}
    cf. \text{\cite[Section 4]{Lee}}.
\end{remark*}

\subsection{Weighted Ricci curvature}\label{weightedsec}
Recall (see e.g. \cite{LeeSM}) that a \emph{positive density} on $M$ is a smooth choice of a function
$$\omega_x: (T_xM)^n \to [0,\infty), \qquad \, x \in M,$$
with the property that for every linear transformation $T : T_xM \to T_xM$,
$$\omega_x(T(v_1),\dots,T(v_n)) = |\det T|\cdot\omega_x(v_1,\dots,v_n), \qquad v_1,\dots,v_n \in T_xM.$$
If $M$ is oriented, then specifying a positive density amounts to specifying a volume form. We shall say that a measure $\mu$ on $M$ has a smooth density if there exists a positive density $\omega$ on $M$, which we will call the \emph{density} of $\mu$, such that for every open set $U \subseteq M$,
$$\mu(U) = \int_U \omega.$$
From now on, we fix a smooth manifold $M$, a Lagrangian $L:TM\to \RR$ satisfying assumptions (I)-(III) from Section \ref{presec}, and a measure $\mu$ with a smooth density $\omega$. We denote by $H:T^*M\to \RR$ the corresponding Hamiltonian.

\medskip
Define an operator $\bL$ on $C^2$ functions $u:M \to \RR$ by:
$$\bL  u : = \div_\mu({\nabla}u),$$
where for a vector field $X$, the function $\div_\mu X$ is defined by the relation
$$\sL_X\omega = (\div_\mu X)\omega,$$
where $\mathscr L$ denotes Lie derivative and $\omega$ is the density of the measure $\mu$. If the function $u$ is $C^{1,1}$ then $\bL u$ still makes sense as an element of $L^\infty_{\mathrm{loc}}$, and is well-defined at every point of differentiability of $\nabla u$.

\medskip 
The operator $\nabla$, and therefore also the operator $\bL$, are not linear in general. Furthermore, since $d\phi(\nabla u)$ and $du(\nabla\phi)$ may differ, it is \emph{not} true in general that $\int (\bL u)\cdot\phi d\mu = \int u \cdot (\bL \phi)d\mu$. We mention in passing that the equation $\bL u = 0$ is the Euler-Lagrange equation for the functional $\cE(u) = \int_MH(du)d\mu$.

\medskip
There exists a function $\psi: TM\setminus\bz \to \RR$ such that, in any canonical local chart $(x^i,v^i)$ on $TM$,
$$\pi^*\omega = e^{-\psi}\,\sqrt{\det g}\,|dx^1\wedge \dots \wedge dx^n|.$$
Indeed, it is not hard to check that the right hand side is well defined, i.e. does not depend on the choice of canonical coordinates; that such $\psi$ exists then follows from the fact that both sides vanish on an $n$-tuple of vectors if one of the vectors is vertical (objects with this property are sometimes called \emph{semi-basic}). 

\medskip
\begin{definition}[Weighted Ricci curvature]\label{weightedriccidef}\normalfont
    For every $N \in (-\infty,\infty]\setminus[1,n]$, we define the \emph{weighted Ricci curvature} $\Ric_{\mu,N}$ by
    $$\Ric_{\mu,N}:TM\setminus\bz \to \RR, \qquad \Ric_{\mu,N} : = \Ric_\mu - \frac{(\Sigma\psi - \Lambda_\parallel)^2}{N-n} + 2\Lambda_\perp^2+\Lambda_\parallel^2,$$
    where
    $$\Ric_\mu : = \Ric + \Sigma^2(\psi),$$    
    and $\Lambda_\parallel$ and $\Lambda_\perp^2$ are defined in \eqref{Lambdacomponentsdef}. Here, by $\Sigma^2(\psi)$ we mean $\Sigma(\Sigma(\psi))$, i.e. the second derivative of the function $\psi$ by the vector field $\Sigma$. If $\Sigma\psi - \Lambda_\parallel = 0$ then we also define $$\Ric_{\mu,n} := \Ric_\mu + 2\Lambda_\perp^2+\Lambda_\parallel^2.$$
\end{definition}
\begin{remark*}
    Note that $\Ric_\mu \ne \Ric_{\mu,\infty}$ unless $\Lambda = 0$.
\end{remark*} 
\begin{proposition}[Weighted Bochner-Weitzenb{\"o}ck formula]\label{Weightedbochnerprop}
    Let $U\subseteq M$ be an open set and let $u:U\to \RR$ be a solution to the Hamilton-Jacobi equation \eqref{HJeq}. Then:
    \begin{align}\label{weightedBochnereq}
        (d\bL u)(\nabla u) + |\nabla^2u|^2 + \Ric_{\mu}(\nabla u) & = 0,
    \end{align}
    where $|\cdot|$ is the Frobenius norm with respect to the Riemannian metric $(\nabla u)^*g$. Moreover, for every $N \in (-\infty,\infty]\setminus[1,n)$,
    \begin{equation}\label{weightedBochnerineq}
        (d\bL u)(\nabla u) + \frac{(\bL u)^2}{N-1} + \Ric_{\mu,N}(\nabla u) \le 0.
    \end{equation}
\end{proposition}

\begin{proof}
Since $u$ solves the Hamilton-Jacobi equation, 
\begin{equation}\label{nablaustarnablaueq}
    (\nabla u)_*(\nabla u) = \Sigma\vert_{\Graph(\nabla u)}.
\end{equation}
By the definition of $\psi$, the measure $\mu$ is given by 
$$\mu = e^{-\psi\circ\nabla u}\,\Vol_{(\nabla u)^*g}.$$
Thus, using \eqref{trGammaeq}, \eqref{nabla2ucooreq} and \eqref{nablaustarnablaueq},
\begin{align*}
    \bL u & = \div_\mu(\nabla u)\\  
    & = \partial_{x^j}(\nabla u)^j + (\nabla u)(\log\sqrt{\det (\nabla u)^*g}-\psi\circ\nabla u)\\
    & = \partial_{x^j}(\nabla u)^j + \Gamma_j^j\circ\nabla u - (\Sigma\psi)\circ\nabla u\\
    & = (\nabla^2u)_j^j - (\Sigma\psi)\circ\nabla u\\
    & = \Delta u - (\Sigma\psi)\circ\nabla u,
\end{align*}
and therefore
$$
    (d\bL u)(\nabla u) = (d\Delta u)(\nabla u) - (\Sigma^2\psi)\circ\nabla u.
$$
Together with \eqref{appBochnereq} and the definition of $\Ric_\mu$ this implies \eqref{weightedBochnereq}.

\medskip
Let us abbreviate, by a slight abuse of notation,
$$\Lambda : = \,\downarrow\,\Lambda\circ\nabla u, \qquad \Lambda_\parallel : = \Lambda_\parallel\circ\nabla u, \qquad \Lambda_\perp^2 = \Lambda_\perp^2\circ\nabla u.$$
The definitions of $\Lambda_\parallel$ and $\Lambda_\perp^2$ imply that
\begin{equation*}
    \Lambda_\parallel = \frac{\left<\Lambda,\nabla u\right>_{(\nabla u)^*g}}{|\nabla u|^2_{(\nabla u)^*g}} \qquad \text{ and } \qquad \Lambda_\perp^2 = \frac{|\Lambda - \Lambda_\parallel\nabla u|_{(\nabla u)^*g}^2}{|\nabla u|^2_{(\nabla u)^*g}}.
\end{equation*}
In the sequel, norms and inner products are with resepct to $(\nabla u)^*g$. By \eqref{Lambdaeq},
$$(\nabla^2u)\nabla u = \rD_{\nabla u}\nabla u = \Lambda.$$
We can therefore write
$$\nabla ^2u = \nabla^2_\parallel u + \nabla^2_\perp u,$$
where $\nabla^2_\perp u$ is self-adjoint and annihilates $\nabla u$, and
    \begin{align*}\nabla^2_\parallel u 
    & = \frac{(\nabla u)^\flat \otimes(\Lambda - \Lambda_\parallel \nabla u) + (\Lambda - \Lambda_\parallel\nabla u)^\flat\otimes \nabla u}{\left<\nabla u ,\nabla u\right>} + \Lambda_\parallel\cdot\frac{(\nabla u)^\flat\otimes \nabla u}{\left<\nabla u ,\nabla u\right>}.
    \end{align*}
Here $\flat$ is the musical isomorphism $v^\flat = \left<v,\cdot\right>$. Note that the summands in the definition of $\nabla^2_\parallel u$ are mutually orthogonal in the Frobenius inner product, since $\Lambda_\parallel \nabla u$ is the component of $\Lambda$ in the direction of $\nabla u$. Hence by the definition of $\Lambda_\perp^2$ and $\Lambda_\parallel$,
$$|\nabla_\parallel^2u|^2 = 2\Lambda_\perp^2 + \Lambda_\parallel^2.$$
By the Cauchy-Schwarz inequality,
\begin{equation}\label{HessuCSeq}
    \left|\nabla^2_\perp u\right|^2 \ge \frac{\tr(\nabla^2_\perp u)^2}{n-1} = \frac{(\Delta u - \Lambda_\parallel)^2}{n-1},
\end{equation}
where the second equality holds true since $\left<(\nabla^2u)\nabla u,\nabla u\right>/\left<\nabla u,\nabla u\right> = \Lambda_\parallel$. Thus
\begin{align}\label{Hessuboundeq}
    |\nabla^2u|^2 \ge  \frac{(\Delta u - \Lambda_\parallel)^2}{n-1} + |\nabla_\parallel^2u|^2 = \frac{(\Delta u - \Lambda_\parallel)^2}{n-1} + 2\Lambda_\perp^2 + \Lambda_\parallel^2.
\end{align} 
As in \cite[Lemma 3.21]{Kl}, we now use the inequality
	$$\frac{x^2}{a} + \frac{y^2}{b} \ge \frac{(x-y)^2}{a + b}$$
	which holds for $b>0$ and $a\notin[-b,0]$, to obtain
	\begin{equation}\label{squaretrickeq}
		\frac{(\Sigma\psi - \Lambda_\parallel)^2}{N - n} + \frac{(\Delta u - \Lambda_\parallel)^2}{n-1} \ge \frac{(\Sigma\psi - \Delta u)^2}{N - 1} = \frac{(\bL u)^2}{N-1}
	\end{equation}
	since $\bL u = \Delta u - (\Sigma\psi)\circ\nabla u$.  Note that this inequality still holds true (and in fact becomes an equality) if $N = n$ and $\Sigma\psi - \Lambda_\parallel = 0$ and we understand the first term on the left hand side to be zero. Putting together \eqref{weightedBochnereq}, \eqref{Hessuboundeq} and \eqref{squaretrickeq} we get
    \begin{align*}
        (d\bL u)(\nabla u) + \Ric_{\mu,N}(\nabla u) & = (d\bL u)(\nabla u) + \Ric_{\mu}(\nabla u) - \frac{(\Sigma\psi - \Lambda_\parallel)^2}{N-n} + 2\Lambda_\perp^2+\Lambda_\parallel^2\\
        & \le -|\nabla^2u|^2  - \frac{(\Sigma\psi - \Lambda_\parallel)^2}{N-n}+ 2\Lambda_\perp^2+\Lambda_\parallel^2\\
        & \le -\frac{(\Delta u - \Lambda_\parallel)^2}{n-1} + \frac{(\Delta u - \Lambda_\parallel)^2}{n-1} - \frac{(\bL u)^2}{N-1}\\
        & = -\frac{(\bL u)^2}{N-1},
    \end{align*}
    as desired.
\end{proof}
From the proof of Proposition \ref{Weightedbochnerprop} we obtain also the equality case:
\begin{lemma}\label{equalitylemma}
    For every $x\in M$ and every $v \in S_xM$, there exists a solution $u$ to the Hamilton-Jacobi equation defined in neighborhood of $x$ such that $\nabla u\vert_x = v$ and equality holds in \eqref{weightedBochnerineq} at the point $x$.
\end{lemma}
\begin{proof}
    Equality holds in \eqref{HessuCSeq} if and only if $\nabla^2_\perp u$ is a scalar multiple of $\pi_{(\nabla u)^\perp}$, where $\pi_{(\nabla u)^\perp}$ denotes projection onto the orthogonal complement of $\nabla u$, while equality holds in \eqref{squaretrickeq} if and only if 
    \begin{equation}\label{Deltauequality}
        (N-n)\Delta u + (n-1)(\Sigma \psi)\circ\nabla u - (N-1)\Lambda_\parallel= 0.
    \end{equation}
    By Lemma \ref{characlemma}, for any $c_0 \in \RR$ we can find a local solution $u$ to the Hamilton-Jacobi equation satisfying $\nabla u\vert_x = v$ and $(\nabla ^2 u)^\perp = c_0\cdot\pi_{(\nabla u)^\perp}$. Taking
    $$c_0 : = \frac{\Lambda_\parallel - \Sigma\psi}{N-n}$$
    guarantees that \eqref{Deltauequality} holds (if $N = n$ and $\Sigma\psi = \Lambda_\parallel$ then any choice of $c_0$ works).
\end{proof}
\subsection{The curvature-dimension condition}
From Proposition \ref{Weightedbochnerprop} and Lemma \ref{equalitylemma} we obtain the equivalence (i)$\iff$(ii) in Theorem \ref{mainthm}:
\begin{corollary}\label{equivalencecor}
    Let $L$ be a Lagrangian satisfying assumptions (I)-(III) from Section \ref{presec} and smooth on $TM \setminus \bz$. For every $N \in (-\infty,\infty]\setminus[1,n)$ and ${K} \in \RR$, the following are equivalent:
    \begin{enumerate}
        \item $\Ric_{\mu,N} \ge {K}$ on $SM$.
        \item For every local solution $u$ to the Hamilton-Jacobi equation $H(du) = 0$, 
        \begin{equation*}(d\bL u)(\nabla u) + \frac{(\bL u)^2}{N-1} + {K} \le 0.\end{equation*}
    \end{enumerate}
\end{corollary}

\begin{definition}[The curvature-dimension condition]\label{CDdef}\normalfont
    If any of the equivalent conditions in Corollary \ref{equivalencecor} holds, then we say that the pair $(\mu,L)$ satisfies $\CD({K},N)$. 
\end{definition}
The second condition in Corollary \ref{equivalencecor} still makes sense if $L$ is only $C^2$, since by Lemma \ref{characlemma} we can always find a local $C^3$ solution with prescribed second-order data at a point. Thus we can take this condition to be the definition of the $\CD(K,N)$ condition for $C^2$ Lagrangians. We will show in Theorem \ref{DCthm} that the equivalence (ii)$\iff$(iii) in Theorem \ref{mainthm} requires only $C^2$ regularity of $L$. The third condition in Theorem \ref{mainthm} can thus be taken as the definition of the $\CD(K,N)$ condition for Lagrangians of regularity lower than $C^2$.

\begin{example}[Weighted Riemannian manifolds]\normalfont
    Let $(M,g)$ be an $n$-dimensional Riemannian manifold, let $L$ be the Riemannian Lagrangian
    $$L(v) = \frac{|v|_g^2 + 1}{2}, \qquad v \in T_xM, \, x \in M,$$
    and let
    $$\mu = e^{-\psi}\Vol_g,$$ where $\psi : M \to \RR$ is a smooth function and $\Vol_g$ is the Riemannian measure. The gradient is the Riemannian gradient and the operator $\bL$ is given by
    $$\bL u = \Delta_gu - \left<\nabla^g\psi,\nabla^gu\right>,$$
    where $\Delta_g$ is the Riemannian Laplacian and $\nabla^g$ is the Riemannian gradient.  The Hamilton-Jacobi equation is $|\nabla u|_g \equiv 1$. 
    The $\CD_{}({K},N)$ condition for is then equivalent to the inequality
    $$\Ric_g + \Hess_g\psi - \frac{(d\psi)^2}{N-n} \ge {K}\cdot g,$$
    see e.g. \cite[Chapter 14, Theorem 14.8]{Vil}.
\end{example}
\begin{example}[Weighted Finsler manifolds]\normalfont
    If $L = (F^2+1)/2$, where $F$ is a Finsler metric, then the $\CD_{}({K},N)$ condition is equivalent to the condition $\Ric_{\mu,N} \ge {K}$, where $\Ric_{\mu,N}$ is the weighted Ricci curvature of the Finsler metric $F$ and the measure $\mu$, see \cite[Part II]{OhBook}.
\end{example}
\begin{example}[Needles]\label{needleexample}\normalfont
    Let $M = I \subseteq \RR$ be an interval, let $l$ be the Lagrangian $l = (dt^2+1)/2$ (where $dt^2$ is the Euclidean metric on $I$), and let $m = e^{-\psi}m_0$, where $m_0$ is the Lebesgue measure on $I$ and $\psi : I \to \RR$ is a smooth function. Then the pair $(m,l)$ satisfies $\CD({K},N)$ if and only if
    $$\ddot\psi - \frac{\dot\psi^2}{N-1} \ge {K}.$$
\end{example}
More examples will be discussed in detail in Section \ref{examplesec}.

\section{Dominated functions}\label{domsec}

From this point on, we fix a smooth manifold $M$, a Lagrangian $L$ satisfying assumptions (I)-(III) from Section \ref{assumpsec}, and a measure $\mu$ with a smooth density.

\medskip
In this chapter we will establish some properties of \emph{dominated functions}. In the language of optimal transport, dominated functions are precisely \emph{$\cc$-convex functions}, see e.g. \cite[Chapter 5]{Vil}; from the analytic perspective, they are viscosity subsolutions to the Hamilton-Jacobi equation, see \cite[Section 5]{FM}. In the Riemannian case, dominated functions are simply 1-Lipschitz functions. The results of the present chapter will be used in the next chapter to construct needle decompositions and displacement interpolations.

\begin{definition}\normalfont
    A function $u : M \to \RR$ is said to be \emph{dominated} if for every piecewise-$C^1$ curve $\gamma:[a,b] \to M$,
    \begin{equation*}u(\gamma(b)) - u(\gamma(a)) \le \int_a^bL(\dot\gamma(t))dt.\end{equation*}
    Equivalently, 
    $$u(x') - u(x) \le \cc(x,x') \qquad \text{ for all } x,x' \in M.$$
\end{definition}

 \begin{definition}[Calibrated curves]\normalfont
    A \emph{calibrated curve} is a piecewise-$C^1$ curve $\gamma : I \to M$, where $I \subseteq \RR$ is an interval, which satisfies
     $$u(\gamma(t')) - u(\gamma(t)) = \int_t^{t'} L(\dot\gamma(s))ds \qquad \text{ for all }t < t', \quad t,t' \in I.$$
 \end{definition}

\begin{remark*}\normalfont
    The notion of a dominated function depends on the Lagrangian $L$, and the notion of a calibrated curve depends on both $L$ and the dominated function $u$. Since it will always be clear from the context which Lagrangian and which dominated function we are considering, we will not indicate this dependence in the notation.
\end{remark*}

 Let us state a few elementary properties of dominated functions and calibrated curves, whose simple proofs can be found in \cite{Cont,CI,FM}.

\begin{proposition}\label{dominatedprop}
    Let $u : M \to \RR$ be a dominated function and let $\gamma:(a,b)\to M$ be a calibrated curve. Then
    \begin{enumerate}[(a)]
        \item The function $u$ is locally Lipschitz.
        \item The curve $\gamma$ is a minimizing extremal.
        \item For all $ t < t' \in I$, 
        \begin{equation*}
            u(\gamma(t')) - u(\gamma(t)) = \cc(\gamma(t),\gamma(t')).
        \end{equation*}
        \item The function $u$ is differentiable at $\gamma(t)$ for all $t \in (a,b)$ and satisfies $$\dot\gamma(t) = {\nabla}u\vert_{\gamma(t)} \qquad \text{ and } \qquad H(du\vert_{\gamma(t)}) = 0.$$
        Together with (b), this implies that for all $a < t < t' < b$,
        \begin{equation*}
            du\vert_{\gamma(t')} = \Phi^H_{t'-t}\, du\vert_{\gamma(t)}. 
        \end{equation*}
        \item If two calibrated curves contain a common point in their relative interiors, then they are both part of a single calibrated curve.   
    \end{enumerate}
\end{proposition}

Here and in the sequel, we say that a function between two smooth manifolds is \emph{locally Lipschitz} if it is locally Lipschitz when expressed in local coordinates about any given point. This notion is independent of the coordinates chosen.

\subsection{Transport rays}

For the rest of the chapter we fix a dominated function $u:M \to \RR$.

\begin{definition}[Transport rays]\normalfont
    A set ${\alpha} \subseteq M$ will be called a \emph{transport ray} if
    \begin{enumerate}[(i)]
        \item for every $x,x' \in {\alpha}$,
        \begin{equation}\label{trcond}\text{either} \quad u(x') - u(x) = \cc(x,x') \quad \text{ or } \quad u(x) - u(x') = \cc(x',x).\end{equation}
        \item ${\alpha}$ is maximal with respect to inclusion, i.e. there does not exist a set with the above property which strictly contains the set  ${\alpha}$.
    \end{enumerate} 
     A transport ray can be a singleton; in this case we say that it is \emph{degenerate}; otherwise it will be called \emph{nondegenerate}. Note that for every $x,y \in M$ satisfying \eqref{trcond}  there exists a transport ray containing $x$ and $y$ (and in particular, every point in $M$ is contained in some transport ray). Indeed, if we write $x \sim x'$ when \eqref{trcond} holds, then the triangle inequality and the fact that the function $u$ is dominated imply that $\sim$ is an equivalence relation; the equivalence classes of $\sim$ are then precisely the transport rays.
\end{definition}

\begin{lemma}\label{transportraycalibratedlemma}
    Every transport ray is the image of a unique calibrated curve. Conversely, the image of a calibrated curve with a maximal domain of definition is a transport ray.
\end{lemma}

\begin{proof}
    Let ${\alpha} \subseteq M$ be a transport ray. We may assume that ${\alpha}$ is not a singleton, for otherwise we can take the calibrated curve $\gamma$ to be constant. Let $x \ne y \in {\alpha}$. Then without loss of generality, \begin{equation}\label{transportrayeq1}u(y) - u(x) = \cc(x,y).\end{equation}
      Let $\gamma$ be a minimizing extremal satisfying $\gamma(0) = x$ and $\gamma(T) = y$ for some $T > 0$, and let $t \in [0,T]$. Then, since $u$ is dominated, the quantities
    $$\cc(x,\gamma(t)) - [u(\gamma(t))-u(x)] \qquad \text{ and } \qquad \cc(\gamma(t),y) - [u(y) - u(\gamma(t))]$$
    are both nonnegative. But by Lemma \ref{triangleineqlemma} and \eqref{transportrayeq1} they sum to zero:
    $$\cc(x,\gamma(t)) + \cc(\gamma(t),y) - [u(y) - u(x)] = \cc(x,y) - [u(y) -u(x)]=0.$$
    Thus both summands vanish, and in particular
    $$\cc(\gamma(t),y) = u(y) - u(\gamma(t)).$$
    Let $t' \in (t,T)$. By the same argument as above, 
    $$\cc(\gamma(t),\gamma(t')) = u(\gamma(t')) - u(\gamma(t)).$$
    We have shown that for every $0 < t < t' <T$,
    $$u(\gamma(t')) - u(\gamma(t)) = \cc(\gamma(t),\gamma(t')) = \int_t^{t'}L(\dot\gamma(t))dt,$$
    whence $\gamma\vert_{[0,T]}$ is a calibrated curve (the second equality holds since $\gamma$ is a minimizing extremal). We will show that the set $\alpha$ is contained in the image of $\gamma$. Here we take $\gamma$ to be defined on a maximal interval on which it is a minimizing extremal; this interval may be larger than $[0,T]$. 

    \medskip
    Let $z \in {\alpha}$. By the definition of a transport ray, we have
    $$u(z) - u(y) \in \{\cc(y,z),-\cc(z,y)\} \qquad \text{ and } \qquad u(x) - u(z) \in \{\cc(z,x),-\cc(x,z)\}.$$
    It cannot be the case that both $u(z) - u(y) = \cc(y,z)$ and $u(x)-u(z) = \cc(z,x)$, because then 
    $$\cc(x,y) + \cc(y,z) + \cc(z,x) = 0,$$
    a contradiction to the assumption of supercriticality. An analysis of the remaining cases shows that there exists a relabeling $x',y',z'$ of the points $x,y,z$ such that
    $$\cc(x',y') + \cc(y',z') = \cc(x',z').$$
    It then follows from Lemma \ref{triangleineqlemma} that the points $x,y,z$ lie on a single minimizing extremal, i.e. $z = \gamma(t)$ for some $t \in \RR$.
    
    \medskip
    Since $z$ is an arbitrarily element of ${\alpha}$, the above arguments show that there exists an interval $I \subseteq \RR$ such that
    \begin{itemize}
        \item The set ${\alpha}$ is contained in $\gamma(I)$,
        \item The curve $\gamma\vert_I$ is calibrated.
    \end{itemize} 
    
    \medskip
    In fact, we may take the interval $I$ to be the convex hull of $\gamma^{-1}({\alpha})$. We claim that ${\alpha} = \gamma(I)$. Indeed, let $t \in I$. By maximality of ${\alpha}$, in order to prove that $\gamma(t) \in {\alpha}$ it suffices to prove that for every $x \in {\alpha}$, 
    \begin{equation}\label{transportcalibneedeq}\text{either $\qquad \cc(x,\gamma(t)) = u(\gamma(t)) - u(x)\qquad $ or $\qquad\cc(\gamma(t),x) = u(x) - u(\gamma(t))$.}\end{equation}
    Let $x \in {\alpha}$. Then $$x = \gamma(t')$$ for some $t' \in \RR$. Assume that $t' < t$ (the other case is analogous). Since $t$ lies in the convex hull of $\gamma^{-1}({\alpha})$, there exists also $t'' \ge t$ such that $\gamma(t'') \in {\alpha}$. By the definition of $\alpha$,
    $$\text{either $\quad \cc(\gamma(t'),\gamma(t'')) = u(\gamma(t'')) - u(\gamma(t'))\quad $ or $\quad\cc(\gamma(t''),\gamma(t')) = u(\gamma(t')) - u(\gamma(t''))$.}$$
    But $t''>t'$ and $\gamma\vert_I$ is calibrated, so the former holds, whence
    $$u(\gamma(t'')) - u(\gamma(t')) = \cc(\gamma(t'),\gamma(t'')) = \cc(\gamma(t'),\gamma(t)) + \cc(\gamma(t),\gamma(t'')).$$
    Since the function $u$ is dominated, this necessarily implies that
    $$\cc(\gamma(t'),\gamma(t)) = u(\gamma(t)) - u(\gamma(t')).$$
    This proves \eqref{transportcalibneedeq}, thus completing the proof that ${\alpha} = \gamma(I)$.

    \medskip
    In the other direction, if $\gamma : I \to M$ is a calibrated curve defined on a maximal interval, then its image $\gamma(I)$ clearly satisfies \eqref{trcond}. If $\gamma(I)$ it is contained in a larger set with this property, then the previous argument shows that the interval of definition of $\gamma$ can be extended further, in contradiction to the assumption.
\end{proof}

\begin{definition}[Upper transport rays, initial and final points]\normalfont
    An \emph{upper transport ray} is a set of the form 
    $$\gamma(I\cap(a,\infty)),$$
    where $\gamma : I \to M$ is a non-constant, maximal calibrated curve. When $a = \inf I > -\infty$, the point $\gamma(a)$ will be called the \emph{initial point} of the transport ray $\gamma(I)$. \emph{Final points} of transport rays are defined analogously, with $(a,\infty)$ replaced by $(-\infty,b)$.
\end{definition}

\begin{definition}[Saturation and transport sets]\normalfont
    Let $A \subseteq M$.
    \begin{itemize}
        \item The \emph{saturation} of $A$, denoted by $\hat A$, is defined to be the union of all transport rays intersecting $A$. If the set $A$ is Borel measurable set satisfies $\hat A = A$ then we say that $A$ is a \emph{transport set}.
        \item The \emph{upper saturation} of $A$, denoted by $A^+$, is defined to be the union of all upper transport rays whose initial point lies in $A$. A Borel measurable set $A$ satisfying $A^+ = A$ will be called an \emph{upper transport set}.
    \end{itemize}
    See Figure \ref{satfig}.
\end{definition}

\begin{figure}[t]
    \includegraphics[width=\textwidth]{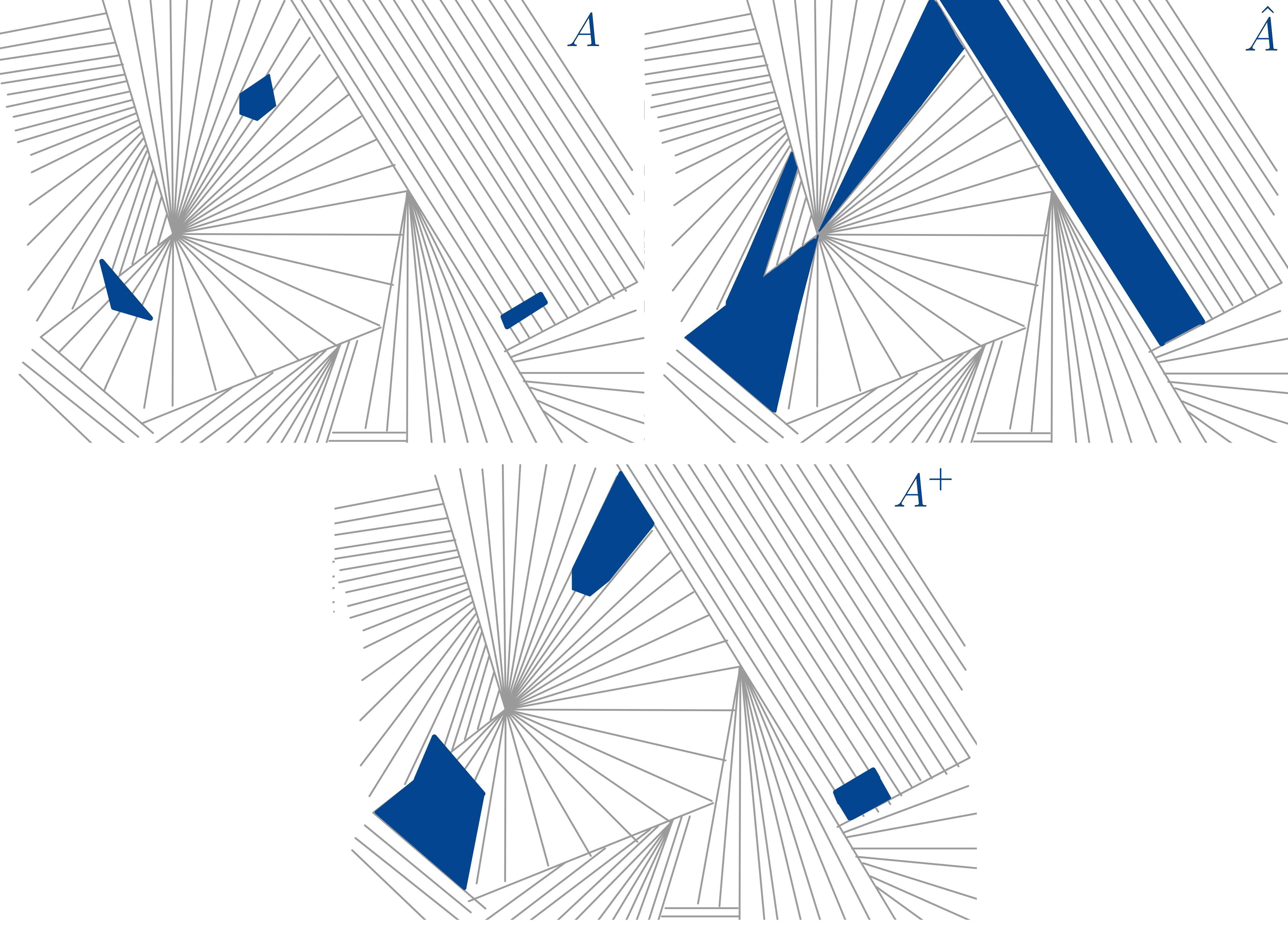}
    \caption{A set and its saturation and upper saturation.}
    \label{satfig}
\end{figure}

\begin{lemma}\label{endslemma}
    The sets of all initial and final points of transport rays have zero measure. 
\end{lemma}
\begin{proof}
    This a consequence of Lemma \ref{measurableendslemma}, and the proof is the same as in the Riemannian setting. We refer the reader to \cite[Lemma 3.8]{Kl} for more details.
\end{proof}

Let $\cS \subseteq M$ denote the \emph{strain set} of $u$, which is defined to be the set of points which lie in the relative interior of some nondegenerate transport ray. By Lemma \ref{transportraycalibratedlemma},
$$\cS : = \{\gamma(t)\, \mid \, \gamma:[a,b] \to M \, \text{is a calibrated curve and } t \in (a,b)\}.$$

\medskip
For every $x \in \cS$ we denote by 
$$\gamma_x:I_x \to M$$
the parametrization of the transport ray through $x$ as a calibrated curve, where
\begin{equation}\label{Ixeq}I_x :=(a(x),b(x)) \subseteq \RR\end{equation}
is taken to be the interior of the maximal interval of definition of $\gamma_x$, translated so that
$$\gamma_x(0) = x.$$
The image $\gamma_x(I_x)$ is a transport ray containing $x$ by Lemma \ref{transportraycalibratedlemma}. By Proposition \ref{dominatedprop},
 the function $u$ is differentiable at $\gamma_x(t)$ for every $t \in I_x$, with
\begin{equation}\label{dotgammaxeq}\dot\gamma_x(t)  = {\nabla}u\vert_{\gamma_x(t)}, \quad t \in I_x.\end{equation}

\begin{lemma}\label{measurableendslemma}
    The set $\cS$ and the functions $a,b$ are Borel measurable.
\end{lemma}

\begin{proof}
    See \cite[Lemma 10]{BB}, \cite[Lemma 2.10]{Fig} or \cite[Lemma 2.9]{Kl}.
\end{proof}

\subsection{Strain set regularity}

In this section we prove $C^{1,1}$ regularity of the dominated function $u$ inside its strain set $\cS$. This is a crucial ingredient in the proof of the needle decomposition theorem. For each $k \ge 1$ define $\cS_k \subseteq \cS$ by
$$\cS_k : = \{x \in \cS \, \mid \, (-2^{-k},2^{-k})\subseteq I_x\}.$$
Note that
$$\cS = \bigcup_{k=1}^\infty\cS_k.$$

\begin{lemma}
    The set $\cS_k$ is closed for all $k \ge 1$.
\end{lemma}
\begin{proof}
    The set $\cS_k$ can be characterized as the set of all points $x \in M$ such that there exists a calibrated curve $\gamma:(-2^{-k},2^{-k})\to M$ satisfying $\gamma(0) = x$. By continuity of $u$ and $\cc$, the pointwise limit of a sequence of calibrated curves is a calibrated curve. If a sequence $x_i \in \cS_k$ converges to some $x \in M$, then by compactness of the indicatrix, the sequence $\dot\gamma_{x_i}(0)$ has a convergent subsequence, and the corresponding subsequence of the sequence  $\gamma_{x^i}\vert_{(-2^{-k},2^{-k})}$ converges pointwise to a calibrated curve $\gamma$ satisfying $\gamma(0) = x$, whence $x \in \cS_k$.
\end{proof}
A function $\phi:M\to \RR$ will be called $C^{1,1}$ if its differential $d\phi:M \to T^*M$ is locally Lipschitz. 

 \begin{theorem}[cf. \text{\cite{Ber}, \cite[Theorem 2.10]{Kl}}]\label{regularitythm}
	For every $k \ge 1$ there exists a $C^{1,1}$ function $u_k : M \to \RR$ such that
		$$u_k \equiv u \qquad \text{ and } \qquad du_k \equiv du \qquad \text{ on } \cS_k.$$
\end{theorem}

If $u$ is $C^1$ then one can take $u_k = u$, see \cite{Fat03}. In the case where $L$ is Riemannian, Theorem \ref{regularitythm} was proved in \cite{Kl}. Our proof is a combination of  arguments given in \cite{Fat03} and \cite{Kl} : we prove $C^{1,1}$ regularity of $u$ on $\cS_k$ using the argument in \cite{Fat03} and then apply Whitney's extension theorem as was done in \cite{Kl}. For a proof of Theorem \ref{regularitythm} when $L$ is Ma{\~n}{\'e} critical, see \cite{Ber}.

\begin{lemma}[cf. \text{\cite[Theorem 3.2]{Fat03}}]\label{taylorlemma}
	For every $x_0 \in M$ and every  $k \ge 1$ there exists a coordinate neighborhood $U_{x_0} \ni x_0$ and a constant $C_{x_0}$ such that for every $x \in U \cap \cS_k$,
	$$|u(x) + (du\vert_x)(x'-x) - u(x')| \le C_{x_0}\cdot |x'-x|^2 \qquad \text{ for all } x' \in U.$$
	Here we use the coordinate chart to view $x,x'$ as points in $\RR^n$ and $du\vert_x$ as an element in $(\RR^n)^*$, and the norm is the Euclidean norm.
\end{lemma}
\begin{proof}
	Since $u$ is locally Lipschitz and the Legendre transform is continuous, there exists a neighborhood $U_0 \ni x_0$ and a constant $C_0$ such that $|\nabla u| \le C_0$ on $U_0 \cap \cS_k$ (the function $u$ is differentiable on $\cS_k$ by Proposition \ref{dominatedprop}). Since calibrated curves are integral curves of $\nabla u$, there exist $0 < \eps_0 <2^{-k} $ and neighborhoods $U_0 \supseteq U' \supseteq U \ni x_0$ such that every calibrated curve $\gamma:(-\eps_0,\eps_0) \to M$ with $\gamma(0) \in U$ is contained inside $U'$. In fact, we can take $U$ and $U'$ to be concentric balls about $x_0$ of radius $r$ and $R$ respectively, with respect to some auxiliary Riemannian metric, and assume that every calibrated curve $\gamma:(-{\eps_0},{\eps_0}) \to M$ with $\gamma(0) \in U$ is contained in a ball centered at $x_0$ of radius $R - 2r$. From now on we work in normal coordinates on $U'$, identifying it with an open ball in $\RR^n$ centered at the origin, and identifying $TU'$ with $U' \times \RR^n$. 

	\medskip
	Let $x \in U \cap \cS_k$. Since ${\eps_0} < 2^{-k}$, there exists a calibrated curve $\gamma:(-{\eps_0},{\eps_0}) \to M$  with $\gamma(0) = x$ . By construction, the curve $\gamma$ is contained in a ball of radius $R-2r$ centered at the origin. Let $x' \in U$ and set
	$$v_0 : = \frac{x' - x}{{\eps_0} / 2}.$$
	Define
	$$\theta(t) : = \gamma(t)  + (t + {\eps_0}/2)\cdot v_0 , \qquad -{\eps_0}/2 \le t \le 0.$$
	Then 
		$$\theta(-{\eps_0}/2) = \gamma(-{\eps_0}/2) \qquad \text{ and } \qquad \theta(0) = x'.$$
	Since $\gamma$ is contained in a ball of radius $R-2r$, and $|(t+{\eps_0}/2)v_0 | \le 2|x'-x| \le 2r$, the line segment $\theta$ lies in $U'$. We have
	$$w(t) : = (\theta(t),\dot\theta(t)) - (\gamma(t),\dot\gamma(t)) = ((t+{\eps_0}/2)v_0 ,v_0 ).$$
	Since $u$ is dominated,
	$$u(x') \le u(\gamma(-{\eps_0}/2)) + \int_{-{\eps_0}/2}^0L(\theta(t),\dot\theta(t))ds$$
	and since $\gamma$ is calibrated,
	$$u(x) = u(\gamma(-{\eps_0}/2)) + \int_{-{\eps_0}/2}^0L(\gamma(t),\dot\gamma(t))dt,$$
	whence
	\begin{align*}
		u(x') - u(x) & \le \int_{-{\eps_0}/2}^0\left[L(\theta(t),\dot\theta(t)) - L(\gamma(t),\dot\gamma(t))\right]dt\\
		& = \int_{-{\eps_0}/2}^0 \left[dL\vert_{(\gamma(t),\dot\gamma(t))}\cdot w(t) + O(|w(t)|^2)\right]dt\\
		& = \left[\int_{-{\eps_0}/2}^0\left((t+\eps_0/2)\cdot\frac{\partial L}{\partial x}\bigg\vert_{(\gamma(t),\dot\gamma(t))} + \frac{\partial L}{\partial v}\bigg\vert_{(\gamma(t),\dot\gamma(t))}\right)dt\right]\cdot v_0  + O\left(\frac{|x'-x|^2}{\eps_0}\right)
	\end{align*}
	where we write $L = L(x,v)$ for $(x,v) \in U'\times\RR^n \cong TU'$. Since $\gamma$ is calibrated, it is an extremal and thus satisfies the Euler-Lagrange equation \eqref{ELeq}. Substituting \eqref{ELeq} into the first summand in the integrand and then integrating it by parts, we obtain
	\begin{align*}
		u(x') - u(x) \le \frac{\eps_0}{2}\cdot\frac{\partial L}{\partial v}\bigg\vert_{(\gamma(0),\dot\gamma(0))}\cdot v_0  + O\left(\frac{|x'-x|^2}{\eps_0}\right) = \frac{\partial L}{\partial v}\bigg\vert_{(x,\dot\gamma(0))}\cdot(x'-x) + O\left(\frac{|x'-x|^2}{\eps_0}\right).
	\end{align*}
	
	Since $u$ is differentiable at $x$, and the first term is linear in $x'$ and vanishes when $x' = x$, it must equal $(du\vert_x)(x'-x)$. The remainder term is $O(|x'-x|^2)$ with a constant depending only on ${\eps_0}$ and on second derivatives of $L$ on a compact set independent of $x$ and $x'$, which is bounded away from the zero section; indeed, since $E(\dot\gamma)\equiv 0$, the straight line $s \mapsto \dot\gamma(t) + s(t+\eps_0/2)v_0 $ joining $\dot\gamma(t)$ to $\dot\theta(t)$ is bounded away from zero  independently of $x,x'$ and $t$, provided that $U'$ is sufficiently small. 
    
    \medskip
    A similar analysis on the interval $[0,{\eps_0}/2]$ gives the bound in the other direction.
\end{proof}

\begin{lemma}\label{whitneylemma}
	For every $x_0 \in M$ and every $k \ge 1$ there exists a neighborhood $U_{x_0,k} \ni x_0$ and a $C^{1,1}$ function $u_{x_0,k} : U_{x_0,k} \to \RR$ such that
	\begin{equation}\label{ux0eq} u_{x_0,k} = u \qquad \text{ and } \qquad du_{x_0,k} = du \qquad \text{ on } U_{x_0,k}\cap\cS_k.\end{equation}
\end{lemma}
\begin{proof}
	Fix $k \ge 1$ and $x_0 \in M$ and let $U_{x_0},C_{x_0}$ be as in the conclusion of Lemma \ref{taylorlemma}; then
    \begin{equation}\label{uestimates2}|u(x') - u(x) - (du\vert_x)(x'-x)| \le C_{x_0,k}\cdot|x' - x|^2 \qquad \forall \, x,x' \in U_{x_0,k}\cap\cS_k,\end{equation}
	where we view $U_{x_0,k}$ as a subset of $\RR^n$ using a fixed coordinate chart depending only on $x_0$.  Since $u$ is locally Lipschitz, we can also assume that
	\begin{equation}\label{uestimates1}|u| \le C_{x_0,k} \qquad \text{ and } \qquad |du| \le C_{x_0,k} \qquad \text{ on } \quad U_{x_0,k}\cap\cS_k.\end{equation}
	Inequality \eqref{uestimates2} also implies that, by replacing $C_{x_0,k}$ by a larger constant, we can guarantee that
	\begin{equation}\label{uestimates3}|du\vert_x - du\vert_{x'}| \le C_{x_0,k}\cdot|x-x'| \qquad \forall \, x,x' \in U_{x_0,k}\cap \cS_k.\end{equation}
	Indeed, this is the content of \cite[Lemma 3.8]{Fat03}; the proof is simple so we repeat it. Let $x,x{'} \in U_{x_0,k}\cap \cS_k$ and choose $x'' \in U_{x_0,k}$ with 
    \begin{equation}\label{absxxprimeeq}
        |x-x''| \le \min\{|x-x{'}|,|x''-x{'}|\}.
    \end{equation}
    Then
	\begin{align*}
		& \,|du\vert_x - du\vert_{x{'}}| \cdot |x{'} - x''|\\
		= & \,|(du\vert_x)(x{'}-x) + (du\vert_x)(x-x'') - (du\vert_{x{'}})(x{'}-x'')|\\
		\le & \,|(du\vert_x)(x{'}-x) + u(x) - u(x{'})|+|(du\vert_x)(x-x'') - u(x) + u(x'')|\\
		& +|u(x{'}) - u(x'') - (du\vert_{x{'}})(x{'}-x'')|\\
		\le & \,C_{x_0,k}\cdot\left(|x{'}-x|^2 + |x-x''|^2 + |x{'}-x''|^2\right) \\
		\le & \,5C_{x_0,k}\cdot |x{'}-x''|\cdot|x-x{'}|,
	\end{align*}
    where in the penultimate passage we used \eqref{uestimates2}, and in the last passage we used the triangle inequality and \eqref{absxxprimeeq}.	
    
    \medskip
    It now follows from \eqref{uestimates2}, \eqref{uestimates1}, \eqref{uestimates3} and Whitney's extension theorem (see \cite[Chapter VI]{Stein} or \cite[Section 2.2]{Kl}) that there exists a $C^{1,1}$ function $u_{x_0,k}:U_{x_0,k} \to \RR$ satisfying \eqref{ux0eq}.
\end{proof}

\begin{proof}[Proof of Theorem \ref{regularitythm}]
	Let $k\ge 1$. Lemma \ref{whitneylemma} provides us with a cover $\{U_{x,k}\}_{x \in M}$ of $M$. Extracting a locally finite refinement, we obtain a countable collection $\{U_{j,k},u_{j,k}\}_{j=1}^\infty$ such that $M = \bigcup_jU_{j,k}$ and such that each $u_{j,k}$ is $C^{1,1}$ on $U_{j,k}$ and satisfies
	$$u_{j,k} = u \qquad \text{and} \qquad  du_{j,k} = du \qquad \text{on } U_{j,k}\cap\cS_k.$$
	Let $\{\phi_{j,k}\}_{k\ge 1}$ be a partition of unity subordinate to $\{U_{j,k}\}$, i.e. each $\phi_{j,k}$ is smooth and compactly supported on $U_{j,k}$ and $\sum_j\phi_{j,k} \equiv 1$ on $M$. Set
	$$u_k : = \sum_{j=1}^\infty \phi_{j,k}u_{j,k}.$$
	Then $u_k$ is $C^{1,1}$, and since $\{\phi_{j,k}\}$ is a partition of unity,
	$u_k\equiv u$ and $du_k \equiv du$ on $\cS_k$.
\end{proof}

\subsection{Ray clusters}

Our next task is to divide the strain set $\cS$ into countably many \emph{ray clusters}, which are transport sets admitting a convenient parametrization. The section loosely follows Chapter 3 of \cite{Kl}. Let $\{u_k\}$ be a sequence of $C^{1,1}$ functions on $M$ satisfying
\begin{equation}\label{ukeq}u_k \equiv u \qquad \text{ and } \qquad du_k \equiv du \qquad \text{ on } \cS_k;\end{equation}
such a sequence exists by Theorem \ref{regularitythm}.

\begin{definition}[Twice differentiable function]\normalfont
    A differentiable function $\phi : M \to \RR$ will be called \emph{differentiable twice} at a point $x \in M$ if the map $d\phi : M \to T^*M$ is differentiable at $x$, and moreover the exterior derivative of the one-form $d\phi$ vanishes at $x$:
    $$dd\phi\vert_x = 0.$$
    Equivalently, the tangent space to the graph of $d\phi$ at the point $p : = d\phi\vert_x$ exists and is a Lagrangian subspace of $T_pT^*M$.
\end{definition}

\begin{lemma}[\text{\cite[Corollary 3.13]{Kl}}]\label{dtlemma}
    Let $\phi: M \to \RR$ be a $C^{1,1}$ function. Then $\phi$ is differentiable twice almost everywhere.
\end{lemma}

Here and below, the term \emph{almost everywhere} is understood with respect to some, hence any, choice of density on the manifold $M$. 

\begin{definition}[Ray Clusters]\normalfont\label{clusterdef}
    A compact set $R_0\subseteq \cS$ will be called a \emph{seed of a ray cluster} if there exist an open, connected, bounded  set $\Omega_0\subseteq \RR^{n-1}$, a smooth embedding $F_0:\overline{\Omega}_0 \to M$, and $k_0 \ge 1$ such that:
    \begin{enumerate}[$(i)$]
        \item Every transport ray intersects the set $R_0$ at most once.
        \item The hypersurface $F_0(\Omega_0)$ is transverse to the vector field ${\nabla}u_{k_0}$.
        \item The functions $a$ and $b$ defined in \eqref{Ixeq} are continuous on the set $R_0$.
        \item The function $u_{k_0}$ is differentiable twice on the set $R_0$.
        \item $R_0 \subseteq F_0(\Omega_0)\cap \cS_{k_0}.$
    \end{enumerate}
    A \emph{ray cluster} is a Borel set $R \subseteq M$  satisfying 
    $$R = \hat R \subseteq \hat R_0$$
    where $R_0$ is a seed of a ray cluster. Thus a ray cluster is a transport set contained in the saturation of some seed of a ray cluster. See Figure \ref{clusterfig}.
\end{definition}
\begin{figure}
    \centering
    \includegraphics[width=.5\textwidth]{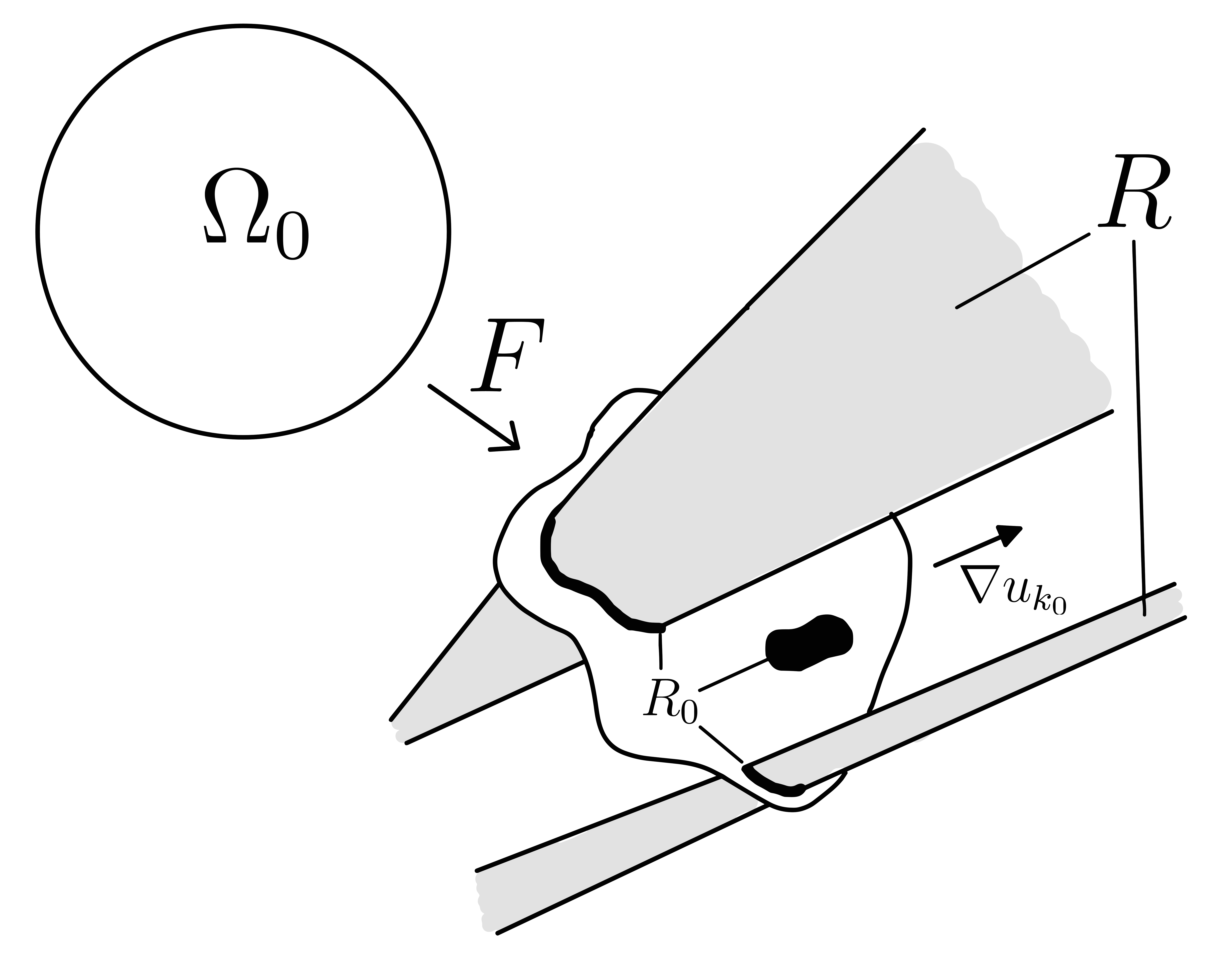}
    \caption{A ray cluster.}
    \label{clusterfig}
\end{figure}
\begin{remark*}
    Our definition of a ray cluster is different from the one given in \cite{Kl}.
\end{remark*}

Immediately from the definition we get:

\begin{lemma}\label{subsetlemma}
    A transport set contained in a ray cluster is itself a ray cluster.
\end{lemma}

We now show that each ray cluster admits a locally-Lipschitz parametrization. 
The \emph{Jacobian determinant} of a function $F : \RR^{n} \to M$ at a point of differentiability $y = (y^1,\dots,y^n) \in \RR^n$ is defined by
$$\textstyle\det_\mu dF\vert_y := \omega(dF\vert_y(\partial/\partial y^1),\dots,dF\vert_z(\partial/\partial y^n)),$$
where $\omega$ is the density of $\mu$.

\begin{figure}[t]
    \centering
    \includegraphics[width =.4\textwidth]{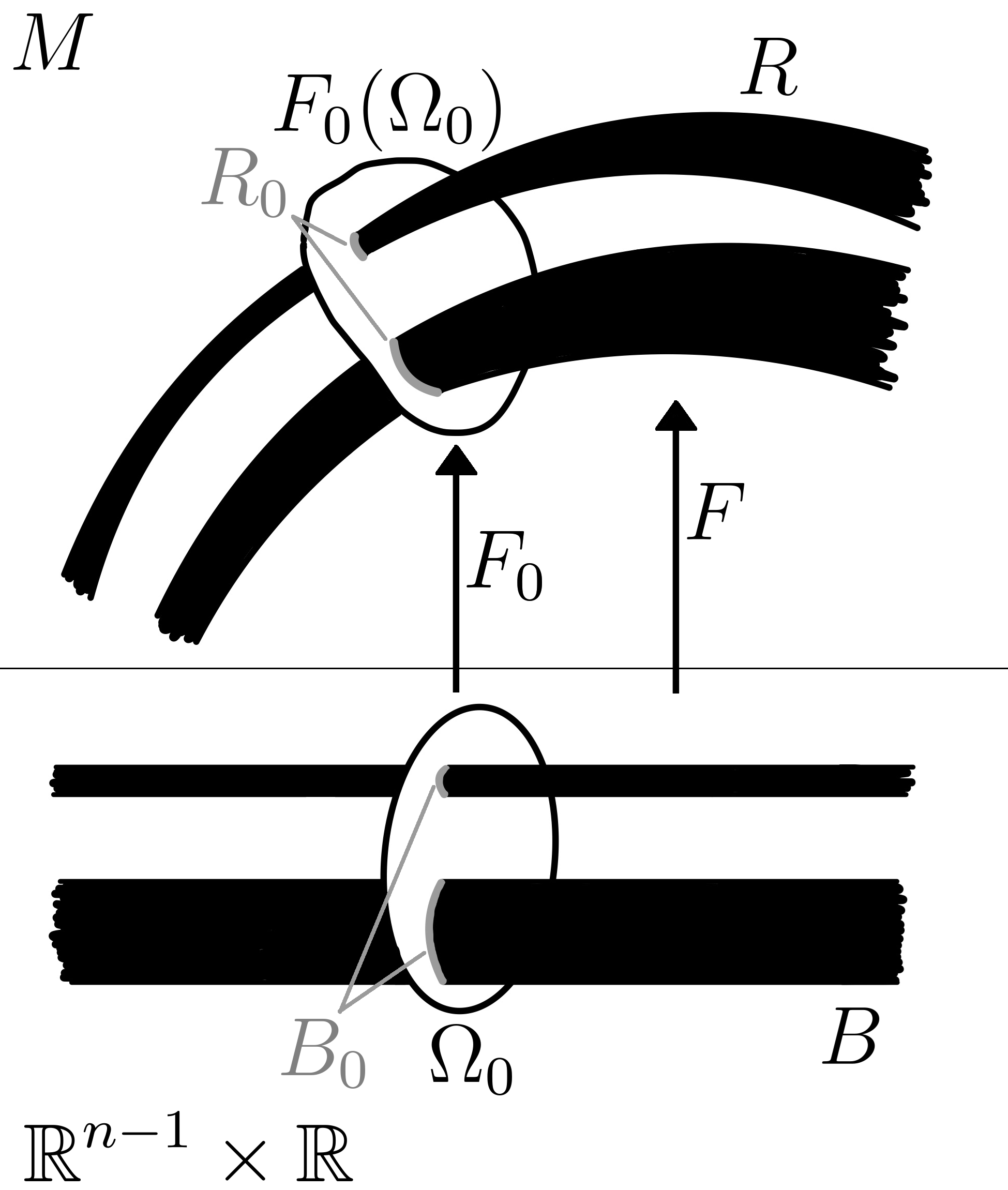}
    \caption{A ray cluster and its parametrization.}
    \label{flattenfig}
\end{figure}

\begin{proposition}[Lipschitz parametrization of ray clusters, cf. \text{\cite[Proposition 3.5]{Kl}}]\label{flattenprop}
    Let $R$ be a ray cluster. There exist a Borel set $B \subseteq \RR^{n-1}\times \RR$ and a locally-Lipschitz, one-to-one map $F : B \to R$ with the following properties:
    \begin{enumerate}[$(i)$]
        \item The set $B$ takes the form
        \begin{equation}\label{Bdef}
            B = \{(y,t) \in \RR^{n-1}\times \RR \, \mid \, y \in B_0, \quad t \in (a_y,b_y)\},
        \end{equation}
        where $B_0 \subseteq \RR^{n-1}$ is a compact set and $-\infty \le a_y < b_y \le \infty$ are continuous functions of $y$. See Figure \ref{flattenfig}.
        \item For every $y \in B_0$, the curve $t \mapsto F(y,t)$ is calibrated, and its image is the relative interior of a transport ray.
        \item For every Lebesgue density point $y \in B_0$ and every $t \in (a_y,b_y)$, the point $F(y,t)$ is a Lebesgue density point of the ray cluster $R$, the map $F$ is differentiable at $(y,t)$, and the Jacobian 
        $$J(y,t): =\textstyle\det_\mu dF\vert_{(y,t)}, \qquad (y,t) \in B$$
        is positive, continuously differentiable in the $t$ variable, and for sufficiently large $k$ satisfies
        $$\frac{\dot J(y,t)}{J(y,t)} = \bL u_k(F(y,t))  \qquad \text{ for all } t \in (a_y + 2^{-k},b_y - 2^{-k}),$$
        where dot denotes differentiation with respect to $t$.
        \item We have the change of variables formula
        \begin{equation}\label{COVeq}
            \int_Rhd\mu = \int_{B_0}\int_{a_y}^{b_y}h(F(y,t)) \cdot J(y,t)\,dt\,dy
        \end{equation}
        for every Borel function $h : M \to \RR$. In particular,
        \begin{equation}\label{nulleq1}\mu\left(R\setminus F(B)\right) = 0.\end{equation}
    \end{enumerate}
\end{proposition}

\begin{remark*}
    Since the function $F$ is defined on the set $B$ which is not open, by \emph{differentiable at $(y,t)$} we mean that some locally-Lipschitz extension of $F$ to an open set containing $B$ is  differentiable at $(y,t)$, and that the differential is uniquely determined by the values of $F$ on the set $B$ (this is true when $(y,t)$ is a Lebesgue density point of $B$). We do not dwell on this technical point here; the details are treated carefully in \cite[Chapter 3]{Kl}.
\end{remark*}

Before proving Proposition \ref{flattenprop}, let us establish some preliminary lemmas (and in particular, justify the expression $\bL u_k$ appearing in Proposition \ref{flattenprop}, as $u_k$ is only $C^{1,1}$). For $x \in \cS$ set 
$$I_{x,k} : = I_x\cap\gamma_x^{-1}(\cS_{k-1}) = \left(a(x) + 2^{-k+1},b(x) - 2^{-k+1}\right),$$
with the understanding that $I_{x,k} = \varnothing$ if $a(x) + 2^{-k+1} \ge b(x) - 2^{-k+1}$.

\begin{lemma}\label{Gktlemma}
    For $t \in \RR$ and $k \ge 1$, set
    $$A_k^t : = \{x \in \cS_k \, \mid \, t \in I_x\}.$$
    The map
    $$G_k^t : U_k^t \to M \qquad  \qquad G_k^t(x) := \pi(\Phi_t^H(du_k\vert_x))$$
    is defined and locally-Lipschitz on an open set $U_k^t\supseteq A_k^t$, and
    $$G_k^t(x) = \gamma_x(t) \qquad \text{ for all $x \in A_k^t$.}$$
\end{lemma}

\begin{lemma}\label{endcontlemma}
    Let $R$ be a ray cluster and let $R_0$ and $k_0$ be as in Definition \ref{clusterdef}. The functions $a,b$ are continuous on the set $R\cap \cS_k$ for all $k \ge k_0$.
\end{lemma}

\begin{lemma}[Regularity of $G_k^t$ and $u_k$ on ray clusters]\label{Gktlemma2}
    Let $R_0 \subseteq \cS$ be a seed of a ray cluster and let $R$ be a ray cluster satisfying $R = \hat R \subseteq \hat R_0$. 
    Let $x \in R_0$ be a Lebesgue density point of the ray cluster $R$, let $k_0$ be as in Definition \ref{clusterdef} and let $k\ge k_0$. For every $t \in I_{x,k}$ the following hold:
    \begin{enumerate}[(i)]
        \item The point $\gamma_x(t)$ is a Lebesgue density point of the set $R \cap \cS_k$.
        \item The map $G_k^t$ is defined and differentiable at $x$, and its differential is invertible and depends $C^1$-smoothly on $t$.
        \item The map $du_k:M \to T^*M$ is differentiable at $\gamma_x(t)$ and its differential depends $C^1$-smoothly on $t$. In particular, the function $\bL u_k\circ\gamma_x$ is well defined and continuous on $I_{x,k}$.
    \end{enumerate}
\end{lemma}

\begin{proof}[Proof of Lemma \ref{Gktlemma}]
    Since $du_k$ is locally Lipschitz and the Hamiltonian flow $\Phi_t^H$ is $C^1$, general existence and smoothness theorems for ordinary differential equations imply that the map $G_k^t$ is defined and locally-Lipschitz on an open set $U_k^t \subseteq M$. 

    \medskip
    Let $x\in A_k^t$. Then $du\vert_x = du_k\vert_x$ by \eqref{ukeq}. Therefore, by \eqref{dotgammaxeq} and the fact that $\gamma_x$ is an extremal, 
    \begin{align*}
        G_k^t(x) & = \pi\left(\Phi_t^H(du_k\vert_x)\right)\\
        & = \pi\left(\Phi_t^H(du\vert_x)\right)\\
        & = \pi\left(\Phi_t^L({\nabla}u\vert_x)\right)\\
        & = \pi\left(\Phi_t^L(\dot\gamma_x(0))\right)\\
        & = \gamma_x(t),
    \end{align*}
    and in particular $G_k^t$ is defined at $x$, i.e $x \in U_k^t$.
\end{proof}

\begin{proof}[Proof of Lemma \ref{endcontlemma}]
    Since $R = \hat R \subseteq \hat R_0$, for each $x \in R\cap\cS_k$ there exist $x_0 \in R_0$ and $t_0 \in I_{x_0,k+1}$ such that $x = \gamma_{x_0}(t_0)$, and then we have
    $$a(x) = a(x_0) - t_0 \qquad \text{ and } \qquad b(x) = b(x_0) - t_0.$$
    Thus, since $a$ and $b$ are continuous on $R_0$, it suffices to show that $x_0$ and $t_0$ depend continuously on $x$, i.e. that if $x_i$ is a sequence of points in $R_0$ and $t_i \in I_{x_i,k+1}$ is a sequence of real numbers such that $$\gamma_{x_i}(t_i) \to \gamma_{x_0}(t_0)\in \cS_k,$$ then $\lim x_i = x_0$ and $\lim t_i = t_0$. By compactness of $R_0$, we may assume that the sequence $x_i$ converges to some $x_0' \in R_0 \subseteq \cS_{k_0}\subseteq\cS_k$. We may also assume that $t_i$ converge to some $t_0'$, since they are uniformly bounded by Lemma \ref{diamlemma}. Note that $t_0' \in I_{x_0'}$ by continuity of the functions $a,b$ on $R_0$. By \eqref{ukeq} and the continuity of $\nabla u_k$,
    $$\lim\dot\gamma_{x^i}(0) = \lim\nabla u\vert_{x_i} = \lim \nabla u_k\vert_{x^i} = \nabla u_k\vert_{x_0'}.$$
    It follows that
    $$\gamma_{x_0}(t_0) = x = \lim_i\gamma_{x^i}(t_i) = \lim_i\pi(\Phi_{t_i}^L(\dot\gamma_{x^i}(0))) = \pi(\Phi_{t_0'}^L(\nabla u_k\vert_{x_0'})) =G_k^{t_0'}(x_0') = \gamma_{x_0'}(t_0')$$
    by Lemma \ref{Gktlemma}. But then it follows that the point $x$ lies both on a transport ray through $x_0$ and on a transport ray through $x_0'$; by Proposition \ref{dominatedprop}, the two transport rays coincide, so since each transport ray intersects $R_0$ at most once, $x_0 = x_0'$ and $t_0 = t_0' = \lim t_i$.
\end{proof}

\begin{proof}[Proof of Lemma \ref{Gktlemma2}]
    Fix $t\in I_{x,k}$. Since $x \in R_0 \subseteq \cS_{k_0} \subseteq \cS_k$, Lemma \ref{Gktlemma} implies that there exists a neighborhood $U\ni x$ which is contained in the domain $U_k^t$ of the function $G_k^t$. By Lemma \ref{endcontlemma}, the neighborhood $U$ can be chosen such that 
    $$I_{x,k} \subseteq I_{x',k+1} \qquad \text{ for all } \quad x'\in U\cap R.$$
    It follows that
    \begin{equation}\label{gammaxprimeeq}
        \gamma_{x'}(I_{x,k}) \subseteq \gamma_{x'}(I_{x',k+1}) = \gamma_{x'}(a(x') + 2^{-k},b(x') - 2^{-k}) \subseteq \cS_k \qquad \text{for all} \quad x' \in U\cap R.\end{equation}
    It follows from \eqref{ukeq}, Lemma \ref{Gktlemma} and Proposition \ref{dominatedprop}(d) that whenever $x'\in U\cap R$,
	\begin{equation}\label{Gconj}du_k\vert_{G_k^t(x')} = \Phi_t^H(du_k\vert_{x'}) \qquad \text{for all $t \in I_{x,k}$.}\end{equation}
	Since $x \in R_0$, the function $u_k$ is differentiable twice at $x$, whence $G_k^t = \pi \circ \Phi_t^H\circ du_k$ is differentiable at $x$, with its derivative given by the chain rule:
    \begin{equation*}
        (dG_k^t)\vert_x = d\pi \circ d\Phi_t^H\circ (Tdu_k)\vert_x.
    \end{equation*}
    
    Here and below we denote by $$(Tdu_k)\vert_x : T_xM \to T_{du_k\vert_x}T^*M$$ the differential of the map $du_k : M \to T^*M$ at the point $x$, so as not to confuse it with the exterior derivative $ddu_k$. Since the Hamiltonian vector field $X_H$ is $C^1$, the differential $d\Phi_t^H$ of its flow is $C^1$ with respect to $t$. Therefore the differential of $G_k^t$ at $x$ is $C^1$ with respect to $t$.

    \medskip
    Write
	$$z :=  G_k^t(x) = \gamma_x(t).$$
	Let $x' \in U\cap R$ and let $z' :=  G_k^t(x')$. By \eqref{Gconj}, 
	$$x' = (\pi\circ du_k)(x') =  (\pi\circ \Phi_{-t}^H \circ \Phi_t^H \circ du_k)(x') =(\pi\circ \Phi_{-t}^H \circ du_k \circ  G_k^t)(x') = (\pi \circ \Phi_{-t}^H \circ du_k)(z').$$
	For the rest of the proof fix some Riemannian metric on $U$. Since $du_k$ is locally-Lipschitz and $\Phi_{-t}^H$ is $C^1$, it follows that
	$$d(x',x) = O(d(z',z)).$$
	As $x'$ is an arbitrary point in $U\cap R$ and $x$ is a Lebesgue density point of $R$, this proves that the differential of the map $ G_k^t$ is invertible at $x$. Moreover, since $x$ is a Lebesgue point of $R$ and $z' \in R \cap \cS_k$ by \eqref{gammaxprimeeq}, it also follows that $z$ is a Lebesgue density point of $R\cap \cS_k$. Thus (i) and (ii) are proved.
    	
	\medskip
	Let $v \in T_zM$ and let $w \in T_xM$ satisfy $$d G_k^t\vert_{x}(w) = v.$$
    Let $\theta$ be a curve satisfying $\dot\theta(0) = v$ and let $r_j > 0$ be any sequence converging to zero. 	Since $x$ is a Lebesgue density point of $R$, there exists a sequence $x_j \in U\cap R$ converging to $x$ such that
	$$d(x_j,\exp_x(r_j \cdot w)) = o(r_j),$$
	where $\exp$ is the exponential map of the auxiliary Riemannian metric. Write 
	$$z_j : =  G_k^t(x_j).$$
	 In the following computation we abuse notation by writing ${p} = {p}' + o(r_j)$ for ${p},{p}' \in T^*M$  to indicate that $\lim_{j \to \infty}r_j^{-1}d({p},{p}') \to 0$, where the distance is measured using, say, the Sasaki metric on $T^*M$. By \eqref{Gconj} and local Lipschitz continuity of $du_k$ and $\Phi_t^H$,
	\begin{align*}
		du_k\vert_{z_j} & = (du_k\circ G_k^t)(x_j)\\
		& = (\Phi_t^H\circ du_k)(x_j)\\
		& = (\Phi_t^H\circ du_k)(\exp_x(r_jw)) + o(r_j).
	\end{align*}
	Thus, by second order differentiability of $u_k$ at $x$, 
	\begin{align*}
		du_k\vert_{z_j} & = \Phi_t^H(du_k\vert_x) + r_j \cdot (d\Phi_t^H \circ Tdu_k\vert_x)(w) + o(r_j)\\
		& = du_k\vert_z + r_j\cdot \left[d\Phi_t^H\circ Tdu_k\vert_x\circ (d G_k^t\vert_x)^{-1}\right](v) + o(r_j),
	\end{align*}
	where in the second passage we used \eqref{Gconj} and the definition of $w$. On the other hand, by local Lipschitz continuity of $ G_k^t$ and the fact that $d G_k^t(w) = v = \dot\theta(0)$,
	$$du_k\vert_{z_j} = (du_k\circ  G_k^t)(x_j) = (du_k\circ  G_k^t)(\exp_x(r_j\cdot w)) + o(r_j) = du_k\vert_{\theta(r_j)} + o(r_j).$$
	Putting the last two computations together we get
	$$du_k\vert_{\theta(r_j)} = du_k\vert_z + r_j \cdot \left[d\Phi_t^H\circ Tdu_k\vert_x\circ (d G_k^t\vert_x)^{-1}\right](v) + o(r_j).$$
	Recall that $\theta$ is an arbitrary curve emanating from $z = \gamma_x(t)$ with initial velocity $v \in T_zM$. Hence the above formula implies the existence of the differential $Tdu_k\vert_{\gamma_x(t)}$ and establishes the formula
    $$Tdu_k\vert_{\gamma_x(t)} = d\Phi_t^H\vert_{du_k\vert_x}\circ  Tdu_k\vert_x\circ (d G_k^t\vert_x)^{-1}, \qquad t \in I_{x,k}.$$
    The maps $(dG_k^t\vert_x)^{-1}$ and $d\Phi_t^H\vert_{du_k\vert_x}$ are continuously differentiable in $t$, whence the differential $ Tdu\vert_{\gamma_x(t)}$ is continuously differentiable in $t$ for $t \in I_{x,k}$.
\end{proof}

We are now ready to prove Proposition \ref{flattenprop}.

\begin{proof}[Proof of Proposition \ref{flattenprop}]
    By the definition of a ray cluster, there exists a seed of a ray cluster $R_0$ such that $R = \hat R\subseteq  \hat R_0$. Let $F_0:\overline{\Omega}_0 \to M$ and $k_0 $ be as in Definition \ref{clusterdef}. Set 
    $$B_0 : = F_0^{-1}(R_0).$$
    Since the map $F_0$ is a smooth embedding, the set $B_0$ is compact, and since $R_0 \subseteq F_0(\Omega_0)$, 
    $$F_0(B_0) = R_0.$$
    Let $y \in B_0$, and write 
    $$x : = F_0(y)$$
    Set
    $$a_y : = a(x) \qquad \text{ and } \qquad b_y:=b(x),$$
    and define
    \begin{equation}\label{Fdef}
        F(y,t) := \gamma_x(t). \qquad t \in (a_y,b_y).
    \end{equation}
    Then $F$ is defined on the set $B \subseteq \RR^{n-1}\times \RR$ given by \eqref{Bdef}, and for every $y \in B_0$ the curve $F(y,\cdot) = \gamma_x$ is calibrated by the definition of $\gamma_x$. The functions $a_y,b_y$ are continuous on $B_0$ by Definition \ref{clusterdef}. Thus (i) and (ii) are proved.
    
    \medskip
    Suppose that $y$ is a Lebesgue density point of $B_0$. Then $x$ is a Lebesgue density point of the ray cluster $R$. Indeed, since the functions $a,b$ are continuous at the point $x$, there exists $\eps > 0$ and a neighborhood $U \ni x$ such that $(-\eps,\eps) \subseteq I_{x'}$ for every $x' \in R_0 \cap U$. It follows that for sufficiently small $r > 0$, the ray cluster $R$ contains the set
    $$A_r : = \{\gamma_{x'}(t) \, \mid \, x' = F_0(y'), \quad y' \in B_0, \quad |y - y'| < r, \quad -r< t < r\},$$
    whose measure, since $y$ is a Lebesgue density point of $B_0$, is
    \begin{align*}
        \mu(A_r) & = \omega(dF_0(\partial/\partial x^1),\dots,dF_0(\partial/\partial x^{n-1}),{\nabla}u\vert_x)\cdot c_{n-1}r^{n-1}\cdot 2r + o(r^{n-1}) \\
        & = (1 + o(1))\mu(V_r)\qquad \text{ as $r \to 0$,}
    \end{align*}
    where $c_{n-1}$ is the volume of the unit ball in $\RR^{n-1}$, and 
    $$V_r : = \{\gamma_{x'}(t) \, \mid \, x' = F_0(y'), \quad y' \in \Omega_0, \quad |y - y'| <r, \quad -r < t < r\}.$$
    Since $F(\Omega_0)$ is a smooth embedded hypersurface transverse to ${\nabla}u$, the sets $V_r$ shrink nicely to $x$ (see \cite[Section 3.4]{Fol}); so since $A_r \subseteq R$, it follows that $x$ is a Lebesgue density point of $R$. By Lemma \ref{Gktlemma2}, the point $\gamma_x(t)$ is a Lebesgue density point of $R$ for all $t \in I_{x,k}$.
    
    \medskip
    Let $k \ge k_0$. Since $F_0(y)= x\in R_0 \subseteq \cS_{k_0} \subseteq \cS_k$, Lemma \ref{Gktlemma} implies that
    $$F(y,t) = G_k^t(F_0(y)). \qquad \text{for all $t \in I_{x,k}$}.$$
    In particular, by Lemma \ref{Gktlemma}, the function $F$ extends to a locally-Lipschitz function on an open set containing $B_0\times\{0\}$, namely $\tilde F(y,t) : = G_k^t(F_0(t))$. By Lemma \ref{Gktlemma2}, the function $G_k^t$ is differentiable at $x$ for all $t \in I_{x,\eps}$ with an invertible derivative which is continuously differentiable with respect to $t$. Furthermore, 
    \begin{equation}\label{dFpartialteq}
        dF\vert_{(x,t)}(\partial/\partial t) = \dot \gamma_x(t) = {\nabla}u_k\vert_{\gamma_x(t)} = \Phi^L_t({\nabla}u_k\vert_x)\qquad \text{ for all $t \in I_{x,k}$.}
    \end{equation}
    Thus, the fact that the image of $F_0$ is transverse to ${\nabla}u = {\nabla}u_k$ at $x$ implies that $F$ is differentiable at $(y,t)$ with an invertible differential depending $C^1$-smoothly on $t$, for all $t \in I_{x,k}$. Since this last statement is true for every $k \ge k_0$ and $I_x = \bigcup_{k \ge k_0}I_{x,k}$, we conclude that it is true on the entire interval $I_x = (a_y,b_y)$. 
    
    \medskip
    Since the differential of $F$ is invertible at $(y,t)$ for all $t \in (a_y,b_y)$, the function $J = \det_\mu dF$ does not change sign on the line $\{y\} \times (a_y,b_y)$. But the set $F(B_0\times\{0\}) = R_0$ is contained in the hypersurface $F_0(\Omega_0)$, which is smooth and transverse to ${\nabla}u_{k_0}$, so the sign of $J$ is constant on $B_0\times\{0\}$. We can therefore take the Jacobian determinant $J$ to be positive at $(y,t)$ for every Lebesgue point $y \in B_0$ and every $t \in (a_y,b_y)$. 
    
    \medskip
    Since $dF$ is $C^1$ with respect to $t$, the vector fields $dF(\partial/\partial y_i)$ are $C^1$ along $\gamma_x$. By \eqref{dFpartialteq} and the definition of the Jacobian determinant, on $I_{x,k}$ we have
    \begin{align*}
        \dot J & = \frac{d}{dt}\omega(dF(\partial/\partial y_1),\dots,dF(\partial /\partial y_{n-1}),dF(\partial /\partial t))\\
        & = (\sL_{\nabla u_k}\omega)(dF(\partial/\partial y_1),\dots,dF(\partial /\partial y_{n-1}),dF(\partial /\partial t))\\
        & = (\bL u_k\circ F)\cdot\omega(dF(\partial/\partial y_1),\dots,dF(\partial /\partial y_{n-1}),dF(\partial /\partial t))\\
        & = (\bL u_k\circ F)\cdot J.
    \end{align*}
    This finishes the proof of (iii). 
    
    \medskip
    A standard change of variables formula for Lipschitz maps (see \cite[Theorem 3.9]{EvGa} and the discussion at the end of the proof of \cite[ Proposition 3.5]{Kl}) implies that for every Borel function $f : B \to \RR$,
    $$\int_{B_0}\int_{a_y}^{b_y}f(y,t)\cdot J(y,t)\,dt\,dy = \int_{F(B)}f(F^{-1}(x))d\mu(x).$$
    If $h : M \to \RR$ is a Borel function, then applying the above change of variables formula to $f : = h \circ F$ gives
    $$\int_{B_0}\int_{a_y}^{b_y}h(F(y,t))\cdot J(y,t) \, dt \, dy = \int_{F(B)}h(x)d\mu(x).$$
    Thus, in order to prove formula \eqref{COVeq}, it remains to show \eqref{nulleq1}. The definition \eqref{Fdef} of $F$ implies that $F(B)$ consists of \emph{relative interiors} of transport rays intersecting the set $R_0 = F(B_0)$; on the other hand, $R$ consists of the \emph{full} transport rays, i.e. $R = \hat R_0 = \widehat{F(B_0)}$. Thus $R\setminus F(B)$ consists of \emph{initial and final points} of transport rays intersecting $R_0$, whence \eqref{nulleq1} follows from Lemma \ref{endslemma}.
\end{proof}

After having established the desired properties of ray clusters, let us show that the strain set can be decomposed into ray clusters, up to a null set.

\begin{lemma}\label{localclusterlemma}
    For every $k \ge 1$ and every $x_0 \in \cS_k$, there exists a neighborhood $U \ni x_0$ such that the set $U\cap \cS_{k-1}$ is contained, up to measure zero, in a countable collection of ray clusters.
\end{lemma}

\begin{proof}
    Let $k \ge 1$ and $x_0 \in \cS_k$. Fix $0 < \ell < 2^{-k}$ and $\delta>0$ small, to be specified later. 
    Let
    $$\Omega_0 : = (-\delta,\delta)^{n-1}\subseteq\RR^{n-1}$$
    and let 
    $$\vphi:\Omega_0\times(-\ell,\ell) \to U_0 \subseteq M$$
    be a diffeomorphism onto an open set $U_0$ containing $x_0$, such that for every $c \in (-\ell,\ell)$ the hypersurface $\vphi\left(\Omega_0\times\{c\}\right)$ is transverse to the vector field ${\nabla}u_k$ (this is possible for small enough $\delta,\ell$, since the vector field ${\nabla}u_k$ is locally-Lipschitz and does not vanish at $x_0$). See Figure \ref{localrcfig}. It is also possible to choose $\vphi$ which extends continuously to $\overline{\Omega}_0\times[-\ell,\ell]$.
    
    \medskip 
    By Lemma \ref{dtlemma}, the function $u_k$ is differentiable twice almost everywhere, so we can choose the map $\vphi$ such that $u_k$ is differentiable at $\vphi(y,0)$ for every $y$ in a set $\Omega_0'\subseteq\Omega_0$ of full measure. We can furthermore take $\vphi^{-1}(x_0)$ to be arbitrarily close to $0$.

    \begin{figure}
        \centering
        \includegraphics[width = .8\textwidth]{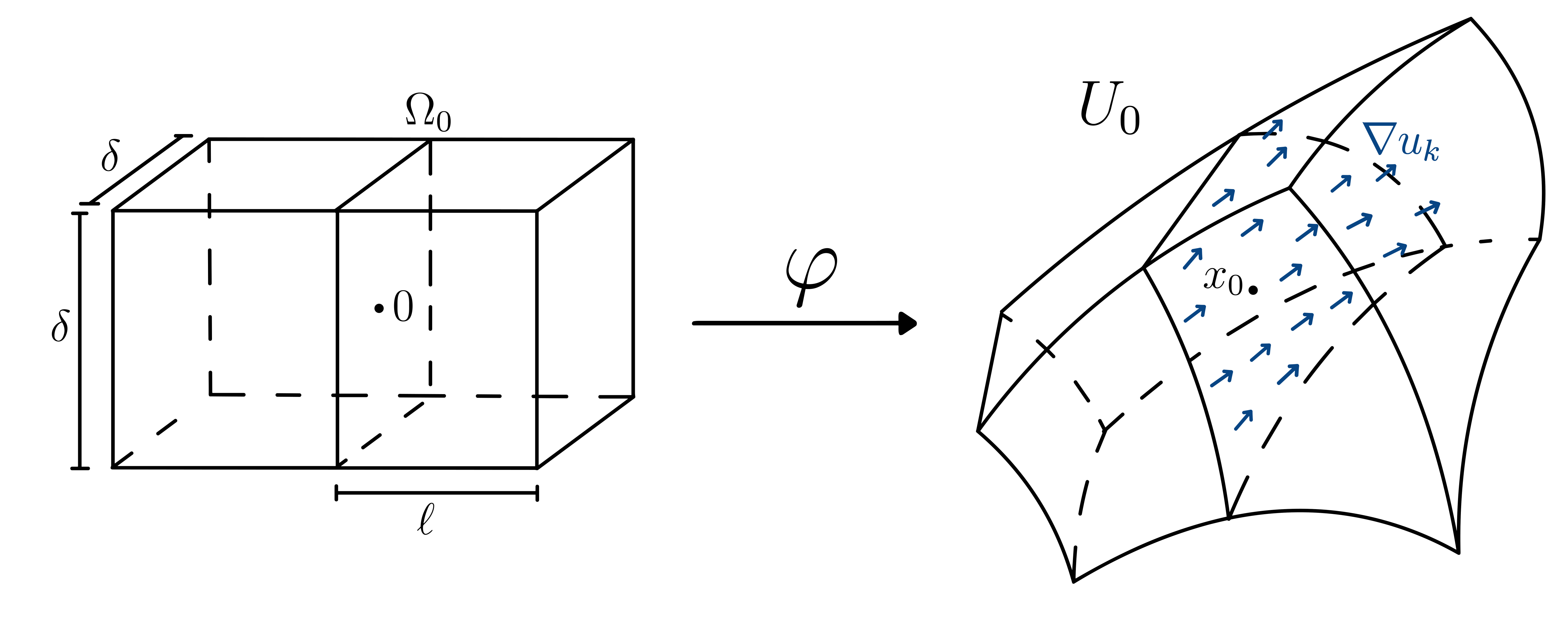}
        \caption{Proof of Lemma \ref{localclusterlemma}}
        \label{localrcfig}
    \end{figure}
    \medskip
    Set
    $$F_0(y) : = \vphi(y,0) \qquad y \in \Omega_0$$
    and 
    $$\tilde R_0 : = F_0(\Omega_0')\cap \cS_k.$$
    Properties (ii),(iv) and (v) of Definition \ref{clusterdef} hold for $\tilde R_0$ by construction, with $k_0 = k$. We claim that if $\delta$ is small enough then property (i) holds, i.e., every transport ray intersects the set $\tilde R_0$ at most once.

   \medskip
   Indeed, otherwise there exists a  sequence $\{\gamma_i\}_{i=1}^\infty$ of transport rays and sequences $s_i < t_i$ such that for each $i \ge 1$, the points $z_i:=\gamma_i(s_i)$ and $w_i:=\gamma_i(t_i)$ belong to  $\tilde R_0$, and both converge to $x_0$ as $i \to \infty$. By reparametrization, we may assume that $s_i = 0$ for all $i$. Since $\tilde R_0 \subseteq \cS_k$, we have
   $$\dot\gamma_i(0) = {\nabla}u\vert_{z_i} = {\nabla}u_k\vert_{z_i}\qquad \text{ and }\qquad \dot\gamma_i(t_i) = {\nabla}u\vert_{w_i} = {\nabla}u_k\vert_{w_i}.$$
   By continuity of the vector field ${\nabla}u_k$, it follows that 
   $$\lim_i\dot\gamma_i(0) = \lim_i\dot\gamma_i(t_i) = {\nabla}u_k\vert_{x_0} = {\nabla}u\vert_{x_0} = \dot\gamma_{x_0}(0).$$
   Lemma \ref{diamlemma} implies that $t_i$ are uniformly bounded, so by passing to a further subsequence we can assume that $t_i \to T > 0$ (the limit is positive since $\dot\gamma_i(0)$ converge to a nonzero vector). It follows that the sequence $\gamma_i\vert_{[0,t_i]}$ converges to a closed transport ray,  which is impossible since a transport ray is a minimizing extremal by Lemma \ref{transportraycalibratedlemma}, and supercriticality implies that minimizing extremals cannot be closed. This finishes  the proof of property (i) in Definition \ref{clusterdef}, provided that $\delta$ is chosen sufficiently small. 
   
   \medskip
   We need to show that we can replace $\tilde R_0$ by a set $R_0$ which is compact and verifies (iii), i.e. the functions $a,b$ are continuous on $R_0$. While this may not be the case for $\tilde R_0$, we claim that it is true for a countable collection of subsets $R_0^j \subseteq \tilde R_0$ which cover $\tilde R_0$ up to measure zero. Indeed, since the functions $a(x),b(x)$ are Borel measurable, Lusin's theorem implies that there exists a sequence $\Omega_0^1 \subseteq \Omega_0^2 \subseteq \dots \subseteq \Omega_0'$ of compact sets such that the functions $a(x),b(x)$ are continuous on $F_0(\Omega_0^j)$ for each $j$, and the union $$\Omega_0'':=\bigcup_j\Omega_0^j$$
   has full measure in $\Omega_0'$. Our construction guarantees that for every $j \ge 1$, the set
   $$R_0^j : = F_0(\Omega_0^j)\cap\cS_k$$
   is a seed of a ray cluster.
   
   \medskip
   Finally, we need to find a neighborhood $U \ni x_0$ such that the set $U\cap\cS_{k-1}$ is contained, up to measure zero, in the union of the ray clusters $$R^j : = \hat R_0^j.$$
   Set
    $$U : = \{G_k^t(F_0(y)) \, \mid \, y \in \Omega_0, \, \, t \in (-\ell,\ell)\}.$$

    Then $U$ is an open neighborhood of $x_0$, since the hypersurface $F_0(\Omega_0)$ is transverse to ${\nabla}u_k$ and the point $x_0$ was taken to be close to $F_0(\Omega_0)$. Since $\Omega_0''$ has full measure in $\Omega_0$, the set
    $$U'' : = \{G_k^t(F_0(y)) \, \mid \, y \in \Omega_0'', \, \, t \in (-\ell,\ell)\}$$
    has full measure in $U$ if $\ell$ is sufficiently small; thus it suffices to prove that 
    $$U''\cap\cS_{k-1}\subseteq \bigcup_j R^j.$$
   Let $x' \in U''\cap\cS_{k-1}$. By the definition of $U''$, we may write 
   $$x' = G_k^t(x) = \gamma_x(t), \qquad \text{ where } \qquad x = F_0(y) \qquad \text{ for some $y \in \Omega_0^j$ and some $t \in (-\ell,\ell)$.}$$
   Since $|t| < \ell < 2^{-k}$ and $\gamma_x(t) =x' \in \cS_{k-1}$, necessarily $x \in \cS_k$, whence $x \in R_0^j$ by the definition of $R_0^j$. Since $x'$ lies on a transport ray through $x$, it follows that $x' \in \hat R_0^j = R^j$. This finishes the proof of the lemma. 
\end{proof}

\begin{lemma}[cf. \text{\cite[Lemma 3.15]{Kl}}]\label{coverlemma}
    The saturation $\hat \cS$ of the strain set $\cS$ can be covered, up to a set of zero measure, by a countable collection of pairwise-disjoint ray clusters.
\end{lemma}

\begin{proof}
    Let $k \ge 1$. By Lemma \ref{localclusterlemma}, the set $\cS_k$ can be covered by a collection $\{U_x\}_{x \in \cS_k}$ of open sets, such that the intersection $U_x\cap\cS_k$ is contained, up to measure zero, in the union of countably many ray clusters. Extracting a countable subcover of $\{U_x\}_{x \in \cS_k}$, we see that $\cS_k$ can be covered, up to measure zero, by countably many ray clusters. Since $\cS = \bigcup_k\cS_k$, there exists a countable collection $\{\tilde R^j\}_{j=1}^\infty$ of ray clusters such that
    $$\mu\left(\cS\setminus\bigcup_{j=1}^\infty \tilde R^j\right) = 0.$$
    But ray clusters are transport sets; hence
    $$\mu\left(\hat \cS\setminus\bigcup_{j=1}^\infty \tilde R^k\right) = 0.$$
    Setting 
    $$R_1 : = \tilde R_1, \qquad R_j : = \tilde R_j \setminus \bigcup_{i=1}^{j-1}\tilde R_i, \quad j \ge 2$$
    and using Lemma \ref{subsetlemma} we obtain the desired cover.
\end{proof}

\section{Optimal transport}\label{transportsec}

In this chapter we discuss optimal transport with the cost $\cc$. Most of the results below are either known, or follow from a straightforward modification of the arguments in the Riemannian case. For more on optimal transport with Lagrangian costs see \cite{BB,BB07,FF,Fig,Vil}. Our main goals are to prove the needle decomposition theorem, Theorem \ref{needlethm} (restated below as Theorem \ref{needlethm2}), and to establish the equivalence (ii)$\iff$(iii) in Theorem \ref{mainthm}.

\subsection{The Monge-Kantorovich problem}\label{MKsec}

Set 
$$\cP_1(L): = \left\{\begin{array}{c}\text{Absolutely continuous Borel probability measures $\mu_0$ on $M$}\\ \text{such that $\int(|\cc(x_0,\cdot)|+|\cc(\cdot,x_0)|)d\mu_0 < \infty$ for some $x_0 \in M$}\end{array}\right\}.$$
As before, absolute continuity is with respect to any measure on $M$ with a smooth, positive density (in particular, with respect to $\mu$). For $\mu_0,\mu_1\in \cP_1(L)$, define
\begin{equation}\label{CCeq}\CC(\mu_0,\mu_1) := \inf_{\kappa}\int_{M\times M}\cc(x_0,x_1) \, d\kappa(x_0,x_1),\end{equation}
where the infimum is over all \emph{couplings} $\kappa$ between $\mu_0$ and $\mu_1$, i.e. measures on $M\times M$ such that
$$(\pi_0)_*\kappa = \mu_0 \qquad \text{ and } \qquad (\pi_1)_*\kappa = \mu_1,$$
where $\pi_i : M\times M\to M$ denote the projections onto the individual factors. 

\medskip
A classical observation due to Kantorovich \cite[Chapter VIII]{KA} is that the minimization problem \eqref{CCeq} admits a dual formulation. Fix a reference measure $\mu$ on $M$ with a smooth density. For a measurable function $f:M\to \RR$ define
\begin{equation}\label{MKeq}
        \mathrm{K}[f]:=\sup_u\int_Mufd\mu,
\end{equation}
where the supremum is over all dominated functions $u:M\to \RR$ (we will soon add conditions on $f$ which ensure that this is well-defined). A minimizer in \eqref{CCeq} is called an \emph{optimal coupling} between $\mu_0$ and $\mu_1$, and a maximizer in \eqref{MKeq} is called a \emph{Kantorovich potential} for the function $f$. 

\begin{theorem}[Kantorovich Duality \text{\cite[Theorem 5.10]{Vil}}]\label{KDthm}
    For every $\mu_0 = f_0\mu,\mu_1 = f_1\mu \in \cP_1(L)$,
    $$\CC(\mu_0,\mu_1) = \mathrm{K}[f_1-f_0].$$
    Furthermore, for every coupling $\kappa$ between $\mu_0$ and $\mu_1$ and every dominated function $u$, the following are equivalent:
    \begin{enumerate}[(i)]
        \item The coupling $\kappa$ is an optimal coupling between $\mu_0$ and $\mu_1$ and the function $u$ is a Kantorovich potential for $f_1-f_0$.
        \item For $\kappa$-a.e. $(x_0,x_1) \in M\times M$,
        $$u(x_1) - u(x_0) = \cc(x_0,x_1).$$
    \end{enumerate}
\end{theorem}

\begin{lemma}[cf. \text{\cite[Proposition 4]{BB}}]\label{MKlemma}
    Let $f : M \to \RR$ be a $\mu$-integrable function satisfying
    \begin{equation}\label{fmomenteq}\int_M\left(|\cc(x_0,\cdot)| + |\cc(\cdot,x_0)|\right)|f|\,d\mu < \infty\end{equation}
    and
    $$\int_Mfd\mu = 0.$$
    Then the supremum in \eqref{MKeq} is attained, i.e., there exists a Kantorovich potential for the function $f$.
\end{lemma}
\begin{lemma}\label{1liplemma}
    The family of dominated functions is equicontinuous on every compact subset of $M$.
\end{lemma}
\begin{proof}
    Let $A \subseteq M$ be a compact subset and let $g$ be a Riemannian metric on $M$. The function $\cc$ is continuous on $A\times A$ hence uniformly continuous, and it vanishes on the diagonal. Thus the function
    $$w(r) : = \max\{|\cc(x,y)| \, \mid \, x,y \in A, \, \, d_g(x,y) \le r\}$$
    is increasing, tends to $0$ as $r \searrow 0$, and vanishes at $0$. Here $d_g$ is the distance function induced by $g$. By definition, every dominated function satisfies
    $$|u(x) - u(y)| \le \max\{|\cc(x,y)|,|\cc(y,x)|\}| \le w(d_g(x,y)).$$
Hence the class of dominated functions admits the function $w$ as a common modulus of continuity on the set $A$, and is therefore equicontinuous on $A$     .
\end{proof}
\begin{proof}[Proof of Lemma \ref{MKlemma}]
   By Lemma \ref{1liplemma}, the collection of dominated functions is equicontinuous on compact subsets of $M$. Since $\int fd\mu = 0$, there is no loss of generality in minimizing only over dominated functions which vanish at some fixed $x_0 \in M$; this is a closed set in the uniform topology, all of whose members are pointwise smaller in absolute value than the function $|\cc(\cdot,x_0)|+|\cc(x_0,\cdot)|$. By the Arz{\'e}la-Ascoli theorem, a maximizing sequence in the optimization problem \eqref{MKeq} will have a subsequence converging uniformly on compact sets to a dominated function $u$ which, by \eqref{fmomenteq} and the dominated convergence theorem, will be a maximizer.
\end{proof}

The construction of the optimal coupling from the Kantorovich potential was carried out in the Riemannian setting in \cite{EG,CFM,FeMc} and in the Lagrangian setting in \cite{BB,Fig,FF}; we will describe it in Section \ref{interpsec}.

\subsection{Mass balance}\label{mbsec}

Let $f : M \to \RR$ be a function satisfying the assumptions of Lemma \ref{MKlemma} and let $u$ be a Kantorovich potential for $f$. The following properties of $u$ can be found in the Riemannian case in \cite{EG,FeMc,Kl}.
The proof in the Lagrangian case requires only slight modifications, and we include it here mainly for the sake of completeness. 

\begin{lemma}[Mass balance on upper transport sets, c.f. \text{\cite[Lemma 5.1]{EG} and  \cite[Lemma 4.4]{Kl}}]\label{dmblemma}
    Let $A \subseteq M$ be an upper transport set. Then
    $$\int_Afd\mu \ge 0.$$
\end{lemma}

\begin{lemma}[\text{cf. \cite[Lemma 4.3]{Kl}}]\label{udeltalemma}
    Let $A\subseteq M$ be a compact set. For $\delta > 0$ define $u_\delta : M \to \RR$ by
    \begin{equation}\label{udeltadef}
        u_\delta(x) : = \inf_{y \in M}\left[u(y) + \cc(y,x) - \delta\cdot\chi_A(y)\right].
    \end{equation}
    Then 
    \begin{enumerate}
        \item $u_\delta$ is dominated.
        \item $0 \le u -u_\delta \le \delta$.
        \item Set
        \begin{equation}\label{vdeltadef}
            v_\delta(x) : = \frac{u(x) - u_\delta(x)}{\delta}, \qquad x \in M.
        \end{equation}
        There exists a function $v : M \to [0,1]$ such that 
        \begin{equation}\label{veq}
            \lim_{\delta\searrow0}v_\delta(x) = 
        \begin{cases}
            0 & x \in M\setminus A^+\\
            v(x) & A^+\setminus A\\
            1 & x \in A. 
        \end{cases}
    \end{equation}
    \end{enumerate}
\end{lemma}
\begin{proof}
    $ $\newline
    \begin{enumerate}
        \item Let $x,x' \in M$, let $\eps > 0$ and let $y' \in M$ satisfy $$u_\delta(x') \ge u(y') + \cc(y',x') - \delta\cdot\chi_A(y') - \eps.$$
        Then
        \begin{align*}
            u_\delta(x) - u_\delta(x') & \le \inf_{y \in M}\left[u(y) + \cc(y,x) - \delta\cdot\chi_A(y)\right] - u(y') - \cc(y',x') + \delta\cdot\chi_A(y') + \eps\\
            & \le u(y') + \cc(y',x) - \delta\cdot\chi_A(y') - u(y') - \cc(y',x') + \delta\cdot\chi_A(y') + \eps\\
            & \le \cc(x',x) + \eps
        \end{align*}
        using the triangle inequality. Since $\eps$ is arbitrary, this proves that $u_\delta$ is dominated.
        \item By setting $y = x$ in \eqref{udeltadef} we see that $u_\delta(x) \le u(x) -\delta\cdot\chi_A(x) \le u(x)$. Since $u$ is dominated, $u(y) + \cc(y,x) \ge u(x)$ for every $y \in M$, whence $u_\delta(x) \ge u(x) - \delta$.
        \item Rewriting $v_\delta$ as
        \begin{equation}\label{vdeltaeq}v_\delta(x) = \sup_{y \in M}\left(\frac{u(x) - u(y) - \cc(y,x)}{\delta} + \chi_A(y)\right)\end{equation}
        and observing that since $u$ is dominated, the first summand in \eqref{vdeltaeq} is nonpositive, we see that $v_\delta(x)$ is increasing in $\delta$ for all $x \in M$, whence the limit exists, and lies in the interval $[0,1]$ by the previous claim. If $x \in A$ then by setting $y = x$ in \eqref{vdeltaeq} we see that $v_\delta(x) = 1$ for all $\delta > 0$.

        \medskip
        In order to finish the proof of the lemma, we show that if $x \notin A^+$ then $v_\delta = 0$ for all sufficiently small $\delta$. Suppose that $v_{\delta_j} > 0$ for a sequence $\delta_j\searrow 0$. Then there exists a sequence $y_j \in M$ such that 
        $$\frac{u(x) - u(y_j) - \cc(y_j,x)}{\delta_j} + \chi_A(y_j) > 0.$$
        Since the first summand in nonpositive, necessarily $y_j \in A$, and
        $$u(x) - u(y_j) - \cc(y_j,x) > -\delta_j.$$
        Since $A$ is compact, we may pass to a subsequence and assume that $y_j \to y \in A$ whence
        $$u(x) - u(y) - \cc(y,x) \ge 0.$$
        Since $u$ is dominated, equality holds, whence $x$ and $y$ lie on a common transport ray. In fact, $x$ lies on an upper transport ray whose initial point is $y$; since $y \in A$, we conclude that $x \in A^+$.\qedhere
    \end{enumerate}
\end{proof}
\begin{proof}[Proof of Lemma \ref{dmblemma}]
    First suppose that $A$ is any compact set and define $u_\delta$ by \eqref{udeltadef}. By Lemma \ref{udeltalemma}, the function $u_\delta$ is dominated and $0 \le v_\delta \le 1$, where $v_\delta$ is defined in \eqref{vdeltadef}.
   
    \medskip
    Since $u$ is a maximizer in \eqref{MKeq} and $u_\delta$ is dominated,
    $$\int_Mv_\delta f \, d\mu = \frac1\delta\int_M(
        u - u_\delta)\cdot f \, d\mu \ge 0.$$
    Let $v$ be as in \eqref{veq}. Then by the dominated convergence theorem,
    \begin{align}\label{mbcalc1}
        0 \le \int_A\lim_\delta v_\delta \cdot fd\mu = \int_{A^+\setminus A}vf d\mu + \int_Afd\mu \le \int_{A^+\setminus A}|f|d\mu + \int_Afd\mu.
    \end{align}
    Now assume that $A$ is an upper transport set. By regularity of the measure $\mu$, there exists a sequence of compact sets $A_j\subseteq A$ such that
    $$\int_{A\setminus A_j}|f|d\mu \to 0.$$
    Since $A$ is an upper transport set containing $A_j$,
    $$A\setminus A_j \supseteq A_j^+\setminus A_j \qquad \text{ for all } j \ge 1.$$
    It now follows from \eqref{mbcalc1} that
    $$0 \le \int_{A\setminus A_j}|f|d\mu + \int_{A_j}fd\mu \le 2\int_{A\setminus A_j}|f|d\mu + \int_Afd\mu \to \int_Afd\mu.$$
    This finishes the proof of the lemma.
\end{proof}

\begin{corollary}\label{mbcor}
    Let $R \subseteq M$ be a transport set. Then 
    $$\int_Rfd\mu = 0.$$
\end{corollary}

\begin{proof}
    Let $R_{\mathrm{init}}$ be the set of initial points of transport rays contained in $R$. By Lemma \ref{endslemma}, this is a set of measure zero. Since the set $R\setminus R_{\mathrm{init}}$ is an upper transport set, it thus follows from Lemma \ref{dmblemma} that
    $$\int_Rfd\mu \ge 0.$$
    Consider the Lagrangian $$\tilde L(v) : = L(-v), \qquad v \in TM.$$
    It is easy to see that a function $\tilde u$ is $\tilde L$-dominated if and only if $-\tilde u$ is dominated. Therefore the function $-u$ is a Kantorovich potential for the function $-f$ with respect to the Lagrangian $\tilde L$. Transport sets for $u$ and for $-u$ with respect to the Lagrangians $L$ and $\tilde L$, respectively coincide, since the condition \eqref{trcond} is symmetric. Applying Lemma \ref{dmblemma} to $-u$ and $-f$ then gives the reverse inequality.
\end{proof}

\subsection{Needle decomposition}

The goal of this section is to prove Theorem \ref{needlethm}. In order to restate the theorem in a more detailed form, let us introduce the notion of a needle. 
\begin{definition}\normalfont\label{needledef}~
    \begin{itemize}
        \item \underline{Needles}: A \emph{needle} is a Borel measure on $M$ of the form
        $$\gamma_*m$$
        where $\gamma:I \to M$ is a $C^2$ curve defined on a open interval $I\subseteq \RR$, and $m$ is an absolutely continuous measure on $I$ with a positive, $C^2$ density $\rho :I \to (0,\infty)$ with respect to the Lebesgue measure. 
        \item \underline{U}pp\underline{er ends}: An \emph{upper end} of the needle is a measure of the form
        $$\gamma_*\left(m\vert_{[t,\infty)}\right)$$
        for some $t \in \RR$. 
        \item \underline{Curvature-dimension}: Let $N \in (-\infty,\infty]\setminus\{1\}$ and ${K} \in \RR$. The needle will be called a  \emph{$\CD({K},N)$-needle} if the function $$\psi : = -\log\rho$$ satisfies the differential inequality
        \begin{equation}\label{CDneedleeq}-\ddot\psi + \frac{\dot\psi^2}{N-1} +  {K} \le 0.\end{equation}
        Equivalently, the needle is a $\CD({K},N)$-needle if the pair $(m,l)$ satisfies $\CD_{}({K},N)$ in the sense of Definition \ref{CDdef}, where $l$ is the Lagrangian $l = (dt^2+1)/2$.
        \item \underline{Jacobi needles}: The needle will be called a \emph{$(\mu,L)$-Jacobi needle} if for every $t \in I$ there exists a $C^3$ solution $u$ to the Hamilton-Jacobi equation $H(du) = 0$ defined in a neighborhood $U \ni \gamma(t)$, such that
        \begin{equation}\label{hamiltonianneedleeq}
            \dot\gamma = {\nabla} u\qquad \text{ and } \qquad \dot\psi = -\bL  u \qquad \text{ on } \gamma^{-1}(U).
        \end{equation}
    \end{itemize}
\end{definition}
Recall that we say that the pair $(\mu,L)$ satisfies $\CD_{}({K},N)$ if every local solution $u$ to the Hamilton-Jacobi equation $H(du) = 0$ satisfies
\begin{equation}\label{CDBocheq}(d\bL u)(\nabla u) + \frac{(\bL u)^2}{N-1} + {K} \le 0.\end{equation}
We will always take $u$ to be at least $C^3$, which is possible by Lemma \ref{characlemma}.
\begin{lemma}\label{CDlemma}
    Let ${K} \in \RR$ and $N \in (-\infty,\infty]\setminus[1,n)$. The following conditions are equivalent:
    \begin{enumerate}[$(i)$]
        \item The pair $(\mu,L)$ satisfies $\CD_{}({K},N)$.
        \item Every $(\mu,L)$-Jacobi needle is a $\CD({K},N)$-needle.
    \end{enumerate}
\end{lemma}
\begin{proof}
    This is an immediate consequence of Definitions \ref{CDdef} and \ref{needledef}. If the pair $(\mu,L)$ satisfies $\CD(K,N)$ and $\gamma_*m$ is a $(\mu,L)$-Jacobi needle with the density of $m$ given by $\rho = e^{-\psi}$, then \eqref{CDneedleeq} follows from \eqref{hamiltonianneedleeq} and \eqref{CDBocheq}. Conversely, if every $(\mu,L)$-Jacobi needle is a $\CD({K},N)$-needle, and $u$ is a local solution of the Hamilton-Jacobi equation, then every measure of the form $\gamma_*m$, where $\gamma$ is an integral curve of $\nabla u$ and the measure $m$ has density $\rho = e^{-\psi}$ for a function $\psi$ satisfying $\dot\psi = -\bL\circ\gamma$, is a $(\mu,L)$-Jacobi needle, so \eqref{CDneedleeq} implies \eqref{CDBocheq}.
\end{proof}

We can now state the main result of this section which, together with Lemma \ref{CDlemma}, implies Theorem \ref{needlethm}. 
 
\begin{theorem}[Needle decomposition]\label{needlethm2}
    Let $u : M \to \RR$ be a dominated function. There exists a collection $\{\mu_\alpha\}_{\alpha \in \sA}$ of Borel measures on $M$ and a measure $\nu$ on the set $\sA$ such that the following hold:
    \begin{enumerate}[(i)]
        \item The collection $\sA$ is a disjoint union
        $$\sA = \sN\sqcup\sD,$$
        where for each $\alpha \in \sD$ the measure $\mu_\alpha$ is a Dirac measure at a point $x_\alpha \in M$, and for each $\alpha \in \sN$ the measure $\mu_\alpha$ is a $(\mu,L)$-Jacobi needle:
        $$\mu_\alpha = (\gamma_\alpha)_*m_\alpha,$$
        where 
        $$\gamma_\alpha : I_\alpha \to M$$
        is a calibrated curve of $u$ and $m_\alpha$ is a measure on $I_\alpha = (a_\alpha,b_\alpha)$. For $\mu$-a.e. strain point $x$ of the function $u$, there exists $\alpha \in \sN$, depending measurably on $x$, such that the curve $\gamma_\alpha$ passes through $x$.
        \item The measure $\mu$ disintegrates as
        $$\mu = \int_{\sA}\mu_\alpha\,d\nu(\alpha)$$
        in the following sense: for every Borel function $h : M \to \RR$, the function $\alpha\mapsto \int_Mhd\mu_\alpha$ is $\nu$-measurable, and
        \begin{equation}\label{disinteq2}\int_Mh\,d\mu = \int_{\sA}\left(\int_Mhd\mu_\alpha\right)d\nu(\alpha).\end{equation}
    \end{enumerate}
    Suppose further that $u$ is a Kantorovich potential for some function $f$ satisfying the hypotheses of Lemma \ref{MKlemma}. Then 
    \begin{itemize}
        \item[(iii)] For $\nu$-almost every $\alpha \in \sA$,
            \begin{equation}\label{MBeq2}
                \int_Mfd\mu_\alpha = 0,
            \end{equation}
            and for $\nu$-almost every $\alpha \in \sN$,
            \begin{equation}\label{detailedmbeq2}
                \int_Mfd\mu_{\alpha,t} \ge 0
            \end{equation}
            for every upper end $\mu_{\alpha,t} = (\gamma_\alpha)_*\left(m_\alpha\vert_{[t,\infty)}\right)$ of the needle $\mu_\alpha$.
        \item[(iv)] $\mu$-almost every $x$ in the support of $f$ is a strain point of $u$.
    \end{itemize}
\end{theorem}

\begin{proof}
The idea of the proof is as follows. The needles in the decomposition will be supported on transport rays of the function $u$. Each such transport ray is contained in some ray cluster, and the density of the corresponding needle will be determined by the Jacobian of the parametrization constructed in Proposition \ref{flattenprop}. Points outside the strain set will be assigned Dirac measures. In case $u$ is a Kantorovich potential, The mass balance conditions \eqref{MBeq2} and \eqref{detailedmbeq2} will follow from the results of Section \ref{mbsec}. We now proceed with the proof.

\medskip
By Lemma \ref{coverlemma}, the saturation $\hat \cS$ of the strain set $\cS$ is covered, up to a set of measure zero, by a countable collection $\{R^j\}_{j=1}^\infty$ of disjoint ray clusters. By Proposition \ref{flattenprop}, we can parametrize each ray cluster $R^j$ by a locally-Lipschitz, one-to-one map $F^j : B^j \to R^j$ where $B^j\subseteq \RR^{n-1}\times\RR$ takes the form
$$
    B^j = \{(y,t) \in \RR^{n-1}\times \RR \, \mid \, y \in B_0^j, \quad t \in (a_y,b_y)\}.
$$
We can also assume that $F^j$ is onto, since replacing $R^j$ by $F^j(B^j)$ only changes it by a null set, see \eqref{nulleq1}.  Moreover, every point in the complement
$$D : = M \setminus\hat \cS$$ constitutes a degenerate ray cluster. Thus, if we denote by $\sA$ the collection of transport rays which are either degenerate, or are contained in one of the ray clusters $\{R^j\}$, then almost every point in $M$ lies in some transport ray $\alpha \in \sA$:

\begin{equation}\label{fullmeasureeq}
    \mu\left(M\setminus\bigcup_{\alpha \in \sA}\alpha\right) =  \mu\left(M\setminus\left[\bigcup_{j=1}^\infty R^j \cup D\right]\right) = 0.
\end{equation}
In particular, $\mu$-almost every $x \in \cS$ lies on  $\gamma_\alpha$ for some $\alpha \in \sN$.

\medskip
For each transport ray $\alpha \in \sA$ we now define a measure $\mu_\alpha$ supported on $\alpha$. If $\alpha$ is degenerate, i.e. $\alpha = \{x_\alpha\}$ for some $x_\alpha \in M$, then we define $\mu_\alpha$ to be a Dirac measure at the point $x_\alpha$: 
\begin{equation}\label{deltaeq}\mu_\alpha = \delta_{x_\alpha}.\end{equation}
Suppose now that $\alpha$ is nondegenerate, in which case it is contained in some ray cluster $R^j$. There exists a unique $y\in B_0^j$ such that the relative interior of the transport ray $\alpha$ is the image of the calibrated curve
$$\gamma_\alpha(t) := F^j(y,t), \qquad t \in I_\alpha : = (a_{y},b_{y}).$$
We define the measure $\mu_\alpha$ to be
$$\mu_\alpha : = (\gamma_\alpha)_*m_\alpha,$$
where 
$m_\alpha$ is the measure on the interval $I_\alpha$ whose density is
\begin{equation}\label{rhoalphadef}
    \rho_\alpha(t):=\textstyle\det_\mu dF^j(y,t), \qquad t \in I_\alpha,
\end{equation}
with the Jacobian $\det_\mu dF^j$ defined as in Proposition \ref{flattenprop}. Thus, for every Borel measurable function $h : M \to \RR$,
\begin{equation}\label{mualphaeq}\int_Mh\,d\mu_\alpha : = \int_{a_y}^{b_y}h(F^j(y,t))\cdot \textstyle\det_\mu dF^j(y,t)\,dt.\end{equation}

Let $\sN$ and $\sD$ denote the nondegenerate and degenerate transport rays in $\sA$, respectively. For each $j \ge 1$ define a map 
$$\alpha_j : B_0^j \to \sN$$ assigning to each $y \in B_0^j$ the transport ray $\alpha_j(y) = F^j(y,\cdot) \in \sN$. Moreover, define
$$\alpha_0 : {D} \to \sD$$ by assigning to each $x \in {D}$ the degenerate transport ray $\alpha_0(x) : = \{x\} \in \sD$. In other words, $\alpha_0(x_\alpha) = \alpha$ for all $\alpha \in \sD$. The measure $\nu$ is then defined by
$$\nu := \sum_{j=1}^\infty(\alpha_j)_*\left(\Vol_{n-1}\vert_{B_0^j}\right) + (\alpha_0)_*\left(\mu\vert_{{D}}\right),$$
where $\Vol_{n-1}$ denotes the Lebesgue measure on $\RR^{n-1}$. In more detail: the measure $\nu$ is defined on the $\sigma$-algebra of sets whose preimage under $\alpha_j$ is Borel measurable for every $j \ge 0$ (in $\RR^{n-1}$ if $j \ge 1$, or in $M$ if $j=0$); and for every $\nu$-measurable function $\phi:\sA\to \RR$ we set
\begin{equation}\label{nueq}\int_{\sA}\phi\, d\nu : = \sum_{j=1}^\infty\int_{B_0^j}(\phi\circ \alpha_j)\,d\Vol_{n-1} + \int_{{D}}(\phi\circ\alpha_0)\,d\mu.\end{equation}

Note that, since each $B_0^j$ is bounded, the measures $\Vol_{n-1}\vert_{B_0^j}$ are finite. Note also that by construction the maps $\alpha_j$ are measurable. By \eqref{fullmeasureeq},  \eqref{deltaeq},  \eqref{mualphaeq}, \eqref{nueq} and the change of variables formula \eqref{COVeq}, for every Borel function $h : M \to \RR$,
\begin{align*}
    \int_Mhd\mu & = 
    \sum_{j=1}^\infty\int_{R^j}h \, d\mu + \int_Dh\,d\mu
    \\ & = \sum_{j=1}^\infty\int_{B_0^j}\int_{a_y}^{b_y}h(F^j(y,t)) \cdot {\textstyle\det_\mu dF^j(y,t)}\,dt\,dy + \int_{{D}}h(x)\,d\mu(x)\\
    & = \sum_{j=1}^\infty\int_{B_0^j}\left(\int_Mhd\mu_{\alpha_j(y)}\right)dy + \int_{{D}}\left(\int_Mhd\mu_{\alpha_0(x)}\right)d\mu(x)\\
    & = \int_{\sA}\left(\int_Mhd\mu_\alpha\right)d\nu(\alpha),
\end{align*}
which proves \eqref{disinteq2}.

\begin{lemma}
    For $\nu$-almost every $\alpha \in \sN$, the measure $\mu_\alpha$ is a $(\mu,L)$-Jacobi needle.
\end{lemma}

\begin{proof}
    By \eqref{rhoalphadef} and Proposition \ref{flattenprop}, for almost every $\alpha \in \sN$, if we write
    $$\rho_\alpha = e^{-\psi_\alpha}$$
    then
    \begin{equation}\label{needlecond2}
        \dot\psi_\alpha = -(\bL u_k)\circ\gamma_\alpha \qquad \text{ and }\qquad \dot\gamma_\alpha = {\nabla}u_k \qquad \text{on $I_{\alpha,k}$}
    \end{equation}
    for all $k$ sufficiently large. Recall that $I_\alpha = \bigcup_{k}I_{\alpha,k}$. It thus remains to show that for every $t_0 \in I_{\alpha,k}$, the function $u_k$ can be replaced by a locally defined $C^3$ solution $\tilde u$ to the Hamilton-Jacobi equation in such a way that \eqref{needlecond2} still holds on $\gamma_\alpha^{-1}(U)$ for some neighborhood $U \ni \gamma_\alpha(t_0)$. By Lemma \ref{characlemma}, there exists a local $C^3$ solution $\tilde u : U \to \RR$ to the Hamilton-Jacobi equation with the second-order data \eqref{needlecond2} at the point $$x_0:=\gamma_\alpha(t_0) \in U.$$
    The graph of $\tilde u$ is a $C^2$ Lagrangian submanifold $\tilde{S} \subseteq S^*M$ passing through the covector
    $$p_0 : = du_k\vert_{x_0}$$  and satisfying
	$$T_{p_0}{S} = T_{p_0}\tilde{S}$$
    where ${S}\subseteq S^*M$ is the graph of the one-form $du_k$. Recall that for every  $x \in \cS_{k-1}\cap U$,
    $$du_k\vert_{\gamma_x(t)} = \Phi_{t}^Hdu_k\vert_x, \qquad t \in (-2^{-k},2^{-k}).$$
    Since $x_0$ is a Lebesgue density point of $\cS_{k-1}$, and $du_k$ is differentiable at $\gamma_\alpha(t)$ for all $t \in (-2^{-k},2^{-k})$, it follows that
		$$T_{p_t}{S} = (d\Phi_t)\left(T_{p_0}{S}\right), \qquad t \in (-2^{-k},2^{-k}), \qquad \text{where} \qquad p_t : = \Phi_t^H(p_0).$$
	Since $\tilde u$ solves the Hamilton-Jacobi equation, the integral curves of ${\nabla}\tilde u$ are extremals. Thus, applying the same analysis to $d \tilde u$, we conclude that $p_t \in \tilde {S}$ for all $t \in (-2^{-k},2^{-k})$, and
		$$T_{p_t}\tilde{S} = (d\Phi_t)(T_{p_0}\tilde{S}) = (d\Phi_t)(T_{p_0}{S}) = T_{p_t}{S} \qquad t \in (-2^{-k},2^{-k}),$$
	that is, the graphs of $d\tilde u$ and $du_k$ are tangent at $p_t$. This implies that 
    $$(\bL \tilde u)(\gamma_{\alpha}(t)) = (\bL u_k)(\gamma_{\alpha}(t)) \qquad \text{ for all } \qquad t \in (-2^{-k},2^{-k})$$
    and completes the proof of the lemma.
\end{proof}

\begin{remark*}\normalfont
    Observe that, in particular, we have improved the smoothness of the density $\rho_\alpha$ of the needle $\mu_\alpha$ from $C^1$ to $C^2$.
\end{remark*}

We have established clauses (i) and (ii) of Theorem \ref{needlethm2}. Assume now that $u$ is a Kantorovich potential for the function $f$. Let $\sA' \subseteq \sA$ be a $\nu$-measurable set. Then, being a Borel measurable union of transport rays, the set $A = \bigcup_{\alpha \in \sA'}\alpha$ is a transport set. Lemma \ref{dmblemma} and Corollary \ref{mbcor} together with the disintegration formula \eqref{disinteq2} imply that
$$\int_{\sA'}\left(\int_Mfd\mu_\alpha\right)d\nu(\alpha) = \int_Afd\mu = 0.$$
Since $\sA'$ is an arbitrary $\nu$-measurable subset of $\sA$, this proves \eqref{MBeq2}. 

\medskip
By the definition of the measures $\mu_\alpha$ and $\nu$, in order to prove \eqref{detailedmbeq2}, it suffices to show that for every $j \ge 1$, almost every $y \in B_0^j$ satisfies
$$\int_t^{b_y}f(F(y,t))\textstyle\det_\mu dF(y,t)dt \ge 0 \qquad \text{for all $t \in (a_y,b_y)$}.$$
Suppose that some set $Y \subseteq B_0^j$ violates this condition, i.e., for every $y \in Y$ there exists $t_y\in(a_y,b_y)$ such that 
\begin{equation}\label{massbalanceviolationeq}\int_{t_y}^{b_y}f(F(y,t))\textstyle\det_\mu dF(y,t)dt < 0.\end{equation}
We can take $t_y$ to depend measurably on $y$, and then the set 
$$R^+ : = \{F(y,t) \, \mid \,  y \in Y, \, t \in (t_y,b_y)\}$$
is an upper transport set, since each set $\{F(y,t) \, \mid \, t \in (t_y,b_y)\}$ is an upper transport ray. Lemma \eqref{dmblemma} and the change of variables formula \eqref{COVeq} imply  that
$$\int_Y\int_{t_y}^{b_y}f(F(y,t)){\textstyle\det_\mu dF(y,t)}dt= \int_{R^+}fd\mu\ge 0.$$
It follows from \eqref{massbalanceviolationeq} that the set $Y$ has measure zero. This proves \eqref{detailedmbeq2}.

\medskip
It remains to prove (iv). Since $\mu(\hat \cS\setminus \cS) = 0$ by Lemma \ref{endslemma}, it suffices to prove that $f =  0$ almost everywhere on $M\setminus \hat \cS$. But every $x \in M \setminus\hat \cS$ takes the form $x = x_\alpha$ for some degenerate transport ray $\alpha \in \sD$, and then by the mass-balance condition \eqref{MBeq2}, 
$$0 = \int_Mfd\mu_\alpha = f(x_\alpha).$$
The proof of Theorem \ref{needlethm2} is complete.
\end{proof}

\subsection{Displacement interpolations}\label{interpsec}

Displacement interpolations are cost-minimizing paths in the space of probability measures. When the cost is defined using a fixed time interval, the displacement interpolation is unique \cite[Chapter 7]{Vil}. The cost we use here, in contrast, allows for a free time interval, and the resulting theory extends optimal transport theory with distance cost (rather that squared distance, or distance to the power $q$ for $q>1$), in which displacement interpolations are not unique. This non-uniqueness is, however, of a simple kind - it stems from the fact that the one-dimensional transport problem on each needle may admit multiple solutions. Once a certain monotonicity is imposed on the solution, the displacement interpolation becomes unique. We single out the displacement interpolation satisfying this monotonicity property, as was done in \cite{FeMc}. In the next section, in analogy with the metric theory, we will prove that the $\CD(K,N)$ condition is characterized by convexity properties of the entropy functional $\rS_N$ along this interpolation.

\medskip
Recall that 
$$SM\times_0[0,\infty) : = \faktor{SM\times[0,\infty)}{\sim}$$
where $(v,0)\sim(v',0)$ whenever $\pi(v) = \pi(v')$. For $\lambda \in [0,1]$, define 
$$\Exp_\lambda:SM\times_0[0,\infty) \to M, \qquad \Exp_\lambda(v,\ell) = \pi\left(\Phi_{\lambda \ell}^Lv\right),$$
where $\Phi_t^L$ is the Euler-Lagrange flow. Thus $\Exp_\lambda(v,\ell) = \gamma(\lambda \ell)$, where $\gamma:[0,\ell]\to M$ is a zero-energy extremal satisfying $\dot\gamma(0) = v$. Note that $\Exp_\lambda$ respects the equivalence relation; in fact,
$$\Exp_\lambda(v,0) = \Exp_0(v,\ell) = \pi(v), \qquad v \in SM,\,\,\ell\ge 0, \,\, \lambda\in[0,1].$$
If the Euler-Lagrange flow is not complete then we take $\Exp_\lambda$ to be defined on a maximal open subset of $SM\times_0[0,\infty)$.

\begin{definition}[Transport plans and displacement interpolations]\normalfont
    A \emph{transport plan} is a finite Borel probability measure on $SM\times_0[0,\infty)$. A family $\{\mu_\lambda\}_{0 \le \lambda \le 1}$ of Borel probability measures on $M$ will be called a \emph{displacement interpolation} between $\mu_0$ and $\mu_1$ if there exists a transport plan $\Pi$ such that
    \begin{equation}\label{mulambdadef}\mu_\lambda = (\Exp_\lambda)_*\Pi, \qquad \lambda \in [0,1],\end{equation}
    and if for every $0 \le \lambda \le \lambda' \le 1$, the measure
    \begin{equation}\label{kappadef2}\kappa_{\lambda,\lambda'}:= (\Exp_\lambda\times\Exp_{\lambda'})_*\Pi\end{equation}
    on $M\times M$ is an optimal coupling between $\mu_\lambda$ and $\mu_{\lambda'}$.
\end{definition}

Thus a transport plan is a random choice of initial velocity $v$ and time interval $\ell$ of a minimizing extremal $\gamma$, and a family $\{\mu_\lambda\}_{\lambda \in [0,1]}$ is a displacement interpolation if the law of the random point $\gamma(\lambda\ell)$ is the measure $\mu_\lambda$ and the resulting coupling between $\mu_\lambda$ and $\mu_{\lambda'}$ is optimal for all $0 \le \lambda \le \lambda' \le 1$.

\medskip
The following lemma, regarding uniqueness of a monotone optimal coupling between measures on the real line, is well-known, see \cite{Mc}, \cite[Section 3]{Amb}, \cite[Theorem 28]{FeMc} or \cite[Chapter 2]{Santa}. 
\begin{lemma}[Monotone displacement interpolations on the real line]\label{1Dintlemma}
    Let $m_0,m_1$ be absolutely continuous measures on the real line satisfying
    \begin{equation}\label{1Dmbeq}m_0(\RR) = m_1(\RR)<\infty.\end{equation}
    There is a unique coupling $\kappa$ between $m_0$ and $m_1$ which is monotone in the sense that
    $$(t_0' - t_0)(t_1' - t_1) \ge 0 \qquad \text{for $\kappa\otimes\kappa$-almost every $((t_0,t_1),(t_0',t_1')) \in (\RR\times\RR)\times (\RR\times\RR)$.}$$
    It is given by 
    $$\kappa = (\mathrm{Id}\times\rT)_*m_0$$
    where
    $$\rT:\supp m_0\to\supp  m_1$$
    is an increasing map satisfying
    \begin{equation}\label{m0tinftyeq}
        m_0([t,\infty)) = m_1([\rT(t),\infty)) \qquad \text{ for all } \quad  t \in \supp m_0.
    \end{equation}
    If we define 
    \begin{equation}\label{Tlambdadef}\rT_\lambda(t) : = (1-\lambda)t + \lambda\rT(t), \qquad t \in \supp m_0, \,\, \lambda \in [0,1],\end{equation}
    then the family 
    $$\left\{m_\lambda = (\rT_\lambda)_*m_0\right\}_{\lambda \in [0,1]}$$
    is a displacement interpolation between $m_0$ and $m_1$ with respect to the Euclidean squared distance cost, and $m_\lambda$ is absolutely continuous for all $0 \le \lambda \le 1$. 

    \medskip
    If, in addition, the measures $m_0,m_1$ satisfy
    $$m_0([t,\infty)) \le m_1([t,\infty)) \qquad \text{ for all $t \in \RR$,}$$
    Then the coupling $\kappa$ also satisfies
    $$t_0 \le t_1 \qquad \text{for $\kappa$-almost every $(t_0,t_1) \in \RR\times \RR$.}$$
\end{lemma}
\begin{definition}\normalfont
    The coupling $\kappa$ will be called \emph{the monotone coupling} between $m_0$ and $m_1$, and the family $m_\lambda$ will be called \emph{the monotone displacement interpolation} between $m_0$ and $m_1$.
\end{definition}

\begin{remark*}\normalfont
    In Lemma \ref{1Dintlemma} we allow the measures $m_i$ to have mass different from $1$, since in Proposition \ref{interprop} below we can only guarantee that the measures $m_{\alpha,i}$ on the needles have the same (finite) mass, but not that they are probability measures. All the definitions and basic facts about one-dimensional optimal transport extend to this setting in an obvious way.
\end{remark*}

\begin{proposition}[Distinguished displacement interpolations]\label{interprop}
    Let $\mu_0=f_0\mu,\mu_1=f_1\mu \in \cP_1(L)$, let $u$ be a Kantorovich potential for the function $f = f_1-f_0$, and let $\mu = \int_{\sA}\mu_\alpha d\nu(\alpha)$ be the corresponding disintegration of measure provided by Theorem \ref{needlethm2}. There exists a displacement interpolation $\{\mu_\lambda\}_{0 \le \lambda \le 1}$ between $\mu_0$ and $\mu_1$ with the following properties:
    \begin{itemize}
        \item The measure $\mu_\lambda$ is absolutely continuous for all $\lambda \in [0,1]$:
        $$\mu_\lambda = f_\lambda\mu.$$
        \item For every $\alpha \in \sN$, the family 
        $$m_{\alpha,\lambda}:=(f_\lambda\circ\gamma_\alpha)m_\alpha$$
        is the monotone displacement interpolation between the measures $m_{\alpha,0}$ and $m_{\alpha,1}$ on the interval $I_\alpha$.
        \item For every $\alpha \in \sD$, the function $\lambda\mapsto\int_Mf_\lambda \,d\mu_\alpha = f_\lambda(x_\alpha)$ is constant.
    \end{itemize}
\end{proposition}

\begin{remark*}\normalfont
    Since $\mu_0,\mu_1 \in \cP_1(L)$, the function $f = f_1 - f_0$ satisfies the moment condition \eqref{fmomenteq}, so a Kantorovich potential $u$ indeed exists.
\end{remark*}

\begin{proof}
    We use the notation of Theorem \ref{needlethm2}. For every $\alpha \in \sN$, the curve $\gamma_\alpha$ is calibrated, i.e.
    \begin{equation}\label{gammaalphacalibeq}u(\gamma_\alpha(t')) - u(\gamma_\alpha(t)) = \cc(\gamma_\alpha(t),\gamma_\alpha(t')), \qquad a_\alpha < t \le t' < b_\alpha.\end{equation}
    We also have mass balance: for every $\alpha \in \sN$,
    $$\int_t^{b_\alpha}(f_1\circ\gamma_\alpha- f_0\circ\gamma_\alpha)dm_\alpha \le 0, \qquad t \in (a_\alpha,b_\alpha),$$
    with equality when $t = b_\alpha$, and for every $\alpha \in \sD$,
    \begin{equation}\label{sDmbeq}f_0(x_\alpha) = f_1(x_\alpha).\end{equation}
    For every $\alpha \in \sN$, let
    $$m_{\alpha,0}=(f_0\circ\gamma_\alpha)m_\alpha \qquad \text{ and } \qquad m_{\alpha,1}=(f_1\circ\gamma_\alpha)m_\alpha.$$
    The mass balance condition implies that hypothesis \eqref{1Dmbeq} in Lemma \ref{1Dintlemma} is satisfied. Let  $$\kappa_\alpha = (\mathrm{Id}\times\rT_\alpha)_*m_{\alpha,0} \qquad \text{ and } \qquad m_{\lambda,\alpha} : = (\rT_{\alpha,\lambda})_*m_{\alpha,0} \qquad \lambda \in [0,1]$$
    be the monotone coupling and the monotone displacement interpolation between the measures $m_{0,\alpha}$ and $m_{1,\alpha}$, respectively. The measure $m_{\alpha,\lambda}$ is absolutely continuous for all $0 \le \lambda \le 1$, i.e. there exists a Borel function $f_{\alpha,\lambda}:I_\alpha \to \RR$ such that
    $$m_{\alpha,\lambda} = f_{\alpha,\lambda}m_\alpha.$$
    For $\alpha \in \sD$, we trivially define the map $\rT_{\alpha,\lambda}:\{0\}\to\{0\}$ by $\rT_{\alpha,\lambda}(0) := 0$ for all $0 \le \lambda \le 1$ and the coupling $\kappa_\alpha$ by $\kappa_\alpha := \delta_{(0,0)}$.

    \medskip
    Define a transport plan $\Pi$ by
    $$\Pi : = \int_{\sA}\Pi_\alpha\,d\mu(\alpha),$$
    where $\Pi_\alpha$ is the pushforward of $\kappa_\alpha$ via the map $(t,t')\mapsto(\dot\gamma_\alpha(t),t'-t) \in SM\times_0[0,\infty)$. That is, 
    for every measurable function $\phi$ on $SM\times_0[0,\infty)$,
    \begin{equation}\label{Pidef}\int_{SM\times_0[0,\infty)}\phi(v,\ell) \,d\Pi(v,\ell) : = \int_{\sA}\left(\int_{I_\alpha\times I_\alpha}\phi(\dot\gamma_\alpha(t),t'-t)\,d\kappa_\alpha(t,t')\right) \,d\nu(\alpha)\end{equation}
    (with the understanding that $\dot\gamma_\alpha = 0$ if $\alpha \in \sD$). Define measures $\mu_\lambda$ by \eqref{mulambdadef} and couplings $\kappa_{\lambda,\lambda'}$ by \eqref{kappadef2}:
    $$\mu_\lambda = (\Exp_\lambda)_*\Pi, \qquad  \text{ and } \qquad \kappa_{\lambda,\lambda'}:= (\Exp_\lambda\times\Exp_{\lambda'})_*\Pi, \qquad 0 \le \lambda \le \lambda' \le 1.$$
    By clauses (iv) and (i) in Theorem \ref{needlethm2}, almost every $x \in M$ is contained in the image of a unique transport ray $\gamma_\alpha$, and $\alpha$ depends measurably on $x$. Thus there exist $\mu$-a.e.-defined  measurable functions $f_\lambda : M \to \RR$ such that
    \begin{equation}\label{flambdadef}f_\lambda(\gamma_\alpha(t)) = f_{\alpha,\lambda}(t), \qquad \alpha \in \sN, \,\, t \in I_\alpha, \, \, \lambda \in [0,1].\end{equation}
    If $\alpha \in \sD$ then we set $f_\lambda(x_\alpha) = f_0(x_\alpha)$ for all $0 \le \lambda\le 1$, which is consistent with \eqref{sDmbeq}.

    \medskip
    Since each $\gamma_\alpha$ is an extremal, for every $\alpha \in \sN$ and every $a_\alpha < t \le t' < b_\alpha$,
    \begin{equation}\label{Explambdadotgammaalphaeq}\Exp_\lambda(\dot\gamma_\alpha(t),t'-t) = \gamma_\alpha(t + \lambda(t'-t)), \qquad \lambda \in [0,1].\end{equation}
    Thus, for every Borel set $A \subseteq M$,
    \begin{align*}
        \mu_\lambda(A)  = \Pi(\Exp_\lambda^{-1}(A))& = 
        \int_{SM\times_0[0,\infty)}\chi_A(\Exp_\lambda(v,\ell))\,d\Pi(v,\ell)\\
        & = \int_{\sA}\left(\int_{I_\alpha\times I_\alpha}\chi_A(\gamma_\alpha(t + \lambda(t'-t)))\,d\kappa_\alpha(t,t')\right)\,d\nu(\alpha)\\
        & = \int_{\sA}\left(\int_{I_\alpha}\chi_A(\gamma_\alpha(t + \lambda(\rT_{\alpha}(t)-t)))\,dm_{\alpha,0}(t)\right)\,d\nu(\alpha)\\
        & = \int_{\sA}\left(\int_{I_\alpha}\chi_A(\gamma_\alpha(\rT_{\alpha,\lambda}(t)))\,dm_{\alpha,0}(t)\right)\,d\nu(\alpha)\\
        & = \int_{\sA}\left(\int_{I_\alpha}\chi_A(\gamma_\alpha(t))\,dm_{\alpha,\lambda}(t)\right)\,d\nu(\alpha)\\
        & = \int_{\sA}\left(\int_{I_\alpha}\chi_A(\gamma_\alpha(t))f_{\alpha,\lambda}(t)\,dm_\alpha(t)\right)\,d\nu(\alpha)\\
        & = \int_{\sA}\left(\int_{I_\alpha}\chi_A(\gamma_\alpha(t))f_\lambda(\gamma_\alpha(t))\,dm_\alpha(t)\right)\,d\nu(\alpha)\\
        & = \int_{\sA}\left(\int_Af_\lambda d\mu_\alpha\right)\,d\nu(\alpha)\\
        & = \int_Af_\lambda \, d\mu.
    \end{align*}
    Here $\chi_A$ is the indicator set of $A$. In particular, $\mu_\lambda$ is absolutely continuous for all $0 \le \lambda\le 1$, and the measures $\mu_0$ and $\mu_1$ are indeed $f_0\mu$ and $f_1\mu$ respectively.
    
    \medskip
    It remains to prove that for every $0 \le \lambda \le \lambda' \le 1$, the coupling $\kappa_{\lambda,\lambda'}$ between $\mu_\lambda$ and $\mu_{\lambda'}$ is optimal. By Theorem \ref{KDthm}, it suffices to prove that 
    $$u(x_{\lambda'}) - u(x_{\lambda}) = \cc(x_{\lambda},x_{\lambda'}) \qquad \text{ for $\kappa_{\lambda,\lambda'}$-a.e. $(x_{\lambda},x_{\lambda'}) \in M\times M$}.$$
    By \eqref{Explambdadotgammaalphaeq} and the definitions of $\kappa_{\lambda,\lambda'}$ and $\Pi$, this is equivalent to
    $$u(\gamma_\alpha(t + \lambda(t'-t))) - u(\gamma_\alpha(t + \lambda'(t'-t))) = \cc(\gamma_\alpha(t+\lambda(t'-t)),\gamma_\alpha(t+\lambda'(t'-t)))$$
    for $\kappa_\alpha$-a.e. $(t,t') \in I_\alpha\times I_\alpha$, for $\nu$-a.e. $\alpha \in \sN$. But this is true by \eqref{gammaalphacalibeq} and the fact that $t' \ge t$ for $\kappa_\alpha$-a.e.$(t,t') \in I_\alpha\times I_\alpha$.
\end{proof}

\subsection{Displacement convexity of entropy}\label{convsec}

In this section we prove the equivalence (ii)$\iff$(iii) in Theorem \ref{mainthm}. Recall that for $\mu_0=f_0\mu \in \cP_1(L)$ we set
$$\rS_N[\mu_0|\mu] = -\int_Mf_0^{-1/N}d\mu_0 \qquad \text{ and } \qquad \rS_\infty[\mu_0|\mu] = \int_M\log f_0\,d\mu_0.$$

\begin{definition}[Distortion coefficients, see \cite{CMS,Sturm2,Vil}]\label{distcoeffdef}\normalfont
    For ${K} \in \RR, N > 1$, $\ell \ge 0$ and $t \in [0,1]$, define the \emph{distortion coefficients} $$\tau^{{K},N}_t,\sigma^{{K},N}_t,\beta^{{K},N}_t:[0,\infty)\to [0,\infty)$$ by:
    \begin{align*}
    \sigma_t^{{K},N}(\ell) &: =
    \frac{\sin\left(t\ell\sqrt{\frac{{K}}{N - 1}}\right)}{\sin\left(\ell\sqrt{\frac{{K}}{N - 1}}\right)}\\\noalign{\vskip15pt}
    \tau_t^{{K},N}(\ell) : = t^{1/N}\cdot\sigma_t^{{K},N}(\ell)^{1-1/N}&\qquad \text{ and } \qquad
    \beta_t^{{K},N}(\ell) : = 
        t^{1-N}\cdot\sigma_t^{{K},N}(\ell)^{N - 1}
    \end{align*}
    with the following conventions: 
    \begin{itemize}
        \item If ${K} > 0$ and $\ell\sqrt{{K}/(N - 1)} \ge \pi$ then $\sigma_t^{{K},N}(\ell) = \infty$ for all $t\in[0,1]$.
        \item If ${K} = 0$ then we set $\sigma_t^{0,N} \equiv \tau_t^{0,N} \equiv t$ and $\beta_t^{0,N} \equiv 1$.
        \item If ${K} < 0$ we understand $\sin(\sqrt{-1}\cdot r) = \sqrt{-1}\cdot \sinh(r)$ (note that the sign ambiguity is cancelled out).
        \item If $N = \infty$ then we set $\sigma_t^{{K},\infty}\equiv \tau_t^{{K},\infty} \equiv 1$ and $\beta_t^{{K},\infty}(\ell) = \exp(\frac{{K}}{6}(1-t^2)\ell^2)$.
        \item $\sigma_t^{{K},N}(0) = \tau_t^{{K},N}(0) = t$ and $\beta_t^{{K},N}(0) = 1$ for every ${K},t$ and $N$.
    \end{itemize}
    We also view $\tau^{{K},N}_t$ as a function on $SM\times_0[0,\infty)$ by setting 
    $$\tau^{{K},N}_t(v,\ell) = \tau^{{K},N}_t(\ell), \qquad (v,\ell) \in SM\times_0[0,\infty).$$
\end{definition}
\begin{theorem}[Displacement convexity]\label{DCthm}
    Let $K \in \RR$ and let $N \in [n,\infty]$. The following conditions are equivalent:
    \begin{enumerate}[(i)]
        \item The pair $(\mu,L)$ satisfies $\CD(K,N)$ in the sense of Definition \ref{CDdef}.
        \item For every $\mu_0 = f_0\mu,\mu_1=f_1\mu \in \cP_1(L)$ there exists a displacement interpolation $\mu_\lambda = (\Exp_\lambda)_*\Pi$ between $\mu_0$ and $\mu_1$ such that for every $\lambda \in [0,1]$,
        \begin{equation*} \rS_{N}[\mu_\lambda|\mu] \le 
        \begin{cases}    
        \int\left[\tau_{1-\lambda}^{{K},{N}}\cdot \left(-f_0^{-1/{N}}\circ\Exp_0\right) + \tau_{\lambda}^{{K},{N}}\cdot \left(-f_1^{-1/{N}}\circ\Exp_1\right)\right]d\Pi & {N} < \infty\\\noalign{\vskip9pt}
        (1-\lambda)\cdot\rS_\infty[\mu_0|\mu] + \lambda\cdot\rS_\infty[\mu_1|\mu] - \frac{K}{2}\cdot\lambda\cdot(1-\lambda)\cdot \int\ell^2 d\Pi & {N} =\infty
        \end{cases}
        \end{equation*}
        for all $\lambda \in [0,1]$. Here $\ell : SM\times_0[0,\infty)\to [0,\infty)$ denotes the  second variable.
    \end{enumerate}
\end{theorem}

\medskip
The one-dimensional Euclidean version of Theorem \ref{DCthm} is well known \cite{Sturm2,Vil,CGS}. Here and in the sequel, for a measure $m$ and a measurable set $A$ of positive measure, the measure $m\vert_A$ denotes the normalized probability measure $m\vert_A := m(A)^{-1}m(A\cap(\cdot))$.

\begin{lemma}[Displacement convexity on the real line]\label{DCthm1D}
    Let $m$ be a measure  on an interval $I\subseteq \RR$ with a smooth density and let $K \in \RR$ and $N \in (1,\infty]$. The following are equivalent:
    \begin{enumerate}[(i)]
        \item The pair $(m,l)$ satisfies $\CD_{}({K},N)$, where $l$ is the Euclidean Lagrangian $l = (dt^2+1)/2$.
        \item For every pair $m_0=f_0m,m_1=f_1m$ of absolutely continuous measures on $I$ of equal mass, if $\kappa$ and $\{m_\lambda\}_{\lambda \in [0,1]}$ are the monotone coupling and the monotone displacement interpolation between $m_0$ and $m_1$, respectively, then 
        \begin{align*}
        \hspace{-\leftmargin}
        \rS_{N}[m_\lambda|m] \le \begin{cases}-\int\left[\tau_{1-\lambda}^{{K},{N}}(|t'-t|)\cdot f_0(t)^{-1/{N}} + \tau_{\lambda}^{{K},{N}}(|t'-t|)\cdot f_1(t')^{-1/{N}}\right]d\kappa(t,t')& {N} < \infty,\\\noalign{\vskip9pt}
        (1-\lambda)\cdot \rS_\infty[m_0|m] + \lambda\cdot\rS_\infty[m_1|m] - \frac{K\lambda(1-\lambda)}{2}\cdot  \int(t'-t)^2 d\kappa(t,t')& {N} = \infty
        \end{cases}
        \end{align*}
    for every $\lambda \in [0,1]$.
    \item The conclusion of statement (ii) holds when $m_0 = m\vert_{I_0}$ and $m_1 = m \vert_{I_1}$, where $I_0,I_1\subseteq I$ are intervals satisfying $\sup I_0 < \inf I_1$.
    \end{enumerate} 
\end{lemma}

\begin{proof}
    We give a sketch of the proof; details can be found in \cite{Sturm2,Vil,Bor} and \cite[Theorem 3.17]{DJ}. Let $\rT$ and $\rT_\lambda$ be defined by \eqref{m0tinftyeq} and \eqref{Tlambdadef}.
    Differentiation gives
    \begin{align}
    \begin{split}\label{Tprimeq}
        \rho(t)f_0(t) & = \rho(\rT(t))f_1(\rT(t))\rT'(t) \quad \text{ and }\\
          \rT_\lambda'(t) & = (1-\lambda) + \lambda\rT'(t) = \frac{\rho(t)f_0(t)}{\rho(\rT_\lambda(t))f_\lambda(\rT_\lambda(t))}
    \end{split}
    \end{align} 
    for all $t \in \supp f_0$ (the function $\rT$ is increasing hence differentiable outside a countable set). 
    
    \medskip
    Suppose that the pair $(m,l)$ satisfies $\CD(K,N)$. Then the density $\rho = e^{-\psi}$ of the measure $m$ satisfies
    $$\ddot \psi - \frac{\dot \psi^2}{N-1} \ge {K},$$
    see Example \ref{needleexample}. Integrating this inequality (see e.g. \cite[Chapter 14]{Vil}) gives
    \begin{equation}\label{rhooneminuslambdateq}
		\rho((1-\lambda)\, t + \lambda\, t') \ge \bM_{\frac{1}{N-1}}\left(\beta_{1-\lambda}^{{K},N}(t'-t)\cdot\rho(t),\beta_\lambda^{{K},N}(t' - t) \cdot\rho(t');\lambda\right)
	\end{equation}
    for all $t \le t' \in I$, where $\bM$ is the generalized mean defined in Section \ref{BMsec}. In particular,
	\begin{equation}\label{rhoineq}
		\rho(\rT_\lambda(t)) \ge \bM_{\frac{1}{N-1}}\left(\beta_{1-\lambda}^{{K},N}(\rT(t)-t)\cdot\rho(t),\beta_\lambda^{{K},N}(\rT(t)-t) \cdot\rho(\rT(t));\lambda\right)
	\end{equation}
    for all $t\in I$ (note that $t \le \rT(t)$ by Lemma \ref{1Dintlemma}).  H{\"o}lder's inequality implies that for every $q,q_1,q_2 \in [0,\infty]$ satisfying $1/q_1+1/q_2 =1/q$ and every $a_1,a_2,b_1,b_2 \ge 0$,
	\begin{equation}\label{pmeaninequality}\bM_{q_1}(a_1,b_1;\lambda)\cdot \bM_{q_2}(a_2,b_2;\lambda) \ge \bM_{q}(a_1a_2,b_1b_2;\lambda)\end{equation}
	Combining \eqref{Tprimeq}, \eqref{rhoineq}  and \eqref{pmeaninequality} with $q_1 = 1$ and $q_2 = \frac{1}{N-1}$ we obtain
	\begin{align}\label{rhoTlambdaeq}
    \begin{split}
		\rho(\rT_\lambda(t))\rT_\lambda'(t) & \ge \bM_{\frac1N}\left(\beta_{1-\lambda}^{K,N}(\rT(t)-t)\rho(t),\beta_\lambda^{K,N}(\rT(t)-t)\rho(\rT(t))\rT'(t);\lambda\right)\\
        & = \bM_{\frac1N}\left(\beta_{1-\lambda}^{K,N}(\rT(t)-t)\rho(t),\beta_\lambda^{K,N}(\rT(t)-t)\cdot\frac{\rho(t)f_0(t)}{f_1(\rT(t))};\lambda\right).
    \end{split}
	\end{align}
    Suppose first that $N < \infty$. Using the definitions of $\tau,\beta,\rS_N$ and $\kappa$, the condition in (ii) can be written as
    \begin{equation}\label{DCineq1}
        \begin{split}
        \int_If_\lambda(t)^{1-1/N}\rho(t)dt & \ge (1-\lambda)\int_If_0(t)^{1-1/N}\beta_{1-\lambda}^{K,N}(\rT(t)-t)^{1/N}\rho(t)dt \\
        &\qquad + \lambda\int_If_1(t)^{1-1/N}\beta_\lambda^{K,N}(t-\rT^{-1}(t))^{1/N}\rho(t)dt.
        \end{split}
    \end{equation}
    In order to prove \eqref{DCineq1} we begin by making a change of variables via the map $\rT_\lambda$:
    \begin{align*}
    \begin{split}
        \int_If_\lambda(t)^{1-1/N}\rho(t)dt & = \int_If_\lambda(\rT_\lambda(t))^{1-1/N}\rho(\rT_\lambda(t))\rT_\lambda'(t)dt\\
        & = \int_I\rho(t)^{1-1/N}f_0(t)^{1-1/N}\rho(\rT_\lambda(t))^{1/N}\rT_\lambda'(t)^{1/N}dt,
    \end{split}
    \end{align*}
    where in the second passage we used \eqref{Tprimeq}. 
    Applying \eqref{rhoTlambdaeq} and then using the definition of $\bM_{\frac1N}$, we get
    \begin{align*}
        & \int_If_\lambda(t)^{1-1/N}\rho(t)dt \\
        \qquad \qquad & \ge \int_I\rho(t)^{1-1/N}f_0(t)^{1-1/N}\bM_{\frac1N}\left(\beta_{1-\lambda}^{K,N}(\rT(t)-t)\rho(t),\beta_\lambda^{K,N}(\rT(t)-t)\cdot\frac{\rho(t)f_0(t)}{f_1(\rT(t))};\lambda\right)^{1/N}dt\\
        & = (1-\lambda)\int_I\rho(t)f_0(t)^{1-1/N}\beta_{1-\lambda}^{K,N}(\rT(t)-t)^{1/N}dt\\
        & \qquad + \lambda\int_I\rho(t)f_0(t)f_1(t)^{-1/N}\beta_\lambda^{K,N}(\rT(t)-t)^{1/N}dt\\
        & = (1-\lambda)\int_I\rho(t)f_0(t)^{1-1/N}\beta_{1-\lambda}^{K,N}(\rT(t)-t)^{1/N}dt\\
        & \qquad + \lambda\int_I\rho(\rT(t))f_1(\rT(t))^{1-1/N}\rT'(t)\beta_\lambda^{K,N}(\rT(t)-t)^{1/N}dt,
    \end{align*}
    where in the last passage we used \eqref{Tprimeq} again. By changing variables via the map $\rT$ in the second integral we obtain \eqref{DCineq1}. 
    
    \medskip
    If $N = \infty$ then \eqref{rhoTlambdaeq} reads
    $$\rho(\rT_\lambda(t))\rT_\lambda'(t)\ge \left(\beta_{1-\lambda}^{K,\infty}(\rT(t) - t)\rho(t)\right)^{1-\lambda}\left(\beta_\lambda^{K,\infty}(\rT(t) - t)\frac{\rho(t)f_0(t)}{f_1(\rT(t))}\right)^{\lambda}, \qquad t \in \supp f_0.$$
    Recalling that $\beta_t^{K,\infty}(\ell) = \exp\left(\frac{K}{6}(1-t^2)\ell^2\right)$, we get
    $$\rho(\rT_\lambda(t))\rT_\lambda'(t)\ge \rho(t)e^{\frac{K}{2}\lambda(1-\lambda)(\rT(t)-t)^2}\left(\frac{f_0(t)}{f_1(\rT(t))}\right)^{\lambda}, \qquad t \in \supp f_0,$$
    whence by \eqref{Tprimeq}
    $$\log\left(\frac{f_0(t)}{f_\lambda(\rT_\lambda(t))}\right) \ge \frac{K}{2}\lambda(1-\lambda)(\rT(t)-t)^2+\lambda\log\left(\frac{f_0(t)}{f_1(\rT(t))}\right).$$
    Upon rearrangement this becomes
    $$\log(f_\lambda(\rT_\lambda(t))) \le (1-\lambda)\log(f_0(t))  + \lambda\log(f_1(\rT(t))) - \frac{K}{2}\lambda(1-\lambda)(\rT(t)-t)^2.$$
    Thus
    \begin{align*}
        \int_If_\lambda(t)\log(f_\lambda(t))\rho(t)dt & = \int_If_\lambda(\rT_\lambda(t))\log(f_\lambda(\rT_\lambda(t)))\rho(\rT_\lambda(t))\rT_\lambda'(t)dt\\
        & = \int_If_0(t)\rho(t)\log(f_\lambda(\rT_\lambda(t)))dt\\
        & \le \int_If_0(t)\rho(t)\Big((1-\lambda)\log(f_0(t))  + \lambda\log(f_1(\rT(t))) \\
        & \qquad - \frac{K}{2}\lambda(1-\lambda)(\rT(t)-t)^2\Big)dt
    \end{align*}
    which, by the definition of $\rS_\infty$ and $\kappa$, amounts to the desired inequality.

    \medskip
    The implication (ii)$\implies$(iii) is trivial. The implication (iii)$\implies$(i) can be found in e.g. \cite[Proposition 3.4]{CGS}; see also \cite[Theorem 3.17]{DJ} for a similar argument. The idea is to assume by contradiction that inequality \eqref{rhooneminuslambdateq} does not hold for some $t<t'$, and take the intervals $I_0$ and $I_1$ to be $t+\ell_0$ and $t' + \ell_1$, respectively, where $\ell_i$ are small positive numbers, chosen so that equality holds in \eqref{pmeaninequality}.
\end{proof}
\begin{proof}[Proof of Theorem \ref{DCthm}]
    Let $N \in [n,\infty)$ and assume that $\Ric_{\mu,N}\ge K$. Let $$\mu_i=f_i\mu \in \cP_1(L), \qquad i=0,1.$$
    Choose a Kantorovich potential $u$ for the function $f_1-f_0$, let $\mu = \int_{\sA}\mu_\alpha\,d\nu(\alpha)$ be the disintegration of measure provided by Theorem \ref{needlethm2} and let $\mu_\lambda = f_\lambda\mu = (\Exp_\lambda)_*\Pi$ be the displacement interpolation provided by Proposition \ref{interprop}. 

    \medskip
    In order to unify below the notation for the classes $\sN$ and $\sD$, for $\alpha \in \sD$ we set
    $$I_\alpha = \{0\}, \qquad \gamma_\alpha:\{0\}\to M, \qquad \gamma_\alpha(0) = x_\alpha,$$
    as well as
    $$m_\alpha : = \delta_0 \qquad \text{ and } \qquad m_{\alpha,\lambda} = f_\lambda(x_\alpha)m_\alpha = f_\lambda(x_\alpha)\delta_0,$$
    (where $\delta_0$ is a Dirac probability measure at zero) and 
    \begin{equation}\label{SNdeltaeq}\rS_N[m_{\alpha,\lambda}|m_\alpha] =
    \begin{cases}
        -f_\lambda(x_\alpha)^{1-1/N} & N < \infty\\
        f_\lambda(x_\alpha) \log f_\lambda(x_\alpha) & N = \infty
    \end{cases}    
    \end{equation}
    for all $0 \le \lambda \le 1$. By \eqref{disinteq2}, \eqref{flambdadef}, \eqref{SNdeltaeq} and the definition of $\rS_{N}$,
    \begin{align}\label{HNmulambdaeq}
        \rS_{N}[\mu_\lambda|\mu] & = \int_{\sA}\rS_{N}[m_{\alpha,\lambda}|m_\alpha]\,d\nu(\alpha), \qquad \lambda \in [0,1].
    \end{align}

    \medskip
    For every $\alpha \in \sN$ the family $m_{\alpha,\lambda}$ is a displacement interpolation between $m_{0,\alpha}$ and $m_{1,\alpha}$, and the pair $(m_\alpha,(dt^2+1)/2)$ satisfies $\CD(K,N)$ since the measure $\mu_\alpha$ is a $\CD(K,N)$-needle by Lemma \ref{CDlemma}. Hence, by Lemma \ref{DCthm1D} and by \eqref{Explambdadotgammaalphaeq},
    \begin{align*}
        \rS_{N}[m_{\alpha,\lambda}|m_\alpha] \le &  -\int_{I_\alpha\times I_\alpha}\left[\tau_{1-\lambda}^{{K},{N}}(t'-t)\cdot f_0(\gamma_\alpha(t))^{-1/{N}} + \tau_{\lambda}^{{K},{N}}(t'-t)\cdot f_1(\gamma_\alpha(t'))^{-1/{N}}\right]d\kappa_\alpha(t,t')\\
        = &  -\int_{I_\alpha\times I_\alpha}\left[\tau_{1-\lambda}^{{K},{N}}(t'-t)\cdot f_0(\Exp_0(\dot\gamma_\alpha(t),t'-t))^{-1/{N}} \right.\\ & \left. \qqquad + \tau_{\lambda}^{{K},{N}}(t'-t)\cdot f_1(\Exp_1(\dot\gamma_\alpha(t),t'-t))^{-1/{N}}\right]d\kappa_\alpha(t,t').
    \end{align*}
    Note that we have used the fact that $\kappa_\alpha$-almost surely  $t'\ge t$.
    The same inequality holds trivially when $\alpha \in \sD$, since $\tau_{1-\lambda}^{{K},{N}}(0) = 1-\lambda$ and $\tau_{\lambda}^{{K},{N}}(0) = \lambda$, and $f_\lambda(x_\alpha) = f_0(x_\alpha)$ for all $\lambda \in [0,1]$ while $\kappa_\alpha = \delta_{(0,0)}$, so both sides equal $f_0(x_\alpha)^{1-1/{N}}$.

    \medskip
    Integrating with respect to $\alpha$ and using \eqref{Pidef} we see that
    \begin{align*}
        \int_{\sA}\rS_{N}[m_{\alpha,\lambda}|m_\alpha]\,d\nu(\alpha) \le & -\int_{SM\times_0[0,\infty)}\left[\tau_{1-\lambda}^{{K},{N}}(\ell)\cdot f_0(\Exp_0(v,\ell))^{-1/{N}}\right. \\ & \qqquad + \left.\tau_{\lambda}^{{K},{N}}(\ell)\cdot f_1(\Exp_1(v,\ell))^{-1/{N}}\right]\,d\Pi(v,\ell),
    \end{align*}
    which combines with \eqref{HNmulambdaeq} to give the desired inequality when $N < \infty$. The proof for $N = \infty$ is similar.

    \medskip
    Assume now that $\mu$ \emph{does not} satisfy $\CD(K,N)$. Then there exists an open set $U\subseteq M$ and a $C^3$ solution $u:U\to\RR$ to the Hamilton Jacobi equation such that 
    \begin{equation}\label{contradictioneq}(d\bL u)(\nabla u) + \frac{(\bL u)^2}{N-1} + K > 0\end{equation}
    at some $x_0 \in U$. 
    Without loss of generality, the set $U$ is the domain of a $C^3$ coordinate chart
    $$(x^1,\dots,x^n) \qquad -\eps<x^i<\eps$$
    for some $\eps > 0$, such that in this coordinate chart $x_0 = 0$ and
    $$\nabla u = \partial_{x^n}.$$
    Since $u$ is $C^3$, we may also assume that \eqref{contradictioneq} holds on all of $U$. Extend $u$ to a dominated function on all of $M$, for instance by the formula
    $$u(z) : = \sup_{y\in M}[\inf_{x\in U}u(x) + \cc(x,y) - \cc(z,y)], \qquad z \in M.$$
    Since $H(du) = 0$ on the set $U$, it is contained in the strain set of the extended $u$, and in particular
    \begin{equation}\label{uUstraineq}u(\gamma_{\hat x}(t')) - u(\gamma_{\hat x}(t)) = \cc(\gamma_{\hat x}(t),\gamma_{\hat x}(t')), \qquad \hat x\in(-\eps,\eps)^{n-1}, \quad -\eps<t\le t'<\eps,\end{equation}
    where
    $$\gamma_{\hat x}(t) := (\hat x,t).$$
    Write the density $\omega$ in coordinates as
	\begin{equation*}
        \omega = J\,dx^1\wedge\dots\wedge dx^n.
    \end{equation*}
	Then
	\begin{align*}
		(\partial_{x^n} J)\,dx^1\wedge\dots\wedge dx^n & = \sL_{\partial_{x^n}}(J\,dx^1\wedge\dots\wedge dx^n)\\
		& = \sL_{\partial_{x^n}}\omega\\
		& = (\div_\mu \partial_{x^n}) \cdot \omega\\
		& = (\div_\mu \nabla u) \,\omega\\
		& = \bL u \cdot J\, dx^1\wedge\dots\wedge dx^n,
	\end{align*}
	whence
		$$\bL u = \partial_{x^n}\log J.$$
	Inequality \eqref{contradictioneq} thus becomes
	\begin{equation*}
		\partial^2_{x^n}\log J + \frac{(\partial_{x^n}\log J)^2}{N-1} + {K} > 0.
	\end{equation*}
    In other words, for every $\hat x = (x^1,\dots,x^{n-1}) \in (-\eps,\eps)^{n-1}$, the measure $m_{\hat x}$ on the interval $(-\eps,\eps)$ given by
    $$dm_{\hat x} = J(\hat x,t)dt$$
    does \emph{not} satisfy $\CD(K,N)$ with respect to the Euclidean cost (see Example \ref{needleexample}). By Lemma \ref{DCthm1D}, for every $\hat x \in (-\eps,\eps)^{n-1}$ there exist $$-\eps<t_0<s_0<t_1<s_1<\eps$$ such that, if $\kappa_{\hat x}$ and $\{m_{\hat x,\lambda} = f_{\hat x,\lambda}m_{\hat x}\}_{\lambda\in[0,1]}$ are the monotone coupling and the monotone displacement interpolation between the measures $m_{\hat x}\vert_{[t_0,s_0]}$ and $m_{\hat x}\vert_{[s_1,t_1]}$, then
        $$
        \rS_N[m_{\hat x,\lambda}|m_{\hat x}] > \begin{cases}-\int\left[\tau_{1-\lambda}^{{K},N}(|t'-t|)\cdot f_{\hat x,0}(t)^{-1/N} + \tau_{\lambda}^{{K},N}(|t'-t|)\cdot f_{\hat x,1}(t')^{-1/N}\right]d\kappa_{\hat x}(t,t')& N < \infty,\\\noalign{\vskip9pt}
        (1-\lambda)\cdot \rS_\infty[m_{\hat x,0}|m_{\hat x}] + \lambda\cdot\rS_\infty[m_{\hat x,1}|m_{\hat x}] - \frac{K\lambda(1-\lambda)}{2}\cdot  \int(t'-t)^2 d\kappa_{\hat x}(t,t')& N = \infty.
        \end{cases}
        $$
        Define a measure $\mu_\lambda$ on $(-\eps,\eps)^n$ by 
        $$\mu_\lambda : = \int\limits_{(-\eps,\eps)^{n-1}}m_{\hat x,\lambda}d\hat x, \qquad \lambda \in [0,1],$$
        where the integration is with respect to the Lebesgue measure on $(-\eps,\eps)^{n-1}$. We can view $\mu_\lambda$ as a measure on $M$ by identifying $(-\eps,\eps)^n$ with the set $U$. 
        As in the proof of Proposition \ref{interprop}, if $\Pi$ is the transport plan satisfying
        $$\int_{SM\times_0[0,\infty)}\phi(v,\ell) \,d\Pi(v,\ell) : = \int_{(-\eps,\eps)^{n-1}}\left(\int_{(-\eps,\eps)^2}\phi(\dot\gamma_{\hat x}(t),t'-t)\,d\kappa_{\hat x}(t,t')\right) \,d\hat x$$
        for every Borel function $\phi : SM\times_0[0,\infty)\to \RR$, then since each $\gamma_{\hat x}$ is a minimizing extremal,
        $$\mu_\lambda = (\Exp_\lambda)_*\Pi ,\qquad \lambda \in [0,1].$$

        \medskip
        It follows from \eqref{uUstraineq} and Theorem \ref{KDthm} that $\mu_\lambda$ is a displacement interpolation between $\mu_0$ and $\mu_1$. In fact, it is the \emph{only} displacement interpolation between $\mu_0$ and $\mu_1$, since there is a unique solution to the Monge problem on the real line when the supports of both measures are intervals (this follows from cyclic monotonicity, see e.g. \cite[Section 3]{Amb}). It also follows from the definitions that
        $$\rS_N[\mu_\lambda|\mu] = \int\limits_{(-\eps,\eps)^{n-1}}\rS_N[m_{\hat x,\lambda}|m_{\hat x}]\,d\hat x,\qquad \lambda \in[0,1].$$
        Thus, by integrating our lower bound on $\rS_N[m_{\hat x,\lambda}|m_{\hat x}]$ over $\hat x$ we conclude that condition (ii) does not hold.
\end{proof}

\section{Consequences of the curvature-dimension condition}\label{appsec}

\subsection{The Brunn-Minkowski inequality}\label{BMsec}

For a detailed survey of the classical Brunn-Minkowski inequality see \cite{Gard}; the results we present here are modeled after the distorted Brunn-Minkowski and Borell-Brascamp-Lieb inequalities for Riemannian manifolds and metric measure spaces \cite{CMS,Sturm2,Vil,CM17,Oh09}. Brunn-Minkowski inequalities for sub-Riemannian manifolds were proved in \cite{BR,BMR,BKS}, and for Lorentzian manifolds in \cite{CM24}.

\medskip
Let $A_0,A_1\subseteq M$ and let $0 < \lambda < 1$. Let $A_\lambda$ denote the set of $\lambda$-midpoints of minimizing extremals joining $A_0$ and $A_1$:
$$A_\lambda : = \{\gamma(\lambda \ell) \, \mid \, \gamma:[0,\ell] \to M \, \, \text{ is a minimizing extremal, $\quad \gamma(0) \in A_0, \, \,\gamma(\ell) \in A_1$}\}.$$

\medskip
For $a,b \ge 0$, $0 < \lambda < 1$ and $q \in [-\infty,\infty]$ define the \emph{generalized mean} $\bM$ by:
$$
	\bM_q(a,b ; \lambda) = 
	\left \{ 
		\begin{array}{cc}
			\begin{array}{cc}
				\big( (1-\lambda) a^q +  \lambda  b^q \big)^{1/q} & q \in \RR\setminus\{0\}\\
				a^{1-\lambda} b^{\lambda} & q = 0 \\
				\max \{ a,b \} & q = +\infty\\
				\min \{ a,b \} & q = -\infty\\
			\end{array} & \qquad ab > 0 \\
		\quad & \quad \\
		0 &  \qquad ab = 0.
		\end{array} \right.
$$

Let us restate Theorem \ref{BMthm0}, which is the main result of this section:
\begin{theorem}\label{BMthm}
    Suppose that the pair $(\mu,L)$ satisfies $\CD_{}({K},N)$ for some ${K} \in \RR$ and some $N \in [n,\infty]$. Then for every pair $A_0,A_1\subseteq M$ of Borel sets and every $0 < \lambda < 1$,
        \begin{equation}\label{BMeq}\mu(A_\lambda) \ge \bM_{\frac{1}{N}}\left(\beta_{1-\lambda}^{{K},N}(A_0,A_1)\cdot\mu(A_0),\beta_{\lambda}^{{K},N}(A_0,A_1)\cdot\mu(A_1);\lambda\right),\end{equation}
        where
        $$\beta_t^{{K},N}(A_0,A_1) := \inf\left\{\beta_t^{{K},N}(\ell)\,\, \Big\vert \, \, \exists \text{ a minimizing extremal $\gamma:[0,\ell] \to M$ joining $A_0$ to $A_1$}\right\}.$$
\end{theorem}

For the definition of the distortion coefficients $\beta^{K,N}_t$ see Definition \ref{distcoeffdef}. In the case $K=0$ and $N=\infty$, the distortion coefficients are identically $1$ and inequality \eqref{BMeq} takes the familiar form \eqref{BMsimpleeq}.

\begin{remark}\normalfont
    When the sets $A_0,A_1$ are Borel, the set $A_\lambda$ is Lebesgue. Indeed, it is well known that images of Borel sets under continuous maps are Lebesgue, and the set $A_\lambda$ is the image of the Borel set $(\Exp_0)^{-1}(A_0)\cap (\Exp_1)^{-1}(A_1) \subseteq SM \times_0 [0,\infty)$ under the continuous map $\Exp_\lambda$.
\end{remark}

\medskip
Theorem \ref{BMthm} can be obtained from Theorem \ref{DCthm} by taking $\mu_i = \mu\vert_{A_i}$. We will prove a more general, functional version of Theorem \ref{BMthm}, due in the Riemannian case to Cordero-Erausquin, McCann and Schmuckenschläger \cite{CMS}:

\begin{restatable}[The Borell-Brascamp-Lieb inequality]{theorem}{bbl}\label{BBLthm}
	Suppose that the pair $(\mu,L)$ satisfies $\CD_{}({K},N)$ for some ${K} \in \RR$ and some $N \in [n,\infty]$. Let $q \in [-1/N,\infty]$, let $0 < \lambda <1$ and let $f_0,f_\lambda,f_1$ be nonnegative $\mu$-integrable functions. Suppose that for every minimizing extremal $\gamma : [0,\ell] \to M$,
	\begin{equation}\label{BBLhypo}
		f_\lambda(\gamma(\lambda \ell)) \ge \bM_q\left(\frac{f_0(\gamma(0))}{\beta_{1-\lambda}^{{K},N}(\ell)},\frac{f_1(\gamma(\ell))}{\beta_\lambda^{{K},N}(\ell)};\lambda\right).
	\end{equation}
	Then
	\begin{equation}\label{mainBBL}
		\int_Mf_\lambda d\mu \ge \bM_{q'}\left(\int_Mf_0d\mu,\int_Mf_1d\mu;\lambda\right),
	\end{equation}
	where
	$$q' = \frac{q}{1 + qN}$$
	and we interpret $q' = -\infty$ if $q = -1/N$, $q' = 1/N$ if $q = \infty$, and $q' = 0$ if $N = \infty$.
\end{restatable}

Theorem \ref{BMthm} follows from Theorem \ref{BBLthm} by taking 
$$f_0 = \beta_{1-\lambda}^{K,N}(A_0,A_1)\cdot\chi_{A_0}, \qquad f_1 = \beta_{\lambda}^{K,N}(A_0,A_1)\cdot\chi_{A_1} \qquad \text{ and } \qquad f_\lambda = \chi_{A_\lambda}$$ 
and setting $q = \infty$.

\begin{proof}[Proof of Theorem \ref{BBLthm}]
    By inner regularity of the measure $\mu$, we may assume that $f_0,f_1$ are compactly supported. We may also assume that $\int f_0d\mu,\int f_1d\mu > 0$, since otherwise the right hand side of \eqref{mainBBL} is zero by definition. Set
\begin{equation}\label{fdef}
	f  : = \frac{f_1}{\int f_1d\mu} - \frac{f_0}{\int f_0d\mu}.
\end{equation}
Let $\mu = \int_{\sA}\mu_\alpha\,d\nu(\alpha)$ be the disintegration of measure provided by Theorem \ref{needlethm}.
By \eqref{fdef} and the mass-balance properties \eqref{MBeq} and \eqref{detailedmbeq},
    \begin{equation}\label{MBeqBM}
        \frac{\int_M f_0d\mu_\alpha}{\int_M f_0d\mu} = \frac{\int_M f_1d\mu_\alpha}{\int_M f_1d\mu}, \qquad \text{ for $\nu$-a.e $\alpha \in \sA$}
    \end{equation}
and
    \begin{equation}\label{detailedmbeqBM}
        \frac{\int_Mf_0d\mu_{\alpha,t}}{\int_Mf_0d\mu_\alpha} \le \frac{\int_Mf_1d\mu_{\alpha,t}}{\int_Mf_1d\mu_\alpha} \qquad \text{ for $\nu$-a.e. $\alpha \in \sN$ and $t \in \RR$,}
    \end{equation}
where
    $$\mu_{\alpha,t} := (\gamma_\alpha)_*(m_\alpha\vert_{[t,\infty)}), \qquad t \in \RR.$$

We will prove below that inequality \eqref{mainBBL} holds for each needle $\mu_\alpha$:

\begin{lemma}\label{needleBMlemma}
	For $\nu$-almost every $\alpha \in \sA$, 
	\begin{equation}\label{needleBMeq}
		\int_M f_\lambda d\mu_\alpha \ge \bM_{q'}\left(\int_M f_0 d\mu_\alpha,\int_M f_1d\mu_\alpha;\lambda\right).
	\end{equation}
\end{lemma}

By \eqref{disinteq}, \eqref{MBeqBM}, \eqref{needleBMeq} and the homogeneity of $\bM$,
\begin{align*}
	\int f_\lambda d\mu & = \int_{\sA}\left(\int f_\lambda d\mu_\alpha\right)d\nu(\alpha)\\
	& \ge \int_{\sA}\bM_{q'}\left(\int f_0d\mu_\alpha,\int f_1d\mu_\alpha;\lambda\right)d\nu(\alpha)\\
	& = \int_{\sA}\left(\int f_0 d\mu_\alpha\,\cdot \, \bM_{q'}\left(1,\frac{\int f_1d\mu_\alpha}{\int f_0 d\mu_\alpha};\lambda\right)\right)d\nu(\alpha)\\
	& = \int_{\sA}\left(\int f_0 d\mu_\alpha\, \cdot \, \bM_{q'}\left(1,\frac{\int f_1d\mu}{\int f_0 d\mu};\lambda\right)\right)d\nu(\alpha)\\
	& = \int f_0d\mu \, \cdot \, \bM_{q'}\left(1,\frac{\int f_1d\mu}{\int f_0 d\mu};\lambda\right)\\
	& = \bM_{q'}\left(\int f_0d\mu,\int f_1d\mu;\lambda\right).
\end{align*}
This concludes the proof of the Theorem \ref{BBLthm} given Lemma \ref{needleBMlemma}.
\end{proof}

\begin{proof}[Proof of Lemma \ref{needleBMlemma}]
	Suppose first that $\alpha \in \sD$, i.e. $\mu_\alpha$ is a Dirac measure at some point $x_\alpha \in M$. Taking the geodesic $\gamma$ in \eqref{BBLhypo} to be a constant curve $\gamma \equiv x_\alpha$ and recalling that $\beta_t^{K,N}(0) = 1$, we obtain
	\begin{align*}
		\int f_\lambda d\mu_\alpha = f_\lambda(x_\alpha) \ge \bM_q\left(f_0(x_\alpha),f_1(x_\alpha);\lambda\right) & \ge \bM_{q'}\left(f_0(x_\alpha),f_1(x_\alpha);\lambda\right) \\ & = \bM_{q'}\left(\int f_0d\mu_\alpha,\int f_1d\mu_\alpha;\lambda\right)
	\end{align*}
	where in the second passage we used $q' \le q$ and H{\"o}lder's inequality.

	\medskip
	Now assume that $\alpha \in \sN$. Then $\mu_\alpha$ is a $\CD(K,N)$-needle, i.e. it takes the form
	$$\mu_\alpha = (\gamma_\alpha)_*m_\alpha,$$
	where $\gamma_\alpha:(a_\alpha,b_\alpha)\to M$ is a minimizing extremal and $m_\alpha$ is an absolutely continuous measure on $(a_\alpha,b_\alpha)$ satisfying $\CD({K},N)$ with respect to the Euclidean metric on $\RR$.

	\medskip
	Let $a_\alpha < t_0 \le t_1 < b_\alpha$. Applying \eqref{BBLhypo} to the minimizing extremal $\gamma_\alpha\vert_{[t_0,t_1]}$ we see that
	\begin{equation}\label{flambdaneedleineq}	
		f_\lambda(\gamma_\alpha((1-\lambda)t_0 + \lambda t_1)) \ge \bM_q\left(\frac{f_0(\gamma_\alpha(t_0))}{\beta^{{K},N}_{1-\lambda}(t_1 -t_0)},\frac{f_1(\gamma_\alpha(t_1))}{\beta^{{K},N}_\lambda(t_1 -t_0)};\lambda\right)
	\end{equation}
	Since $\mu_\alpha$ is a $\CD({K},N)$-needle, the density $\rho_\alpha = e^{-\psi_\alpha}$ of the measure $m_\alpha$ with respect to the Lebesgue measure satisfies
	$$
		\ddot\psi_\alpha \ge {K} + \frac{\dot\psi_\alpha^2}{N - 1} \qquad \text{ on } \quad [t_0,t_1].
	$$
	as in \eqref{rhooneminuslambdateq}, this implies
	\begin{equation}\label{rhoalphaneedleineq}
		\rho_\alpha((1-\lambda)\, t_0 + \lambda\, t_1) \ge \bM_{\frac{1}{N-1}}\left(\beta_{1-\lambda}^{{K},N}(t_1-t_0)\cdot\rho_\alpha(t_0),\beta_\lambda^{{K},N}(t_1 - t_0) \cdot\rho_\alpha(t_1);\lambda\right)
	\end{equation}
	for all $a_\alpha < t_0 \le t_1 < b_\alpha$. Set
	$$h_{\alpha,i} : = \rho_\alpha \cdot (f_i\circ\gamma_\alpha), \qquad i = 0,\lambda,1.$$
	By \eqref{flambdaneedleineq}, \eqref{rhoalphaneedleineq} and \eqref{pmeaninequality},
	\begin{equation}\label{hlambdaneedleeq1}
		h_{\alpha,\lambda}((1-\lambda)\,t_0 + \lambda t_1) \ge \bM_{\frac{1}{N-1 + 1/q}}\left(h_{\alpha,0}(t_0),h_{\alpha,1}(t_1);\lambda\right), \qquad a_\alpha < t_0 \le t_1 < b_\alpha.
	\end{equation}
	By \eqref{MBeqBM},
	$$\frac{\int_{a_\alpha}^{b_\alpha}h_{\alpha,0}(t)dt}{\int f_0d\mu} = \frac{\int_{a_\alpha}^{b_\alpha}{f_0}(\gamma_\alpha(t))\rho_\alpha(t)dt}{\int f_0d\mu} =  \frac{\int_{a_\alpha}^{b_\alpha}{f_1}(\gamma_\alpha(t))\rho_\alpha(t)dt}{\int f_1d\mu} = \frac{\int_{a_\alpha}^{b_\alpha}h_{\alpha,1}(t)dt}{\int f_1d\mu}$$
	and similarly, by \eqref{detailedmbeqBM}, for every $t \in (a_\alpha,b_\alpha)$,
	\begin{equation}\label{stocdomeq}
		\frac{\int_t^{b_\alpha}h_{\alpha,0}(t)dt}{\int_{a_\alpha}^{b_\alpha}h_{\alpha,0}(t)dt} \le \frac{\int_{t}^{b_\alpha}h_{\alpha,1}(t)dt}{\int_{a_\alpha}^{b_\alpha}h_{\alpha,1}(t)dt}
	\end{equation}
	We now apply the following oriented version of the one-dimensional Borell-Brascamp-Lieb inequality, proved in \cite{AK}:
	\begin{lemma}[\text{\cite[Lemma 3.8]{AK}}]\label{dirBBLlemma}
		Let $q \in [-1,\infty]$. Let $h_0,h_1$ be non-negative, integrable functions on an interval $(a,b) \subseteq \RR$ satisfying
		$$
				\frac{\int_t^bh_0}{\int_a^bh_0} \le \frac{\int_t^bh_1}{\int_a^bh_1} \qquad \text{ for every } t \in (a,b).
		$$
		Let $h_\lambda$ be a non-negative integrable function satisfying
		$$
			h_\lambda((1-\lambda)t_0 + \lambda t_1) \ge \bM_q(h_0(t_0),h_1(t_1);\lambda)
		$$
		for every $a < t_0 \le t_1 < b$. Then
		$$
			\int_a^bh_\lambda \ge \bM_{\frac{q}{q+1}}\left(\int_a^bh_0,\int_a^bh_1 ; \lambda\right).
		$$
	\end{lemma} 
	It follows from \eqref{hlambdaneedleeq1}, \eqref{stocdomeq} and Lemma \ref{dirBBLlemma} that 
	$$\int_{a_\alpha}^{b_\alpha}h_{\alpha,\lambda}(t)dt \ge \bM_{\frac{{q}}{1+qN}}\left(\int_{a_\alpha}^{b_\alpha}h_{\alpha,0}(t)dt,\int_{a_\alpha}^{b_\alpha}h_{\alpha,1}(t)dt;\lambda\right)$$
	which, by the definitions of $h_{\alpha,0},h_{\alpha,1},h_{\alpha,\lambda}$ and $q'$ implies \eqref{needleBMeq}.
\end{proof}

\subsection{The Bonnet-Myers theorem}

The following standard lemma is proved exactly as in the Riemannian case, see \cite[Lemma 5.2]{CE}.

\begin{lemma}\label{uniquelyminimizinglemma}
	Let $\gamma:[0,T] \to M$ be a minimizing extremal. For every $t \in (0,T)$, the extremal $\gamma\vert_{[0,t]}$ is uniquely minimizing.
\end{lemma}
\begin{proof}
    Let $t \in (0,T)$, let $\tilde\gamma : [0,\tilde T] \to M$ be a minimizing extremal joining $\gamma(0)$ to $\gamma(t)$ and assume by contradiction that $\tilde\gamma \ne \gamma\vert_{[0,t]}$. Pick $\tilde t \in (0,\tilde T)$ such that $\tilde\gamma(\tilde t)$ does not lie on $\gamma$,  and let $\bar \gamma$ be a minimizing extremal joining $\tilde\gamma(\tilde t)$ to $\gamma(T)$. Since minimizing extremals are differentiable, the concatenation of $\tilde \gamma\vert_{[\tilde t,\tilde T]}$ with $\gamma\vert_{[ t,T]}$ is not minimizing, i.e.
	$$\int_{\bar\gamma}L < \int_{\tilde\gamma\vert_{[\tilde t,\tilde T]}}L + \int_{\gamma\vert_{[t,T]}}L,$$
	whence
	$$\int_{\tilde\gamma\vert_{[0,\tilde t]}}L + \int_{\bar\gamma}L < \int_{\tilde\gamma\vert_{[0,t]}}L + \int_{\gamma\vert_{[t,T]}}L = \int_{\gamma\vert_{[0,t]}}L + \int_{\gamma\vert_{[t,T]}}L = \int_\gamma L.$$
	Thus, by concatenating $\bar\gamma$ to $\tilde\gamma\vert_{[0,\tilde t]}$ we obtain a curve joining $\gamma(0)$ to $\gamma(T)$ with action strictly smaller than that of $\gamma$, contradicting the assumption that $\gamma$ is minimizing.
\end{proof}

The next lemma is also standard, see e.g. \cite[Corollary 2.6]{Sturm2} or \cite[Theorem 14.12]{Vil}. We state it without proof.

\begin{lemma}\label{1DBonnetlemma}
    Let ${K} \ge 0$ and $1 < N < \infty$ and let $m$ be a measure on $\RR$ with a smooth density which is supported on an interval $I\subseteq \RR$ and satisfies $\CD(K,N)$ with respect to the Euclidean metric. 
    \begin{enumerate}
        \item If ${K} > 0$ then the length of $I$ is at most $\pi\cdot\sqrt{(N-1)/{K}}$.
        \item If ${K} = 0$ and $I = \RR$ then the density of $m$ is constant.
    \end{enumerate}
\end{lemma}

Recall that the diameter $\diam(A)$ of a set $A\subseteq M$ is the supremum over all $\ell > 0$ such that there exists a minimizing extremal $\gamma:[0,\ell] \to M$ whose endpoints lie in $A$.

\begin{theorem}[Bonnet-Myers \text{\cite[Theorem 1.31]{CE}}, Sturm \text{\cite[Corollary 2.6]{Sturm2}}, cf. Agrachev \& Gamkrelidze \text{\cite[Section 4]{AG}}]\label{Bonnetthm}
    Suppose that the pair $(\mu,L)$ satisfies $\CD_{}({K},N)$ for some ${K} > 0$ and $1 < N < \infty$. Then $\diam(M) \le \pi\cdot\sqrt{(N-1)/{K}}$. If, in addition, the Euler-Lagrange flow is complete, then $M$ is compact.
\end{theorem}
\begin{proof}
    Assume by contradiction that there exists a minimizing extremal $\gamma$ which is defined on an interval of length $ \ell > \pi\cdot\sqrt{(N-1)/{K}}$. By Lemma \ref{uniquelyminimizinglemma}, we may assume that $\gamma$ is the uniquely minimizing, for otherwise we can replace $\ell$ by $\ell-\eps$ for some small $\eps$. Let $f = f_1 - f_0$, where $f_0$ (resp. $f_1$) is a nonnegative  function defined on a small neighborhood of $\gamma(0)$ (resp. $\gamma(\ell)$) whose total integral is $1$, and apply Theorem \ref{needlethm}. By the mass balance condition, there must exist a $\CD({K},N)$-needle $\mu_\alpha$ which is supported on an interval of length $ > \pi\cdot\sqrt{(N-1)/{K}}$, but this is not possible by Lemma \ref{1DBonnetlemma}. If the Euler-Lagrange flow is complete, then $M$ is compact by Lemma \ref{diamlemma}.
\end{proof}

\subsection{The Bishop-Gromov inequality}

The Bishop-Gromov inequality (see e.g. \cite[Chapter 7]{Pet}) compares the volume growth of a Riemannian manifold with Ricci curvature bounded from below to the volume growth of a model space of constant curvature. The generalization of the Bishop-Gromov inequality to metric measure spaces was proved in Sturm \cite{Sturm2} using the Brunn-Minkowski inequality. 

\medskip
For $x_0 \in M$ and $r>0$ denote by $B_r(x_0)$ the \emph{forward ball of radius $r$ centered at $x_0$}:
$$B_r(x_0) : = \{\gamma(t) \, \mid \, \gamma \text{ is a minimizing extremal with $\gamma(0) = x_0$,} \qquad 0 \le t < r\}$$

\begin{theorem}[The Bishop-Gromov inequality]\label{BGthm}
    Let ${K} \in \RR$ and $1 < N < \infty$. Suppose that the pair $(\mu,L)$ satisfies $\CD_{}({K},N)$. Let $x_0 \in M$ and define
    $$V(r) : = \mu(B_r(x_0)), \qquad r > 0.$$
    Then
    $$\frac{V(r)}{V(R)} \ge \frac{V_{K,N}(r)}{V_{K,N}(R)},$$
    where
    $$V_{K,N}(r) : = \left(\int_0^r\sin\left(t\sqrt{K/(N-1)}\right)dt\right)^{N-1}.$$
\end{theorem}

\begin{proof}
    Let $x_0 \in M$ and let 
    $$u(x) : = \cc(x_0,x), \qquad x \in M.$$
    Then $u$ is dominated, and we may apply Theorem \ref{needlethm2}; let $\mu = \int_\sA\mu_\alpha\,d\nu(\alpha)$ be the resulting disintegration of measure (in fact, in this case the disintegration of measure is just a Lagrangian analogue of polar normal coordinates). The transport rays of $u$ are precisely minimizing extremals starting at $x_0$. Hence for every $\alpha \in \sN$,
    $$\gamma_\alpha^{-1}(B_r(x_0)) = (a_\alpha,\min\{a_\alpha + r,b_\alpha\}),$$
    whence 
    $$\mu_\alpha(B_r(x_0)) = m_\alpha((a_\alpha,a_\alpha+r)).$$
    For $\nu$-almost every $\alpha \in \sN$ the measure $\mu_\alpha$ is a $\CD({K},N)$-needle, so the measure $m_\alpha$ satisfies the $\CD({K},N)$ condition with respect to the Euclidean metric on the interval $(a_\alpha,b_\alpha)$. Hence the measure $m_\alpha$ satisfies the corresponding Bishop-Gromov inequality, i.e.
    \begin{align*}
        \frac{m_\alpha((a_\alpha,a_\alpha+r))}{m_\alpha((a_\alpha,a_\alpha+R))} \ge \frac{V_{K,N}(r)}{V_{K,N}(R)}
    \end{align*}
    for every $0<r<R$, see \cite[Theorem 2.3]{Sturm2}. It follows that, for every $0 < r < R$,
    \begin{align*}
        V(r) & = \int_{\sA}\mu_\alpha(B_r(x_0))d\nu(\alpha)\\
        & = \int_{\sA}m_\alpha((a_\alpha,a_\alpha+r))d\nu(\alpha)\\
        & \ge \int_{\sA}\frac{V_{K,N}(r)}{V_{K,N}(R)}\cdot m_\alpha((a_\alpha,a_\alpha+R))d\nu(\alpha)\\
        & = \frac{V_{K,N}(r)}{V_{K,N}(R)}\cdot\int_{\sA}\mu_\alpha(B_R(x_0))d\nu(\alpha)\\
        & = \frac{V_{K,N}(r)}{V_{K,N}(R)}\cdot V(R),
    \end{align*}
    and the theorem is proved.   
\end{proof}

\subsection{The isoperimetric inequality}

Suppose that $\mu$ is a probability measure:
$$\mu(M) = 1.$$
For $A \subseteq M$ and $\eps > 0$, let $[A]_\eps$ denote the \emph{$\eps$-neighborhood} of the set $A$:
\begin{align*}
    [A]_\eps  : &= \{\gamma(t) \, \mid \, \gamma \text{ is a minimizing extremal with $\gamma(0) \in A$,} \qquad 0 \le t < \eps\}\\
    & = \bigcup_{a \in A}B_\eps(a),
\end{align*}

and let $\mu^+(A)$ denote the \emph{upper minkowski content} of $A$:

$$\mu^+(A) : = \liminf_{\eps\searrow 0}\frac{\mu([A]_\eps) - \mu(A)}{\eps}.$$

The \emph{isoperimetric profile} of the pair $(L,\mu)$ is defined by

$$I_{L,\mu}(s) : = \inf\left\{\mu^+(A) \, \vert \, A \subseteq M, \,\,\mu(A) = s\right\}, \qquad s \in [0,1].$$

For every ${K} \in \RR$, $N\in (-\infty,\infty]\setminus\{1\}$ and $D > 0$, we define a \emph{one-sided model isoperimetric profile} $I_{{K},N,D}^+:[0,1] \to [0,\infty]$ by
$$I_{{K},N,D}^+(s) : = \inf\left\{m^+(A) \quad \bigg\vert \quad m \in \mathcal{N}({K},N,D),\quad \begin{array}{c}
    A \subseteq \RR, \,\,m(A) = s, \\  m\vert_{(a,\infty)}(A) \le s \quad\forall a \in \RR\end{array}\right\},$$
where $\mathcal{N}({K},N,D)$ is the set of all $\CD({K},N)$-needles on $\RR$ with support contained in $[0,D]$ and total mass $1$ (i.e., absolutely-continuous probability measure on $[0,D]$ with a $C^2$ density satisfying $\CD(K,N)$ with respect to the Euclidean metric), and
$$m^+(A) : = \liminf_{\eps\searrow 0}\frac{m(A + [0,\eps)) - m(A)}{\eps}.$$
Note that the neighborhood of $A$ in the definition of $m^+(A)$ is one-sided.  The condition $m\vert_{(a,\infty)}(A) \le m(A)$ for all $a \in \RR$ in the definition of $I^+_{K,N,D}$ is therefore needed in order to exclude trivial minimizers of the form $A = [a_0,\infty)$ for some $a_0 \in (0,D)$. For a detailed description of the standard (i.e. two-sided) model isoperimetric profiles, see \cite{Mil}.

\medskip
The proof of the following theorem follows the proof of \cite[Proposition 5.4]{Kl}.

\begin{theorem}\label{isoperimthm}
    Suppose that the pair $(\mu,L)$ satisfies $\CD_{}({K},N)$ for some ${K} \ge 0$ and some $N \in (-\infty,\infty]\setminus[1,n)$, and that $\diam(M) = D < \infty$. Then
    $$I_{L,\mu} \ge I_{{K},N,D}^+.$$
\end{theorem}
\begin{proof}
    Let $A \subseteq M$, write $s = \mu(A)$, define a function $f : M \to \RR$ by
    $$f : = s - \chi_A,$$
    and let $\mu = \int_\sA\mu_\alpha\,d\nu(\alpha)$ be the disintegration of measure provided by Theorem \ref{needlethm}. Then for $\nu$-almost every $\alpha \in \sA$,
    \begin{equation}\label{isopermbeq1}\mu_\alpha(A) = s\cdot\mu_\alpha(M),\end{equation}
    and for $\nu$-almost every $\alpha\in\sN$,
    \begin{equation}\label{isopermbeq2}\mu_\alpha^+(A) \le s\cdot\mu_{\alpha,t}(M).\end{equation}
    for every upper end $\mu_{\alpha,t}$ of the needle $\mu_\alpha$.

    \medskip
    By the disintegration formula \eqref{disinteq}, for every $\eps > 0$,
    \begin{align*}
       \mu([A]_\eps) - \mu(A) & = \int_{\sA}[\mu_\alpha([A]_\eps) - \mu_\alpha(A)]d\nu(\alpha)\\
        & = \int_{\sA}\left[m_\alpha(\gamma_\alpha^{-1}([A]_\eps)) - m_\alpha(\gamma_\alpha^{-1}(A))\right]d\nu(\alpha)\\
        & \ge \int_{\sA}\left[m_\alpha(\gamma_\alpha^{-1}(A) + [0,\eps)) - m_\alpha(\gamma_\alpha^{-1}(A))\right]d\nu(\alpha)\\
        & = \int_{\sA}m_\alpha(\RR)\cdot\left(\frac{m_\alpha(\gamma_\alpha^{-1}(A) + [0,\eps))}{m_\alpha(\RR)} - \frac{m_\alpha(\gamma_\alpha^{-1}(A))}{m_\alpha(\RR)}\right)d\nu(\alpha),
    \end{align*}
    where the inequality holds true since, by the definition of $[A]_\eps$ and the fact that each $\gamma_\alpha$ is a minimizing extremal,
    $$\gamma_\alpha^{-1}([A]_\eps) \supseteq \gamma_\alpha^{-1}(A) + [0,\eps) \qquad \text{for $\nu$-a.e. $\alpha \in \sA$}.$$
    If $\alpha$ is a Dirac measure then the integrand vanishes. If $\alpha\in\sN$, then by \eqref{isopermbeq1} and \eqref{isopermbeq2}, the set $\gamma_\alpha^{-1}(A)$ satisfies the conditions in the definition of $I_{{K},N,D}^+$ with respect to the probability measure $m_\alpha/m_\alpha(\RR)$ on $\RR$, which is a $\CD({K},N)$-needle with total measure $1$. Thus by Fatou's Lemma and by \eqref{isopermbeq1}:
    \begin{align*}
        \liminf_{\eps\searrow 0}\frac{\mu([A]_\eps) - \mu(A)}{\eps} & \ge \int_{\sA}m_\alpha(\RR)\cdot I_{{K},N,D}^+\left(\frac{m_\alpha(\gamma_\alpha^{-1}(A))}{m_\alpha(\RR)}\right)d\nu(\alpha) \\
        & = \int_{\sA}\mu_\alpha(M)\cdot I_{{K},N,D}^+\left(\frac{\mu_\alpha(A)}{\mu_\alpha(M)}\right)d\nu(\alpha)\\
        & = I_{{K},N,D}^+(s)\cdot\int_{\sA}\mu_\alpha(M)\,d\nu(\alpha)\\
        & = I_{{K},N,D}^+(s)\cdot\mu(M)\\
        & = I_{{K},N,D}^+(s).
    \end{align*}
    Taking infimum over $A$ gives the desired result.
\end{proof}

\subsection{Functional inequalities}

In this section we prove analogues of the Poincaré and log-Sobolev inequalities under the curvature-dimension condition. Recall that $S_xM = SM\cap T_xM = \{v \in T_xM \, \mid \, E(v) = 0\}$ for all $x \in M$. Define
$$E^*:T^*M \to [0,\infty), \qquad E^*(p) : = \max_{S_xM}|p|, \qquad p\in T^*_xM, \, x \in M.$$

\begin{theorem}[Poincar\'e inequality, cf. \cite{PW,LY,ZY,Kl}]\label{pointhm}
    Suppose that the pair $(\mu,L)$ satisfies $\CD(0,\infty)$ and that $\diam(M) = D < \infty$. Then for every $f \in L^1(\mu)\cap L^2(\mu) \cap C^1$ such that $\int fd\mu = 0$,
    \begin{equation}\label{peq}\int_Mf^2d\mu \le \frac{D^2}{\pi^2}\int_ME^*(df)^2\,d\mu.\end{equation}
\end{theorem}

\begin{proof}
    By a truncation argument we may assume that $f$ is compactly-supported. Applying Theorem \ref{needlethm} to the function $f$, we obtain a disintegration of $\mu$ such that for $\nu$-almost every $\alpha \in \sA$,
    \begin{equation*}
        \int_Mfd\mu_\alpha = 0.
    \end{equation*}
    For $\nu$-a.e. $\alpha \in \sN$, the measure $\mu_\alpha$ is a $\CD(0,\infty)$, i.e. the density of $m_\alpha$ with respect to the Lebesgue measure on $I_\alpha$ is log-concave (see Example \ref{needleexample}). The one-dimensional Poincar\'e inequality in \cite{PW} (see also \cite[Section 4.5.2]{BGL}) then implies that for $\nu$-almost every $\alpha \in \sA$, writing $f_\alpha : = f\circ\gamma_\alpha$,
    \begin{align*}
        \int_Mf^2d\mu_\alpha& = \int_{a_\alpha}^{b_\alpha}f_\alpha^2\,dm_\alpha \\
        & \le \frac{(b_\alpha - a_\alpha)^2}{\pi^2}\int_{a_\alpha}^{b_\alpha}\dot f_\alpha^2\,dm_\alpha\\
        & \le \frac{D^2}{\pi^2}\int_{a_\alpha}^{b_\alpha}\dot f_\alpha^2\,dm_\alpha,
    \end{align*}
    where the last inequality holds true by the definition of diameter, since $\gamma_\alpha$ is a minimizing extremal. Since $\dot f_\alpha= df(\dot\gamma_\alpha)$ and  $\dot\gamma_\alpha \in SM$, we conclude that
    $$\int_Mf^2d\mu_\alpha  \le \frac{D^2}{\pi^2}\int_ME^*(df)^2d\mu_\alpha \qquad \text{ for $\nu$-a.e. $\alpha \in \sA$}$$
    (the inequality holds trivially for $\alpha \in \sD$, since the mass balance condition \eqref{MBeq2} implies that $\int_Mf^2d\mu_\alpha = f(x_\alpha)^2 = (\int_Mfd\mu_\alpha)^2 = 0$). Integrating with respect to $\alpha$ and using the disintegration formula \eqref{disinteq} we arrive at \eqref{peq}.
\end{proof}

\begin{theorem}[Log-Sobolev inequality, cf. \text{\cite[Proposition 5.7.1]{BGL}, \cite{Oh17}}]\label{lsthm}
    Suppose that $\mu$ is a probability measure and that the pair $(\mu,L)$ satisfies $\CD({K},\infty)$ for some ${K} > 0$.  Let $f \in (L^2\log L^2)(\mu)\cap L^2(\mu)\cap C^1$ satisfy $\int f^2d\mu = 1$. Then
    $$\int_Mf^2\log f^2 \, d\mu \le \frac{2}{{K}}\int_ME^*(df)^2\,d\mu.$$
\end{theorem}
\begin{proof}
    Again, we may assume that $f$ is compactly-supported. Apply Theorem \ref{needlethm} to the function
    $$f^2 - 1$$
    to obtain a disintegration of $\mu$ such that for $\nu$-almost every $\alpha \in \sA$ we have
    $$\int_Mf^2d\mu_\alpha = 1,$$
    and such that $\nu$-almost every $\alpha \in \sN$ is a $\CD({K},\infty)$-needle. It then follows from \eqref{disinteq} and the 1-dimensional log-Sobolev inequality (see e.g. \cite[Section 5.7]{BGL}) that
    \begin{align*}
        \int_{a_\alpha}^{b_\alpha}f(\gamma_\alpha(t))^2\log f(\gamma_\alpha(t))^2\,dm_\alpha(t) & \le \frac{2}{{K}}\int_{a_\alpha}^{b_\alpha}((f\circ\gamma_\alpha)'(t))^2\,dm_\alpha(t)\\
        & \le \frac{2}{K}\int_{a_\alpha}^{b_\alpha}(E^*(df)(\gamma_\alpha(t)))^2\,dm_\alpha(t),
    \end{align*}
    using in the second passage the fact that $E(\dot\gamma_\alpha)\equiv 0$. Thus, for $\nu$-almost every $\alpha \in \sA$,
    $$\int_Mf^2\log f^2 \, d\mu_\alpha \le \frac{2}{{K}}\int_ME^*(df)^2\,d\mu_\alpha,$$
    where again the inequality holds trivially for $\alpha \in \sD$, since the mass balance condition \eqref{MBeq2} implies that $f(x_\alpha)^2 = 1$ so the left hand side vanishes. We complete the proof by integrating over $\alpha$ and using the disintegration formula \eqref{disinteq}.
\end{proof}

\section{Examples}\label{examplesec}
\subsection{Classical Lagrangians}\label{magsec}

In this section we consider the case where the Lagrangian takes the form:
$$L(v) = \frac{|v|_g^2}{2} + U(x) - \eta(v), \qquad v \in T_xM, \, x \in M,$$
where $g$ is a Riemannian metric, $U$ is a smooth, positive function on $M$, and $\eta$ is a one-form. Note that our sign convention is the opposite of that used in classical mechanics. We leave the verification of some of the facts below as an exercise to the reader.
The Hamiltonian associated to $L$ is 
$$H(p) = \frac{|p + \eta|^2_g}{2} - U(x), \qquad p \in T^*_xM, \, x \in M,$$
The Legendre transform is
$$\cL p = p^\sharp + \eta^\sharp, \qquad p \in T^*M,$$
where $\sharp:T^*M \to TM$ is the musical isomorphism, and the energy is
\begin{equation}\label{magEeq}E(v) = \frac{|v|_g^2}{2} - U(x), \qquad v \in T_xM, \, x \in M.\end{equation}
The operators $\nabla$ and $\bL$ are given by
\begin{equation*}
    \nabla u = \nabla^gu + \eta^\sharp \qquad \text{ and } \qquad \bL u = \Delta_gu - \delta\eta,
\end{equation*}
where $\nabla^g$ and $\Delta_g$ are the Riemannian gradient and Laplacian, respectively, and $\delta$ is the codifferential. The Euler-Lagrange equation is
\begin{equation}\label{magELeq}\nabla^g_{\dot\gamma}\dot\gamma = \rY\dot\gamma + \nabla^gU,\end{equation}
where $\nabla^g$ is the Levi-Civita connection and $\rY = (d\eta)^\sharp$ is the $(1,1)$-tensor defined by the relation
$$\Omega : = d\eta = \left<\rY\,\cdot \, , \, \cdot \, \right>_g.$$

\medskip
We now derive an explicit formula for the weighted Ricci curvature of the Lagrangian $L$. First, let us identify some of the objects introduced in Sections \ref{spraysec} and \ref{Lagrangianconnectionsec}. Let $\Gamma_{ij}^k$ denote the Christoffel symbols of the metric $g$. Then, by \eqref{magELeq}, the coefficients of the Euler-Lagrange semispray are 
\begin{equation}\label{Sigmamageq}\Sigma^k = -\Gamma_{ij}^kv^iv^j + {\rY^{k}}_{i}v^i + (\nabla^gU)^k= \Sigma_g^k + {\rY^{k}}_{i}v^i + (\nabla^gU)^k, \qquad 1 \le k \le n ,\end{equation}
where $\Sigma_g$ is the geodesic spray of $g$, and the connection coefficients are
$$\Gamma_i^j = \Gamma_{ik}^jv^k - \frac12{\rY^j}_i, \qquad 1 \le i,j \le n.$$
The vector field $\Lambda$ is given by
\begin{equation}\label{maglabmdaeq}\Lambda^j = \frac12{\rY^j}_iv^i + (\nabla^gU)^j, \qquad 1 \le j \le n\end{equation}
and the vector fields $E_i$ are given by
\begin{equation}\label{Eimageq}E_i = \partial_{x^i} - \left(\Gamma_{ik}^jv^k - \frac12{\rY^j}_i\right)\partial_{v^j} = E_i^g + \frac12{\rY^j}_i\partial_{v^i},\end{equation}
where $E_i^g$ are the corresponding vector fields for the Lagrangian $g$. Finally, observe that $L_{v^iv^j} = g_{ij}$, i.e. the coefficients $g_{ij}$ of the metric $g$ coincide with the coefficients $g_{ij} = L_{v^iv^j}$ defined in Section \ref{Lagrangianconnectionsec}.

\begin{lemma}\label{mRiclemma}
    The Ricci curvature of the Lagrangian $L$ is given by
    $$\Ric = \Ric_g - \div_g\rY + \frac14|\rY|_g^2 - \Delta_gU,$$
    where the one-form $\div_g\rY = \tr(\nabla^g\rY)$ is viewed as a function on $TM$, and the functions $|\rY|_g,\Delta_g U$ on $M$ are viewed as functions on $TM$ which are constant on each fiber.
\end{lemma}
\begin{proof}
    In the following computation, a comma followed by an index $i$ denotes differentiation with respect to $x^i$. Objects with $g$ superscripts or subscripts are taken with respect to the Riemannian Lagrangian $(g+1)/2$ instead of $L$. We denote $U^j : = (\nabla^gU)^j$. By \eqref{Sigmamageq} and \eqref{Eimageq},
    \begin{align*}
        [E_i,\Sigma] = & [E_i^g,\Sigma_g] + \left[\frac12{\rY^j}_i\partial_{v^j},\Sigma_g\right] + \left[E_i^g,({\rY^j}_kv^k + U^j)\partial_{v^j}\right] + \left[\frac12{\rY^j}_i\partial_{v^j},({\rY^j}_kv^k + U^j)\partial_{v^j}\right] \\
        = & [E_i^g,\Sigma_g] +\frac12{\rY^j}_i\partial_{x^j} - {\rY^k}_i\Gamma_{k\ell}^jv^\ell\partial_{v^j}- \frac12v^k{\rY^j}_{i,k}\partial_{v^j} \\
        & + ({\rY^j}_{k,i}v^k+{U^j}_{,i})\partial_{v^j}- \Gamma_{ik}^\ell v^k{\rY^j}_\ell\partial_{v^j} + ({\rY^j}_kv^k + U^j)\Gamma_{ij}^\ell\partial_{v^\ell} + \frac12{\rY^k}_i{\rY^j}_k\partial_{v^j}\\
        = & [E_i^g,\Sigma_g] + \frac12{\rY^j}_i\partial_{x^j} \\
        & + \Big(-{\rY^k}_i\Gamma_{k\ell}^jv^\ell - \frac12 v^k{\rY^j}_{i,k} + v^k{\rY^j}_{k,i} - v^k\Gamma_{ik}^\ell{\rY^j}_\ell + v^k{\rY^\ell}_k\Gamma_{i\ell}^j + \frac12{\rY^k}_i{\rY^j}_k +{U^j}_{,i} + U^k\Gamma_{ik}^j\Big)\partial_{v^j}.
    \end{align*}
    Recall that $\eps^i = dv^i + \Gamma_k^idx^k= \eps_g^i -\frac12 {\rY^i}_kdx^k$. Thus by \eqref{Riccicoordeq},
    \begin{align*}
        -\Ric = \eps^i([E_i,\Sigma]) 
        = & \eps^i_g([E_i,\Sigma]) - \frac12{\rY^i}_kdx^k([E_i,\Sigma])
        \\ = & \eps^i_g([E_i,\Sigma]) - \frac12{\rY^i}_kdx^k\left([E_i^g,\Sigma^g] + \frac12{\rY^j}_i\partial_{x^j}\right)
        \\= & \eps^i_g([E_i^g,\Sigma_g]) + \frac12{\rY^k}_i\Gamma_{k\ell}^iv^\ell + \frac12{\rY^j}_i\Gamma_{jk}^iv^k - \frac14{\rY^i}_k{\rY^k}_i\\
        & - {\rY^k}_i\Gamma_{k\ell}^iv^\ell - \frac12 v^k{\rY^i}_{i,k} + v^k{\rY^i}_{k,i} - v^k\Gamma_{ik}^\ell{\rY^i}_\ell + v^k{\rY^\ell}_k\Gamma_{i\ell}^i + \frac12{\rY^k}_i{\rY^i}_k\\
        & +{U^i}_{,i} + U^k\Gamma_{ik}^i\\
        = &  -\Ric_g + v^k\left({\rY^i}_{k,i} - \Gamma_{ik}^\ell{\rY^i}_\ell + \Gamma_{i\ell}^i{\rY^\ell}_k\right) +  \frac14{\rY^k}_i{\rY^i}_k +{U^i}_{,i} + U^k\Gamma_{ik}^i\\
        = & -\Ric_g + (\div\rY)_kv^k - \frac14|\rY|_g^2 + \Delta_gU,
    \end{align*}
    where the identity ${\rY^k}_i{\rY^i}_k = -|\rY|_g^2$ holds since $\rY$ is skew-symmetric.
\end{proof}
Combining Lemma \ref{mRiclemma}, identity \eqref{maglabmdaeq} and Definition \ref{weightedriccidef} with $\psi \equiv 0$, we get:
\begin{corollary}\label{weightedmagcor}
    The weighted Ricci curvature of the Lagrangian $L$ with respect to the Riemannian volume measure is given by
    $$\Ric_{\Vol_g,N} = \Ric_g - \div_g\rY + \frac14|\rY|_g^2 - \Delta_gU + 2\Lambda_\perp^2 + \left(1-\frac{1}{N-n}\right)\Lambda_\parallel^2$$
    where
    $$\Lambda_\parallel(v) = \frac{\left<v,\nabla^gU\right>}{|v|^2}, \qquad \text{ and } \qquad \Lambda_\perp^2(v) = \frac{\left|\frac12\rY v + \nabla^gU - |v|^{-2}\left<v,\nabla^gU\right>v\right|^2}{|v|^2}, \qquad v \in TM\setminus\bz.$$
\end{corollary}
\begin{corollary}\label{Kahlercor}
    Suppose that $M$ is a K{\"a}hler manifold of complex dimension $d = n/2$, and that
    $$U \equiv \frac12 \qquad \text{ and } \qquad \rY = c\cdot \mathbf{J},$$
    where $\mathbf{J}$ is the complex structure and $c\in\RR$ is a constant. Then
    $$\Ric_{\Vol_g,n} = \Ric_g + c^2\cdot\frac{d+1}{2}.$$
\end{corollary}
\begin{proof} The complex structure  $\mathbf{J}$ is a parallel isometry and $\Lambda_\parallel \equiv 0$. Thus the result follows from Corollary \ref{weightedmagcor}.\end{proof}

\medskip
The following lemma is easy to verify; see e.g. \cite[Lemma 2.3]{BP}.
\begin{lemma}
    The cost associated with $L$ is given by
    $$\cc(x,y) = \inf_\gamma\left(\int_\gamma\sqrt{2U}\cdot |d\gamma| - \int_\gamma\eta\right),$$
    where the infimum is taken over piecewise-$C^1$ curves joining $x$ to $y$. A curve attaining the infimum satisfies $|\dot\gamma| = \sqrt{2U}$. In particular, if $U \equiv 1/2$ then
    $$\cc(x,y) = \inf_\gamma\left(\Len_g[\gamma] - \int_\gamma\eta\right),$$
    where $\Len_g[\gamma]$ denotes length with respect to the metric $g$, and the infimum is attained by a unit-speed curve.
\end{lemma}

\subsubsection*{Horocycles in complex hyperbolic space}
An noteworthy special case is the Lagrangian on the hyperbolic plane whose Euler-Lagrange flow on the unit tangent bundle is the horocycle flow. Let $\HH$ denote the hyperbolic plane, and let $L$ be the Lagrangian 
\begin{equation}\label{horoLagrangianeq}L = \frac{g+1}{2} - \eta,\end{equation}
where $g$ is the hyperbolic metric and $\eta$ is any primitive of the hyperbolic area form; for instance, in the upper half plane model one can take 
$$L = \frac{dx^2 + dy^2}{2y^2} + \frac12 - \frac{dx}{y}.$$
A Brunn-Minkowski inequality for this Lagrangian was obtained in \cite{AK}. We will now give a proof of Theorem \ref{horoBMthm}, which generalizes this result to \emph{complex hyperbolic space} of arbitrary dimension.

\medskip
Let $\mathbb{C}\HH^d$ denote the unique simply connected K{\"a}hler  manifold of complex dimension $d$ and constant holomorphic sectional curvature $-1$. When $d=1$ then $\mathbb{C}\HH^1 = \HH$ is the hyperbolic plane. Facts about the complex hyperbolic geometry can be found in \cite{Gol}. We define an \emph{oriented horocycle} to be a solution to the differential equation
\begin{equation}\label{horoeq}
	\nabla_{\dot\gamma}\dot\gamma = \mathbf{J}\dot\gamma, \qquad |\dot\gamma|\equiv 1,
\end{equation}
where $\mathbf{J}$ is the complex structure on $\mathbb{C}\HH^d$. Equivalently, an oriented horocycle is a unit-speed curve of geodesic curvature $1$ contained in a totally geodesic copy of the hyperbolic plane (a \emph{complex geodesic}) and oriented according to the induced complex structure.

\medskip
Let $\eta$ be any one-form satisfying $d\eta = \Omega$, where $\Omega$ is the K{\"a}hler form of $\mathbb{C}\HH^d$ (an explicit choice is written below). Since $\Omega^\sharp = \mathbf{J}$, equations \eqref{magEeq} and \eqref{magELeq} with $U\equiv 1/2$ imply that zero-energy extremals are solutions to \eqref{horoeq}. Thus every minimizing extremal is an oriented horocycle. The Riemannian Ricci curvature of $\mathbb{C}\HH^d$ is $-\frac{d+1}{2}\cdot g$ (see e.g. \cite[Theorem 7.5]{KN}). Hence by Corollary \ref{Kahlercor}, the weighted Ricci curvature of the Lagrangian $L$ with respect to the hyperbolic volume measure vanishes identically on the unit tangent bundle:
	$$\Ric_{\mu,n} \equiv 0 \qquad \text{ on } S\mathbb{C}\HH^d,$$ 
where $n=2d$ is the real dimension of $\HH^d$.

\medskip
Let us verify conditions (I)-(III) from Section \ref{assumpsec}. Clearly $L$ is Tonelli and $0$ is a regular value of the energy. Every pair of points in $\mathbb{C}\HH^d$ lies on a unique common complex geodesic, which contains a unique oriented horocycle joining them. Thus, in order to verify condition (II), it suffices to show that every oriented horocycle is uniquely minimizing. Note that this notion is independent of the choice of $\eta$. It is convenient to work in the Siegel model, which is the domain
$$\left\{z \in \bC^d \, \mid \, z_d + \bar z_d - r^2 > 0 \right\},$$
where $r : =\sum_{j=1}^{d-1}|z_j|^2$, endowed with the metric
\begin{align*}
    g = \frac{4}{(z_d + \bar z_d - r^2)^2}&\left(dz_d - \sum_{j=1}^{d-1}\bar z_jdz_j\right)\left(d\bar z_d - \sum_{j=1}^{d-1}z_jd\bar z_j\right) \\&+ \frac{4}{z_d + \bar z_d - r^2}\sum_{j=1}^{d-1}dz_jd\bar z_j.
\end{align*}
The one form 
$$\eta = i\cdot\frac{dz_d - d\bar z_d + \sum_{j=1}^{d-1}z_jd\bar z_j - \bar z_jdz_j}{z_d + \bar z_d - r^2}.$$
is a primitive of the K{\"a}hler form $\Omega = -4i\partial\bar\partial\log(z_d + \bar z_d - r^2)$. With these expressions for $g$ and $\eta$, it is easy to verify that $|\eta|_g = 1$, and that this norm is achieved at the unit vector field
$$V := \frac{z_d + \bar z_d - r^2}{2i}\left(\partial_{z_d} - \partial_{\bar z_d}\right).$$
Thus, for every piecewise-$C^1$ curve $\gamma$ on $\mathbb{C}\HH^d$, 
\begin{equation*}\Len[\gamma] - \int_\gamma\eta \ge 0,\end{equation*}
with equality if and only if $\gamma$ is an integral curve of $V$. In particular, closed curves have positive action, so $L$ is supercritical. The integral curves of $V$ lying on the complex geodesic $r = 0$ are oriented horocycles since, when restricted to this complex geodesic, the vector field $V$ becomes $x\partial_y$ in real coordinates on the right half-plane. It follows by symmetry that every oriented horocycle is uniquely minimizing. Condition (III) holds because the length of an oriented horocycle joining two points is a smooth function of the distance between them. Theorem \ref{horoBMthm} now follows from Theorem \ref{BMthm0}.

\subsubsection*{Contact magnetic geodesics on odd-dimensional spheres}

Let 
$$S^{2d+1} : = \left\{z \in \mathbb{C}^{d+1}\, \mid \, |z| = 1\right\}$$
denote the unit sphere in $\mathbb{C}^{n+1}$ with the round metric
$$g(v,w) = \Re\left<v,w\right>_{\mathbb{C}^{d+1}}, \qquad v,w \in T_zS^{2d+1}, \quad z \in S^{2d+1}.$$
It carries a natural contact structure, namely, the one-form
$$\eta\vert_z := \Re\left<iz,\cdot\right>_{\mathbb{C}^{d+1}} = \frac{i}{2}\sum_{j=1}^{n+1}z^jd\bar z^j - \bar z^jdz^j, \qquad z \in S^{2d+1}.$$
The \emph{Reeb vector field} is given by
$${R}\vert_z := \eta^\sharp\vert_z = i\sum_{j=1}^{n+1}z^j\partial_{z^j}-\bar z^j\partial_{\bar z^j}, \qquad z \in S^{2d+1}.$$
The two-form $\Omega = d\eta$ is given by
$$\Omega = i\sum_{j=1}^{d+1}dz^j\wedge d\bar z^j,$$
which is twice the standard symplectic form on $\mathbb{C}^{d+1}$, and so the $(1,1)$-tensor $\rY = \Omega^\sharp$ is
$$\rY = 2i\pi_{\mathrm{ker}\eta},$$
where $\pi_{\mathrm{ker}\eta}$ is the orthogonal projection onto the kernel of $\eta$. Equivalently,
$$\rY v = 2i\cdot v - 2i\cdot\eta(v)\cdot R, \qquad v \in T_zS^{2d+1}, \quad z \in S^{2d+1}.$$

\medskip
The extremals of the Lagrangian 
\begin{equation}\label{contactLagrangianeq}L_s = \frac{g + 1}{2} - s\cdot \eta\end{equation}
for $s>0$ were studied in detail in a recent paper \cite{ABM}, see also \cite{MN} (note that the one-form used in \cite{ABM} is $\alpha = \eta/2$, and subsequently, in the formulae below $s$ becomes $s/2$). For each $s > 0$, they include all integral curves of the Reeb vector field ${R}$, which are the fibers of the \emph{Hopf fibration} 
$$\pi_{\mathrm{Hopf}}:S^{2d+1}\to \mathbb{C}\mathbf{P}^d \qquad (z_1,\dots,z_d) \mapsto [z_1:\dots:z_d].$$
In addition, each geodesic of the sub-Riemannian manifold $(S^{2d+1},g,\mathrm{ker}\eta)$ is an extremal of $L_s$ for \emph{some} value of $s$.

\begin{definition}[Contact magnetic geodesics]\normalfont\label{contactmagdef}
    Minimizing extremals of the Lagrangian $L_s$ will be called \emph{contact magnetic geodesics of strength $s$}.
\end{definition}

\begin{remark*}\normalfont
    Due to different normalization, strength $s$ here corresponds to strength $s/2$ in \cite{ABM}.
\end{remark*}

\begin{remark*}\normalfont
    The above construction can be carried out on more general \emph{Sasakian} manifolds, see \cite{CFG,DIMN,MN}.
\end{remark*}



\medskip
From the expression for $\rY$ one readily verifies that
$$|\rY|_g^2 = 8d$$
and that
$$|\rY v|^2 = 4|v|^2 - 4\eta(v)^2, \quad v \in TS^{2n+1}.$$
Moreover, the formula $(\nabla_V(i\pi_{\ker\eta}))W = g(V,W)\cdot{R} - \eta(W)V$ for every pair of vector fields $V,W$ (this is the statement that $S^{2d+1}$ is Sasakian, see e.g. \cite{Bl}) shows that
$$\div\rY = -(4d)\eta.$$
The Ricci curvature of $S^{2n+1}$ is 
$$\Ric_g = 2d\cdot g.$$
Plugging the above formulae into Corollary \ref{weightedmagcor}, we get:
\begin{proposition}
    Let $0 \le s < 1$. The weighted Ricci curvature of the Lagrangian $L_s$ with respect to the Riemannian volume measure is given by
    $$\Ric_{\Vol_g,n}(v) = 2d\cdot|v|^2 + 4sd\cdot\eta(v) + 2s^2\left(d + 1 - \frac{\eta(v)^2}{|v|^2}\right), \qquad v \in TM\setminus \bz,$$
    where
    $$n : = 2d+1.$$
    Hence
    $$\Ric_{\Vol_g,n} = 2d(s^2+2s\eta+1) + 2s^2(1-\eta^2) \qquad \text{ on $S(S^{2d+1}) = \{v \in TS^{2d+1} \, \mid \, |v|_g = 1\}$}.$$
    Since $|\eta(v)| \le 1$ for all $v \in S(S^{2d+1})$, with equality when $v = \pm{R}$, it follows that the pair $(\Vol_g,L_s)$ satisfies $\CD(K,n)$ if and only if
    $$2d(s-1)^2 \ge K.$$
    In particular, the pair $(\Vol_g,L_s)$ satisfies $\CD(0,n)$ for all $0 \le s < 1$.
\end{proposition}

The Lagrangian $L_s$ is supercritical if and only if $|s|<1$, see \cite[Section 2]{ABM} and Example \ref{Lagrangiansexample}. In this case, conditions (II) and (III) are satisfied by supercriticality and compactness of $S^{2d+1}$ (a systematic treatment of condition (II) for all values of $s$ is the main objective of \cite{ABM}). Thus Theorem \ref{contactthm} follows from Theorem \ref{BMthm0} with $N = n$ and $K = 0$.

\begin{remark*}\normalfont
    Let us comment on the case $s=1$. A pair of points $z_0,z_1 \in S^{2d+1}$ can be joined by an extremal of $L_1$ if and only if $\left<z_0,z_1\right>\ne 0$, see \cite{ABM}. With a bit of extra work, one can show that for every $z_0,z_1 \in S^{2d+1}$ there exists a sequence of extremals of $L_1$ joining $z_0$ to $z_1$ whose action tends to $0$. Since the Lagrangian $L_1$ is non-negative, and vanishes only on the Reeb vector field, this implies that unless $z_0$ and $z_1$ lie on a common Reeb trajectory, there exists no \emph{minimizing} extremal of $L_1$ joining $z_1$ to $z_0$. In other words, in the sense of Definition \ref{contactmagdef}, contact magnetic geodesics of strength $1$ do not exist, except between points lying on the same Reeb trajectory.
\end{remark*}

\subsubsection*{Mechanical Lagrangians}
Consider the case $\eta \equiv 0$, i.e.
\begin{equation}\label{elecLageq}L(v) = \frac{|v|_g^2}{2} + U(x), \qquad v \in T_xM, \, \, x\in M.\end{equation}
By Corollary \ref{weightedmagcor}, for every $N \in (-\infty,\infty]\setminus[1,n)$ and every $v \in TM\setminus\bz$, the weighted Ricci curvature of the Lagrangian $L$ with respect to the Riemannian measure at the vector $v$ is given by
\begin{align*}
    \Ric_{\Vol_g,N}(v) & = \Ric_g(v) - \Delta_gU + 2\cdot\frac{\left|\nabla^gU - \left<v,v\right>^{-1}\left<\nabla^gU,v\right>v\right|^2}{|v|^2} + \left(1 - \frac{1}{N-n}\right)\cdot\frac{\left<v,\nabla^gU\right>^2}{|v|^4}\\
    & = \Ric_g(v) - \Delta_gU + 2\cdot\frac{|\nabla^gU|^2}{|v|^2} - \left(1 + \frac{1}{N-n}\right)\cdot \frac{\left<v,\nabla^gU\right>^2}{|v|^4}.
\end{align*}
Recall that
$$SM = \left\{v \in TM \, \mid \, |v| = \sqrt{2U}\right\}.$$
Hence, if $v \in SM$ then
$$\Ric_{\Vol_g,N}(v) = \Ric_g(v) - \frac{\Delta_gU}{2U}\cdot |v|^2 + \frac{|\nabla^gU|^2}{2U^2}\cdot|v|^2 -\left(1+\frac{1}{N-n}\right)\cdot\frac{\left<v,\nabla^gU\right>^2}{4U^2}.$$
Since the right hand side is now fiberwise $2$-homogeneous, this proves:
\begin{proposition}
    For every $K \in \RR$ and every $N \in (-\infty,\infty]\setminus[1,n)$, the pair $(\Vol_g,L)$ satisfies $\CD(K,N)$ if and only if
    $$\Ric_g \ge \left(\frac{K +\Delta_gU}{2U} - \frac{|\nabla^gU|^2}{2U^2}\right)g + \left(1+\frac{1}{N-n}\right)\frac{(dU)^2}{4U^2}$$
    as quadratic forms. In particular, the pair $(\Vol_g,L)$ satisfies $\CD(0,\infty)$ if and only if 
    \begin{equation}\label{mechanicalCD0inftyeq}\left(\Delta_gV\right)\,g - \left(dV\right)^2 + \Ric_g \ge 0, \qquad \text{ where } \qquad e^{-2V} := U.\end{equation}
\end{proposition}
\begin{remark*}\normalfont
    Zero-energy extremals of the Lagrangian \eqref{elecLageq} coincide, up to reparametrization, with unit-speed geodesics of the Riemannian metric $(2U)^{-1}\cdot g$. Because of the difference in parametrization, the Ricci curvature of $L$ is not equal to the Ricci curvature of this conformal metric.
\end{remark*}
\begin{remark}\label{shrmk}\normalfont
    Suppose that $\Ric_g \ge 0$. Then, since $$(dV)^2 \le |dV|^2\cdot g = \frac{|dU|^2}{4U^2}\cdot g \qquad \text{ and } \qquad \Delta_gV = -\frac{\Delta_gU}{2U} + \frac{|dU|^2}{2U^2},$$ inequality \eqref{mechanicalCD0inftyeq} holds when $U$ is superharmonic, i.e. when $$\Delta_gU \le 0.$$
\end{remark}
\begin{example}[Classical many-body Lagrangian]\normalfont
    Let $d\ge 3$ and $k \ge 2$, let
    $$M : = \left\{x = (x_1,\dots,x_k) \in (\RR^d)^k \quad \mid \quad x_i \ne x_j \, \, \forall \, i\ne j\right\},$$
    and consider the following \emph{$k$-body Lagrangian} :
    $$L = \frac12\sum_{i=1}^km_i\sum_{\ell=1}^d(dx_i^\ell)^2 + \frac12 + \sum_{i=1}^k\sum_{j=i+1}^k\frac{Gm_im_j}{|x_i-x_j|},$$
    where $G$ and $m_i$ are positive constants representing the gravitational constant and the masses of the bodies, and $x_i = (x_i^1,\dots,x_i^d)$ is the position of the $i$-th body. Thus in this case
    $$U = \frac12 + \sum_{i=1}^k\sum_{j=i+1}^k\frac{Gm_im_j}{|x_i-x_j|},$$
    whence
    $$dU = \sum_{i=1}^k \sum_{\ell=1}^d \sum_{j\ne i} \frac{Gm_i m_j (x_j^\ell - x_i^\ell)}{|x_i - x_j|^3} \, dx_i^\ell$$
    and
    $$\Delta_gU = (3-d)\cdot G \cdot \sum_{i=1}^k\sum_{j=i+1}^k\frac{m_i + m_j}{|x_i-x_j|^3} \le 0,$$
    where $g$ is the (flat) metric
    $$g = \sum_{i=1}^km_i\sum_{\ell=1}^d(dx_i^\ell)^2.$$
    By Remark \ref{shrmk}, the pair $(\Vol,L)$ satisfies $\CD(0,\infty)$, where $\Vol$ is the Lebesgue measure on $(\RR^d)^k$ (which differs by a constant from the measure $\Vol_g$). Note that the energy level $SM$ is characterized by the relation
    $$\frac12\sum_{i=1}^km_i\sum_{\ell=1}^d(dx_i^\ell)^2 = \sum_{i=1}^k\sum_{j=i+1}^k\frac{Gm_im_j}{|x_i-x_j|} + \frac12.$$
    \begin{remark*}\normalfont
        The Lagrangian $L$ does not satisfy the global assumptions (II) and (III) from Section \ref{assumpsec} on the full manifold $M$.
    \end{remark*}
\end{example}

\subsection{Isotropic Lagrangians}

Let $(M,g)$ be a Riemannian manifold. Let 
$$l:M\times [0,\infty) \to (0,\infty)$$
 be a smooth function which is strictly convex and superlinear in the second variable with 
 $$\qquad l'(x,0) = 0 \qquad  x \in M.$$
 Here and below, for functions on $M\times[0,\infty)$, prime $(')$ will denote differentiation with respect to the second variable, while gradients $(\nabla^g)$ and differentials $(d)$ are taken with resepct to the first variable. Let $L$ be the Tonelli Lagrangian
    \begin{equation}\label{isoLeq}L(v) = l(x,|v|_g), \qquad v \in T_xM, \, x \in M.\end{equation}
The corresponding Hamiltonian is given by
    $$H(p) = h(x,|p|_g), \qquad p \in T_x^*M, \, x \in M,$$
where
    $$h(x,\cdot) = l(x,\cdot)^*$$
is the Legendre conjugate of $l(x,\cdot)$, i.e 
    \begin{equation}\label{hprimelprimeeq}h(0) = -l(0) \qquad \text{ and } \qquad h'(x,l'(x,t)) = t, \qquad x \in M, \,\, t \ge 0.\end{equation}
Set 
    $$\alpha(x,t) = \frac{l'(x,t)}{t} \qquad \text{ and } \qquad \beta(x,t) = \frac{h'(x,t)}{t}, \qquad x \in M, \,\, t \ge 0.$$
    Note that $\alpha,\beta > 0$. The Legendre transform is then
    $$\cL p = \beta(x,|p|_g)\cdot p^\sharp \qquad p \in T_x^*M, \, x \in M.$$
In particular, the gradient of a function $u$ is
    $$\nabla u\vert_x = \beta\left(x,\big|du\vert_x\big|_g\right)\cdot\nabla^gu\vert_x, \qquad x \in M.$$
The Euler-Lagrange equation  is
\begin{equation}\label{sphELeq}
    \nabla^g_{\dot\gamma}\left(\alpha(\gamma,|\dot\gamma|_g)\cdot\dot\gamma\right) = \nabla^gl.
\end{equation}
For instance, if $l(x,t) = t^2/2 + U(x)$ for some function $U:M \to \RR$, then we arrive at the mechanical Lagrangian from the previous section, and since $\alpha\equiv 1$ and $\nabla^gl = \nabla^gU$ we recover \eqref{magELeq}.

\medskip
Define functions $w,b : M \to (0,\infty)$ by
    $$h(x,w(x))  = 0 \qquad \text{ and } \qquad b(x) : = \beta(x,w(x)), \qquad x \in M.$$
Our assumptions on $l$ imply that $h(0) < 0$ and that $h$ is strictly increasing and tends to infinity, so $w$ is well-defined. The Hamilton-Jacobi equation is therefore
    \begin{equation}\label{sphHJ}
        |du|_g = w.
    \end{equation}
Let 
    $$\mu = e^{-\psi}\Vol_g$$
be a measure on $M$ with a smooth density and, as before, set $\bL := \div_\mu(\nabla \,\cdot \,).$ Let $u$ be a locally-defined solution to the Hamilton-Jacobi equation \eqref{sphHJ}. Then
    \begin{equation}\label{Sphgrad}
        \nabla u = b\cdot\nabla^gu
    \end{equation}
    and
    \begin{equation*}
    \begin{split}
        \bL u &= \div_\mu(b\cdot\nabla^gu)\\
        & = b\cdot\left(\Delta_gu - \left<\nabla^g\psi,\nabla^gu\right>\right)+\left<\nabla^g b,\nabla^gu\right>\\
        & = b\cdot \Delta_gu + \left<\nabla^g\vphi,\nabla u\right>,
    \end{split}
    \end{equation*}
where in the last passage we set $$\vphi = \log b - \psi$$ and used \eqref{Sphgrad}. Thus
    \begin{align*}
        (d\bL u)(\nabla u) = & \left<\nabla^gb,\nabla u\right>\cdot\Delta_gu + b^2\left<\nabla^gu,\nabla^g\Delta_gu\right> + \Hess_g\vphi(\nabla u) + \left<\nabla^g\vphi,\nabla^g_{\nabla u}\nabla u \right>.
    \end{align*}
By \eqref{sphHJ} and the Riemannian Bochner formula,
    $$\left<\nabla^gu,\nabla^g\Delta_{g}u\right> = \Delta_g(w^2/2) - |\nabla^2_gu|^2 - \Ric_{g}(\nabla^gu).$$
Using \eqref{Sphgrad} again we get
\begin{align}\label{dBlnablaueq}
\begin{split}
    (d\bL u)(\nabla u) = & \left<\nabla^g  b,\nabla u\right>\cdot\Delta_gu + b^2\Delta_g(w^2/2) - b^2|\nabla^2_gu|^2 \\
    & - \Ric_g(\nabla u) + \Hess_g\vphi(\nabla u) + \left<\nabla^g\vphi,\nabla^g_{\nabla u}\nabla u\right>.
\end{split}
\end{align}
By \eqref{sphHJ} and \eqref{Sphgrad}, 
\begin{equation}\label{nablaunormeq}
    |\nabla u|_g = b\cdot|\nabla^gu|_g = b\cdot|du|_g = b\cdot w.
\end{equation}
Since $u$ satisfies the Hamilton-Jacobi equation, the integral curves of $\nabla u$ solve \eqref{sphELeq}, so
\begin{equation}\label{nablagleq}\nabla^gl = \nabla^g_{\nabla u}(a\nabla u) = da(\nabla u)\cdot\nabla u + a\cdot\nabla^g_{\nabla u}\nabla u,\end{equation}
where
$$a(x) : = \alpha(x,b(x)w(x)), \qquad x \in M.$$
It follows from \eqref{hprimelprimeeq} and the definitions of $a$ and $b$ that
\begin{equation}\label{abeq}ab\equiv 1.\end{equation}
Substituting \eqref{nablagleq} into \eqref{dBlnablaueq} and using \eqref{abeq} we get
\begin{align}\label{dBlnablaueq2}
\begin{split}
    (d\bL u)(\nabla u) = & \left<\nabla^g b,\nabla u\right>\cdot\Delta_gu + b^2\Delta_g(w^2/2) - b^2|\nabla^2_gu|^2 - \Ric_g(\nabla u)\\
    &  + \Hess_g\vphi(\nabla u) + b\left<\nabla^g\vphi,\nabla^gl\right> + \left<\nabla^g\vphi,(d\log b)(\nabla u)\cdot \nabla u\right>.
\end{split}
\end{align}
By \eqref{Sphgrad}, \eqref{nablagleq} and \eqref{abeq},
\begin{align*}
    (\nabla^2_gu)(\nabla u) = \nabla^g_{\nabla u}\nabla^gu & = \nabla^g_{\nabla u}(b^{-1}\nabla u)= \nabla^g_{\nabla u}(a\nabla u) = \nabla^gl.
\end{align*}
As in Section \ref{weightedsec}, we separate the self-adjoint operator $\nabla^2_gu$ into its orthogonal and tangential components with respect to $\nabla u$ and the apply the Cauchy-Schwarz inequality:
\begin{align*}
    |\nabla^2_gu|^2 & \ge \frac{2|\nabla^gl|^2}{|\nabla u|^2} - \frac{\left<\nabla^gl,\nabla u\right>^2}{|\nabla u|^4} + \frac{(\Delta_gu - \left<\nabla^gl,\nabla u\right>/|\nabla u|^2)^2}{n-1}\\
    & = \frac{2|\nabla^gl|^2}{b^2w^2}  - \frac{\left<\nabla^gl,\nabla u\right>^2}{b^4w^4} + \frac{(\Delta_gu - \left<\nabla^gl,\nabla u\right>/b^2w^2)^2}{n-1},
\end{align*}
where in the second passage we used \eqref{nablaunormeq}. Equality holds if and only if the restriction of $\nabla^2_gu$ to the orthogonal complement of $\nabla  u$ is scalar. We use this last inequality to bound the first and third terms on the right hand side of \eqref{dBlnablaueq2}:
\begin{align*}
    & \left<\nabla^gb,\nabla u\right>\Delta_gu -b^2|\nabla^2_gu|^2 \\
    \le&   \left<\nabla^gb,\nabla u\right>\Delta_gu - \frac{b^2\left(\Delta_gu - \left<\nabla^gl,\nabla u\right>/b^2w^2\right)^2}{n-1} -\frac{2|\nabla^gl|^2}{w^2}  + \frac{\left<\nabla^gl,\nabla u\right>^2}{b^2w^4}\\
    = &  -\frac{b^2\Big(\Delta_gu - \left<\nabla^gl,\nabla u\right>/b^2w^2- (n-1)\left<\nabla^gb,\nabla u\right>/2b^2\Big)^2}{n-1}\\
    &  + \frac{n-1}{4b^2}\cdot\left<\nabla^gb,\nabla u\right>^2 + \frac{\left<\nabla^gb,\nabla u\right>\left<\nabla^gl,\nabla u\right>}{b^2w^2} -\frac{2|\nabla^gl|^2}{w^2}  +\frac{\left<\nabla^gl,\nabla u\right>^2}{b^2w^4}\\
    \le &  \,\frac{n-1}{4}\cdot\left<\nabla^g\log b,\nabla u\right>^2 + \frac{\left<\nabla^gb,\nabla u\right>\left<\nabla^gl,\nabla u\right>}{b^2w^2} -\frac{2|\nabla^gl|^2}{w^2}  +\frac{\left<\nabla^gl,\nabla u\right>^2}{b^2w^4}.\\
\end{align*}
Plugging this estimate into \eqref{dBlnablaueq2} and recalling that $\vphi = \log b - \psi$ and $|\nabla u| = bw$, we obtain
\begin{equation*}
    (d\bL u)(\nabla u) + A\cdot |\nabla u|^2 + Q(\nabla u) \le 0,
 \end{equation*}
where $A : M \to \RR$ is the function
$$A : = -w^{-2}\Delta_g(w^2/2) - b^{-2}w^{-2}\left<\nabla^gb,\nabla^gl\right> + b^{-1}w^{-2}\left<\nabla^g\psi,\nabla^gl\right> + 2b^{-2}w^{-4}|\nabla^gl|^2,$$
and $Q:TM \to \RR$ is the quadratic form
$$Q : = \Ric_g + \Hess_g(\psi-\log b) + d\log b \cdot d\psi - \frac{n+3}{4}(d\log b)^2 - \frac{db\cdot dl}{b^2w^2} - \frac{(dl)^2}{b^2w^4}.$$
Recall that the differentials of $b$ and $l$ are with respect to the manifold variable. Equality holds if and only if
$$\nabla^2_gu = \frac{(\nabla u)^\flat \otimes\nabla^g l + (\nabla^gl)^\flat \otimes \nabla u}{\left<\nabla u,\nabla u\right>} - \frac{\left<\nabla^gl,\nabla u \right>}{\left<\nabla u,\nabla u\right>}\cdot\frac{(\nabla u)^\flat\otimes\nabla u}{\left<\nabla u,\nabla u\right>} + c_0\cdot\pi_{(\nabla u)^\perp},$$
where $\pi_{(\nabla u)^\perp}$ denotes orthogonal projection onto the orthogonal complement of $\nabla u$, and
$$c_0 : = \frac{\left<\nabla^gb,\nabla u\right>}{2b^2}.$$
Thus we have proven:
\begin{proposition}\label{isoprop}
    The pair $(e^{-\psi}\Vol_g,L)$ satisfies $\CD(K,\infty)$, where $L$ is the Lagrangian given in \eqref{isoLeq}, if and only if
    $$ Q + A\cdot g\ge b^{-2}w^{-2}K\cdot g.$$
\end{proposition}

\subsubsection*{Homogeneous Lagrangians with variable exponent}
Consider the case where
$$l(x,t) = \frac{t^{q(x)} + q(x) - 2 + c}{q(x)}, \qquad x \in M, \,\, t \ge 0$$
for some smooth function $q:M \to (1,\infty)$ and some constant $c \ge 1$. Then
$$h(x,t) = \frac{t^{q^*(x)} - c}{q^*(x)},$$
where $q^*(x)$ is the H{\"o}lder conjugate of $q(x)$ for every $x \in M$. One easily verifies that
$$\alpha(x,t) = t^{q(x) - 2}, \quad \beta(x,t) = t^{q^*(x) - 2}, \quad w(x) = c^{\frac{1}{q^*(x)}} \quad \text{ and } \quad b(x) = c^{\frac{2}{q(x)}-1}$$
for all $x \in M$ and $t \ge 0$. The Hamilton-Jacobi equation is therefore
$$|du|_g^{q^*} = c.$$
The energy is
$$E(v) = \frac{|v|_g^{q(x)}-c}{q^*(x)}, \qquad v \in T_xM, \, x\in M.$$
Thus 
$$SM = \left\{v \in TM \, \mid \, |v|_g = c^{1/q}\right\}.$$

\medskip
A (tedious, but straightforward) computation shows that
$$A = (\log c)\Delta_gr + 2\left(\log^2c - 6(\log c)(1-1/c) + 4(1-1/c)^2\right)|dr|^2$$
and 
$$Q = \Ric_g - 2(\log c)\Hess r - \left((n+2)\log^2c + 4(1-1/c)^2\right)(dr)^2,$$
where
$$r : = 1/q.$$
Thus we get from Proposition \ref{isoprop}:
\begin{proposition}
    The pair $(\Vol_g,L)$, where $L$ is the Lagrangian
    $$L(v) = \frac{|v|_g^{q(x)} + q(x) -2 +c}{q(x)}, \qquad v \in T_xM\,\, x\in M,$$
    satisfies $\CD(K,\infty)$ if and only if
    \begin{align*}
        & \Ric_g - 2(\log c)\Hess r - \left((n+2)\log^2c + 4(1-1/c)^2\right)(dr)^2 \ge\\ & \qquad \qquad  \left(c^{-2r}K - (\log c)\Delta_gr - 2\left(\log^2c - 6(\log c)(1-1/c) + 4(1-1/c)^2\right)|dr|^2\right)\cdot g.
    \end{align*}
\end{proposition}

\begin{remark*}\normalfont
    Suppose that $M$ is compact and let $H$ be the Hamiltonian associated to $L$. The class of functions $u : M \to \RR$ for which $\int_MH(du) < \infty$ is known as the \emph{variable exponent Sobolev space} $W^{1,q(\cdot)}$, see \cite{DHHR}. 
\end{remark*}

\subsubsection*{Convex functions of Finsler metrics}

Let $(M,F)$ be a Finsler manifold and let $\Phi:[0,\infty)\to(0,\infty)$ be a smooth, strictly convex, superlinear function satisfying $\Phi'(0) = 0$. Let $L$ be the Lagrangian
$$L(v) = \Phi(F(v)), \qquad v \in TM.$$
The corresponding Hamiltonian is
$$H(p) = \Phi^*(F^*(p)), \qquad p \in T^*M,$$
where $\Phi^*$ is the Legendre conjugate of $\Phi$ and $F^*$ is the dual norm to $F$. From the perspective of curvature-dimension theory, such Lagrangians were considered in \cite{Oh14}, where some of the facts mentioned below can be found. The Legendre transform is
\begin{equation}\label{Finslerlegendreeq}\cL p = \frac{(\Phi')^{-1}(F^*(p))}{F^*(p)}\cdot\cL_F p,\end{equation}
where $\cL_F$ is the Legendre transform associated to the Finsler metric $F$. The energy is given by

$$E(v) = \Phi^*(\Phi'(F(v))), \qquad v \in TM.$$

Thus minimizing extremals have constant speed
$$s : = (\Phi')^{-1}(\Phi^*)^{-1}(0).$$ 
Their action will therefore be proportional to their length; hence, minimizing extremals are Finsler geodesics with speed $s$.

\medskip
Let us parse the $\CD(K,\infty)$ condition for the Lagrangian $L$. A function $u$ is a solution to the Hamilton-Jacobi equation precisely when 
$$F^*(du) = w : = (\Phi^*)^{-1}(0).$$
If we set 
$$b :=  \frac{(\Phi')^{-1}(w)}{w} = \frac{s}{w}> 0,$$
then by \eqref{Finslerlegendreeq},
$$\nabla u = b\cdot \nabla^Fu \qquad \text{ and } \qquad \bL u = b\cdot \bL^Fu,$$
where the superscripts $F$ indicate that the corresponding operators with respect to the Finsler metric $F$. Thus
\begin{align*}
    (d\bL u)(\nabla u) = b^2\cdot (d\bL^Fu)(\nabla^Fu).
\end{align*}
Since $SM = \{v \in TM \, \mid \, F(v) = s = bw\}$, it follows that 
\begin{equation}\label{FinslerRiceq}\Ric_{\mu,\infty} \ge K \quad \text{ on $SM$} \qquad \iff \qquad \Ric^F_{\mu,\infty} \ge s^{-2}\cdot K \cdot F^2.\end{equation}

\medskip 
The following proposition, sometimes called a \emph{Phi-entropy inequality}, goes back to Bakry and {\'E}mery; see \cite{Ch}, \cite{BG} or \cite[Proposition 7.6.1]{BGL}. 

\begin{proposition}[Phi-entropy inequalities on the real line]\label{Phiprop1D}
    Let $m$ be a probability measure on an interval $I\subseteq \RR$ with a smooth density $e^{-\psi}$ satisfying $\ddot\psi \ge K > 0$. Then for every positive function $f \in L^1(m)\cap C^1$ such that $\int|\Phi(f)|dm < \infty$,
    $$\int_I\Phi(f)dm - \Phi\left(\int_Ifdm\right) \le \frac{1}{2K}\int_I\Phi''(f)\cdot |f'|^2dm.$$
\end{proposition}

In fact, the above result holds for any Markov semigroup satisfying $\CD(K,N)$ in the sense of Bakry and {\'E}mery, and in particular for weighted Riemannian manifolds. Since the Finslerian Laplacian is nonlinear, the following theorem does not fit directly into the Bakry-{\'E}mery framework. For proofs of functional inequalities on Finsler manifolds using semigroup techniques, see \cite[Chapter 16]{OhBook}. A Finsler manifold is called \emph{reversible} if $F(v) = F(-v)$ for every $v \in TM$.

\begin{theorem}[Phi-entropy inequalities for reversible Finsler manifolds]\label{Phientthm}
    Let $(M,F)$ be a geodesically convex, reversible Finsler manifold and let $\mu$ be a measure on $M$ with a smooth density. Suppose that $\Ric_{\mu,\infty}^F \ge KF^2 > 0$, where $\Ric_{\mu,\infty}^F$ is the weighted Finslerian Ricci curvature. Then for every positive function $f \in L^1(\mu)\cap C^1$ such that $\int|\Phi(f)|d\mu < \infty$,
    $$\int_M\Phi(f)d\mu - \Phi\left(\int_Mfd\mu\right) \le \frac{1}{2K}\int_M\Phi''(f)\cdot F^*(df)^2d\mu.$$
\end{theorem}

\begin{remark*}\normalfont
    Since minimizing extremals of $L$ are geodesics with speed $s$, Theorem \ref{Phientthm} can be proved using the Finslerian needle decomposition theorem \cite{Oh18} for the Finsler metric $s^{-1}\cdot F$. Still, since we did not find it elsewhere, we include it here.
\end{remark*}

\begin{proof}
    Apply Theorem \ref{needlethm} to the function 
    $$f - \int_Mfd\mu,$$
    and let $\mu = \int\mu_\alpha\,d\nu(\alpha)$ be the resulting disintegration of measure.
    Then for every $\alpha \in \sA$,
    \begin{equation}\label{Phimbeq}\int_Mfd\mu_\alpha = \mu_\alpha(M)\cdot\int_Mfd\mu.\end{equation}
    Each curve $\gamma_\alpha$ is a minimizing geodesic of the Finsler metric $F$, with speed $s = (\Phi')^{-1}(\Phi^*)^{-1}(0)$. Thus
    \begin{equation}\label{fgammaalphaprimeeq}|(f\circ\gamma_\alpha)'| = |df(\dot\gamma_\alpha)| \le F(\dot\gamma_\alpha)\cdot F^*(df) = s \cdot F^*(df).\end{equation}
    This is where we use the fact that $F$ is reversible, because otherwise the middle inequality may not hold. By \eqref{FinslerRiceq} and the assumption $\Ric_{\mu,\infty}^F\ge K$, the pair $(\mu,L)$ satisfies $\CD(s^2K,\infty)$ with respect to the Lagrangian $L$. Hence, for $\nu$-almost every $\alpha \in \sN$, the measure $\mu_\alpha$ is a $\CD(s^2K,\infty)$-needle, i.e. the density $e^{-\psi_\alpha}$ of the measure $m_\alpha$ satisfies $\ddot\psi_\alpha \ge s^2K$. The disintegration formula implies that $\int|\Phi(f)|d\mu_\alpha < \infty$ for $\nu$-almost every $\alpha \in \sA$. Applying Proposition \ref{Phiprop1D} and using \eqref{fgammaalphaprimeeq}, we see that for $\nu$-almost every $\alpha \in \sN$,
    $$\mu_\alpha(M)^{-1}\int_M\Phi(f)d\mu_\alpha - \Phi\left(\mu_\alpha(M)^{-1}\int_Mfd\mu_\alpha\right) \le \frac{\mu_\alpha(M)^{-1}}{2K}\int_M\Phi''(f)\cdot F^*(df)^2d\mu_\alpha.$$
    By \eqref{Phimbeq}, this implies
    $$\frac{\int_Mfd\mu}{\int_Mfd\mu_\alpha}\cdot\int_M\Phi(f)d\mu_\alpha - \Phi\left(\int_Mfd\mu\right) \le \frac{1}{2K}\cdot \frac{\int_Mfd\mu\cdot\int_M\Phi''(f)\cdot F^*(df)^2d\mu_\alpha}{\int_Mfd\mu_\alpha}.$$
    The inequality also holds true trivially for $\alpha \in \sD$, i.e. when $\mu_\alpha$ is a Dirac measure, since the left hand side vanishes by \eqref{Phimbeq}. Multiplying by $\int_Mfd\mu_\alpha$ and integrating over $\alpha$, we obtain the desired inequality.
\end{proof}

As a corollary, we recover (in the reversible case) the Finslerian \emph{Beckner's inequality} \cite[Corollary 16.11]{OhBook}:
\begin{corollary}[Beckner's inequality for reversible Finsler manifolds]
    Let $(M,F)$ be a geodesically convex, reversible Finsler manifold and let $\mu$ be a probability measure on $M$ with a smooth density. Suppose that $\Ric_{\mu,\infty}^F \ge KF^2 > 0$, where $\Ric_{\mu,\infty}^F$ is the weighted Finslerian Ricci curvature. Then for every positive function $f \in L^2(\mu)\cap C^1$ and every $q \in [1,2)$,
    $$\int_Mf^2d\mu - \left(\int_Mf^qd\mu\right)^{2/q} \le \frac{2-q}{K}\int_MF^*(df)^2d\mu.$$
\end{corollary}
\begin{proof}
    Set 
    $$\Phi(t) = \frac{t^{2/q}+2/q-1}{2/q}$$
    and apply Theorem \ref{Phientthm} to $f^q$. 
\end{proof}
\begin{remark*}\normalfont
    Taking $q\to 2$ in Beckner's inequality gives the logarithmic Sobolev inequality.
\end{remark*}

\end{document}